\documentclass[oneside,english]{amsart}
\usepackage{amstext}
\usepackage{amsthm}
\usepackage{amssymb}
\usepackage{color}
\usepackage{refstyle}
\usepackage{graphicx}
\usepackage{stmaryrd}

\usepackage[bookmarks=true,bookmarksnumbered=false,bookmarksopen=false,
 breaklinks=false,,backref=false,colorlinks=true]{hyperref}
\hypersetup{linkcolor=blue, urlcolor=blue, citecolor=blue}

\makeatletter

\AtBeginDocument{\providecommand\thmref[1]{\ref{thm:#1}}}
\AtBeginDocument{\providecommand\propref[1]{\ref{prop:#1}}}
\AtBeginDocument{\providecommand\figref[1]{\ref{fig:#1}}}
\AtBeginDocument{\providecommand\lemref[1]{\ref{lem:#1}}}
\AtBeginDocument{}
\AtBeginDocument{\providecommand\corref[1]{\ref{cor:#1}}}
\RS@ifundefined{subsecref}
  {\newref{subsec}{name = \RSsectxt}}
  {}
\RS@ifundefined{thmref}
  {\def\RSthmtxt{Theorem~}\newref{thm}{name = \RSthmtxt}}
  {}
\RS@ifundefined{lemref}
  {\def\RSlemtxt{Lemma~}\newref{lem}{name = \RSlemtxt}}
  {}

\numberwithin{equation}{section}
\numberwithin{figure}{section}
\theoremstyle{plain}
\newtheorem{thm}{\protect\theoremname}[section]
  \theoremstyle{remark}
  \newtheorem{rem}[thm]{\protect\remarkname}

  \theoremstyle{plain}
  \newtheorem{prop}[thm]{\protect\propositionname}
  \theoremstyle{definition}
  \newtheorem{defn}[thm]{\protect\definitionname}
  \theoremstyle{plain}
  \newtheorem{lem}[thm]{\protect\lemmaname}
  \theoremstyle{plain}
  \newtheorem{cor}[thm]{\protect\corollaryname}

\theoremstyle{definition}
  \newtheorem*{defn*}{\protect\definitionname}

\AtBeginDocument{
  
}

\makeatother

  \providecommand{\corollaryname}{Corollary}
  \providecommand{\definitionname}{Definition}
  \providecommand{\lemmaname}{Lemma}
  \providecommand{\propositionname}{Proposition}
  \providecommand{\remarkname}{Remark}
\providecommand{\theoremname}{Theorem}

\usepackage{xcolor}

\definecolor{revcolor}{RGB}{0,55,110}
\newcommand{\rev}[1]{#1}

\begin{document}

\title[Strichartz estimates without loss outside two convex]{Strichartz estimates without loss outside two strictly convex obstacles}

\author{David Lafontaine}
\begin{abstract}
We prove global Strichartz estimates without loss outside two strictly
convex obstacles, combining arguments from \cite{Ikawa2,IkawaMult}
with more recent ones inspired by \cite{MR2720226} and \cite{MR2672795}.
This is to be contrasted with the loss in the smoothing estimate for
the Schr\"{o}dinger equation with Dirichlet boundary conditions in the
presence of a trapped geodesic.
\end{abstract}

\thanks{CNRS and Institut de Math\'ematiques de Toulouse, UMR5219; Universit\'e de Toulouse, CNRS; UPS, F-31062 Toulouse Cedex 9, France; \tt david.lafontaine@math.univ-toulouse.fr}

\maketitle

\tableofcontents

\section{Introduction}

The \textit{Strichartz estimates} are a family of dispersive estimates
for the Schr\"{o}dinger flow on a Riemannian Manifold $M$:

\[
\Vert e^{it\Delta}u_{0}\Vert_{L^{p}(0,T)L^{q}(M)}\leq C_{T}\Vert u_{0}\Vert_{L^{2}},
\]
where the exponents $(p,q)$ have to follow the admissibility condition
given by the scaling of the equation:
\[
p,q\geq2,\ \frac{2}{p}+\frac{n}{q}=\frac{n}{2},\ (p,q,n)\neq(2,\infty,2),
\]
and $n$ is the dimension of $M$. If $M$ is a manifold with boundaries,
we understand $\Delta$ to be the Dirichlet Laplacian.

Locally in time, these estimates describe a regularizing effect, reflected
by a gain of integrability. Globally, they describe a decay effect:
the $L_{x}^{q}$ norm has to decay, at least in a $L_{t}^{p}$ sense.

We say that \textit{a loss} occurs in the Strichartz estimates, if
the estimates do not hold but we are able to show that

\[
\Vert e^{it\Delta}u_{0}\Vert_{L^{p}(0,T)L^{q}(M)}\leq C_{T}\Vert u_{0}\Vert_{H^{s}},
\]
with $s>0$.

The story of Strichartz estimates in the case of the flat Laplacian
on $M=\mathbb{R}^{n}$ begins with the work of \cite{Strichartz}
for $p=q$, for the wave equation, then generalized for any admissible
couple and the Schr\"{o}dinger equation by \cite{GV85}, and by \cite{KeelTao}
for the endpoint ($p=2)$ case. In the variable coeficients case,
the situation is more difficult. Strichartz estimates without loss
for the Schr\"{o}dinger equation were obtained in several geometrical
situations where all the geodesics go to infinity -- we say that the
manifold is \textit{nontrapping} -- : \cite{BoucletTzvetkov}, \cite{Bouclet},
\cite{HassellTaoWunsch}, \cite{StaffTata}, in the case of manifold
without boundary; and in the case of a manifold with boundary, by
\cite{MR2672795} who obtained the estimates outside one convex obstacle.

Another tool for analyzing the behavior of the Schr\"{o}dinger flow is
the \textit{local smoothing effect}
\[
\Vert\chi e^{it\Delta_{D}}u_{0}\Vert_{L^{2}(\mathbb{R},H^{1/2}(M))}\leq C\Vert u_{0}\Vert_{L^{2}}
\]
where $\chi$ is a smooth compactly supported function. The influence
of the geometry on the local smoothing effect is now fully understood:
\cite{MR2066943} showed that a necessary and sufficient condition
for this estimate to hold is the non-trapping condition. When trapped
geodesics exist, a loss in the smoothing effect has to occur.

But in the case of the Strichartz estimates, the situation is not
fully understood yet: \cite{MR2720226} showed Strichartz estimates
without loss for asymptotically euclidian manifolds without boundary
for which the trapped set is small enough and exhibit an hyperbolic
dynamics.

We go in the same direction and show global Strichartz estimates without
loss outside two strictly convex obstacles, providing a first example of Strichartz
estimates without loss in a trapped setting for the problem with boundaries:
\begin{thm}
\label{thm:main}Let $\Theta_{1}$ and $\Theta_{2}$ be two smooth,
compact, strictly convex subsets of $\mathbb{R}^{3}$, and $\Omega=\mathbb{\mathbb{R}}^{3}\backslash(\Theta_{1}\cup\Theta_{2})$.
If $(p,q)$ verifies
\[
p>2, \;q\geq2,\quad \frac{2}{p}+\frac{3}{q}=\frac{3}{2},
\]
then, the following estimate holds:
\begin{equation}
\Vert e^{it\Delta_{D}}u_{0}\Vert_{L^{p}(\mathbb{R},L^{q}(\Omega))}\leq C\Vert u_{0}\Vert_{L^{2}}.\label{eq:goal}
\end{equation}
\end{thm}
Strictly convex is here understood in the sense of \cite{MR2672795},
that is with a second fondamental form bounded below.

The behavior of the waves equation outside several convex obstacles
was first investigated by \cite{Ikawa2,IkawaMult}, who showed local
energy decay. Burq used in \cite{plaques} some of his results to
show a control property on an equation closely related to the Schr\"{o}dinger
equation outside several convex bodies.

For data of frequencies $\sim h^{-1}$, if we are able to prove Strichartz
estimates in times $h$
\[
\Vert e^{it\Delta_{D}}u_{0}\Vert_{L^{p}(0,h)L^{q}(\Omega)}\leq C\Vert u_{0}\Vert_{L^{2}},
\]
then the smoothing estimate provides the Strichartz estimates. But
in our framework, a logarithmic loss occurs in the estimate, due to
the presence of a trapped ray:
\[
\Vert\chi e^{it\Delta_{D}}u_{0}\Vert_{L^{2}(\mathbb{R},H^{1/2}(M))}\leq C|\ln h|^{\frac{1}{2}}\Vert u_{0}\Vert_{L^{2}}.
\]
However, \cite{MR2720226} remarked that this logarithmic loss can
be compensated if we show Strichartz estimates up to times $h|\log h|$.
Following their idea and arguments of \cite{MR2672795}, we first
show that we can restrict ourselves to prove the estimate on a neighborhood
of the periodic ray and in times $h\log h$, for datas of frequencies
$\sim h^{-1}$. Then, we reduce again the problem, to data microlocally near the trapped set and away from the boundary. Finally,
following ideas of \cite{Ikawa2,IkawaMult}, \cite{plaques}, we construct
an approximate solution for such data, and we show that this approximation
gives the desired estimate.
\begin{rem}
We restrict here ourselves to the three dimensional case because we
use results of \cite{Ikawa2,IkawaMult} and \cite{plaques} who are
stated in dimension three. The other parts of the proof hold in any
dimension without modification. Because the above mentioned results
should be extended without difficulty to the arbitrary dimensional
case, our proof should work in any dimension.
\end{rem}

\subsection{Notations}

We will use the following notations and conventions:
\begin{itemize}
\item $\Theta_{1},\Theta_{2}$ will be two smooth, compact, strictly
convex subsets of $\mathbb{R}^{3}$, called the obstacles.
\item $\Omega := \mathbb R^3 \backslash (\Theta_{1} \cup \Theta_{2})$.
\item {$X^{i}(x,p)$} the $i$'th point of intersection of the reflected
ray starting at $x$ in the direction $p$ with the obstacles.
\item {$\Xi^{i}(x,p)$} the direction of the reflected ray starting at
$x$ in the direction $p$ after $i$ reflections.
\item {$(X_{t}(x,p),\Xi_{t}(x,p))$} the point and the direction attained
by following the reflected ray starting at $x$ in the direction $p$
during a time $t$ at speed $|p|$.
\item \rev{$\mathcal{R} \subset \mathbb R^3$, which we call the periodic ray: it is the unique segment of $\Omega$ normal both to $\Theta_1$ and $\Theta_2$.}
\item \rev{$\mathcal{K} \subset {}^b T^* \Omega$ the trapped set, which we define as $\mathcal K := \pi_b (\mathcal R \times \mathbb R^* e)$ where $e$ is a vector parallel to $\mathcal R$.}
\item \rev{The compressed co-tangent sets ${}^b T^* \Omega$ and ${}^b T^* (\mathbb R \times \Omega)$, the associated notions of microlocalisation, the canonical projections (both denoted $\pi_b$), the associated bicharacteristic flows denoted respectively $\phi_t$ and $\gamma_t$, are all defined in Appendix \ref{app:sca}. }
\item \rev{Finally, the functional calculus used follows the definition of \cite{Dav}, similarly to \cite{MR2672795}.}
\end{itemize}

\begin{figure}

\includegraphics[scale=0.25]{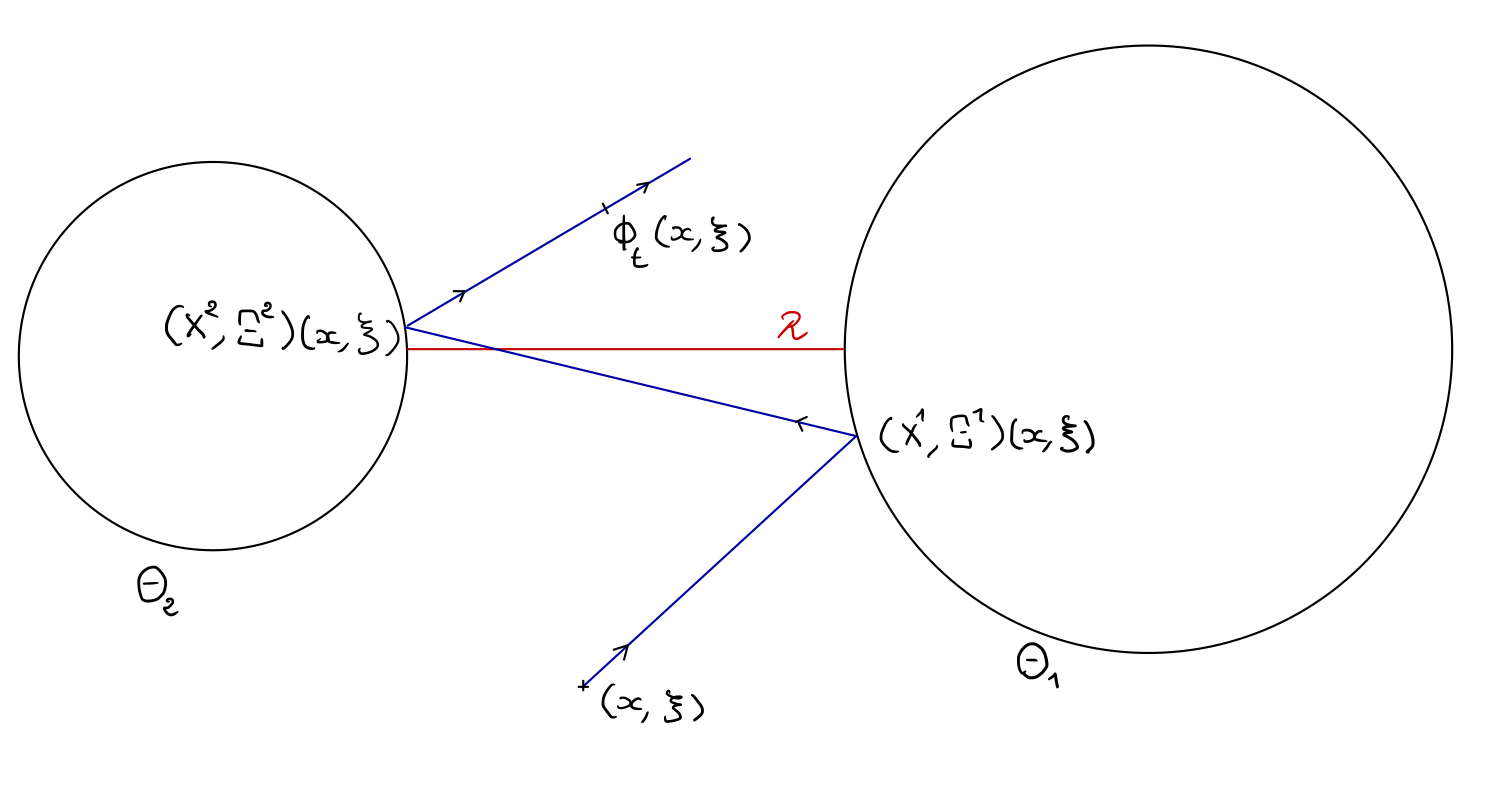}\caption{Notations}

\end{figure}

\rev{
\section{A semiclassical resolvent estimate without loss away from trapping}

The purpose of this Section is to show a semiclassical resolvent estimate without loss away from trapping, provided a polynomial resolvent estimate is at hand: this is given as Theorem \ref{th:res} below. As it is of independent interest, and as the proof doesn't involve the precise geometry of the domain, we do so for general obstacles.

Hence, in this Section only, we let $\mathcal O \subset \mathbb R^d$ be a compact subset of $\mathbb R^d$ with a smooth boundary and no infinite order contact point, and such that $\Omega := \mathbb R^d \backslash \mathcal O$ has one connected component. 

We denote $\Gamma \subset {}^b T^* \Omega$ the trapped-set associated with the broken bicharacteristic flow in $\Omega$:
$$
\Gamma := \Big\{ \rho \in \Sigma, \quad \sup_{t \in \mathbb R} |\pi_x(\phi_{t}(\rho))| < \infty \Big\}, \quad \Sigma := \pi_b (\{ |\xi|^2 = 1 \}).
$$

\begin{thm} \label{th:res}
Assume that $-\Delta_{D,\Omega}$ verifies a polynomial high-frequency truncated outgoing resolvent estimate: there exists $M>0$ so that, for any $\chi \in C^\infty_c(\mathbb R^d)$, there exists $C>0$, $h_0 > 0$ so that for any $h \in (0, h_0]$
\begin{equation} \label{eq:polynomial}
\Vert\chi(-h^2\Delta_{D,\Omega}- 1)^{-1}\chi\Vert_{L^{2}(\Omega)\rightarrow L^{2}(\Omega)}\leq C h^{-M}.
\end{equation}

Then, $-\Delta_{D,\Omega}$ verifies a non-trapping high-frequency outgoing truncated resolvent estimate away from trapping: for any $a, b \in C^\infty_c({}^b T^* \Omega)$ so that $\big(\operatorname{supp} a \cup \operatorname{supp} b\big) \cap \Gamma = \emptyset$, 
there exists $C>0$, $h_0 > 0$ so that for any $h \in (0, h_0]$
\begin{equation} \label{eq:pol}
\Vert \operatorname{Op}_h(b) (- h^2\Delta_{D,\Omega}-1)^{-1}\operatorname{Op}_h(a)\Vert_{L^{2}(\Omega)\rightarrow L^{2}(\Omega)}\leq \frac{C}{h},
\end{equation}
and, for any $\chi \in C^\infty_c(\mathbb R^d)$ so that $\operatorname{supp} \chi \cap \pi_x(\Gamma) = \emptyset$, 
there exists $C>0$, $h_0 > 0$ so that for any $h \in (0, h_0]$
\begin{equation} \label{eq:pol_chi}
\Vert \chi (- h^2\Delta_{D,\Omega}-1)^{-1}\chi\Vert_{L^{2}(\Omega)\rightarrow L^{2}(\Omega)}\leq \frac{C}{h}.
\end{equation}
\end{thm}
The analog of Theorem \ref{th:res} in the case without boundary was proven by \cite{DV}. Observe however that the main propagation result of \cite{DV} is more general, as it is local near components of the trapped set.

Our proof adapts the spatial
commutator argument of \cite[Section~3]{DV}, using
semiclassical defect measures in the spirit of \cite{B02} to treat the boundary.
\begin{proof}[Proof of Theorem \ref{th:res}]
We denote
$$
P_h := -h^2\Delta_{\Omega,D}-1,
$$
and $p(x,\xi):=|\xi|^2 - 1$.

We show (\ref{eq:pol}) and indicate at the end the supplementary ingredient for (\ref{eq:pol_chi}).
Let such $a, b \in C_c^\infty({}^b T^* \Omega)$. Without loss of generality, we can assume that $b$ verifies in addition
\begin{equation}\label{eq:ass_b}
b = 1 \; \operatorname{ near } \; T^* (B(0, R+1) \backslash B(0,R))\cap \{ |\xi|^2 = 1\},
\end{equation}
where $R>0$ is big enough so that $\mathcal O \Subset B(0, R)$. Indeed, if the result holds for such symbols, if $b \in C_c^\infty({}^b T^* \Omega)$ verifies a priori only $\operatorname{supp} b \cap \Gamma = \emptyset$, taking $\widetilde b = 1$ near $\operatorname{supp} b \cup (T^* (B(0, R+1) \backslash B(0,R))\cap \{ |\xi|^2 = 1\})$ with $\operatorname{supp}\widetilde b \cap \Gamma = \emptyset$, we have $\operatorname{Op}_h(b) = \operatorname{Op}_h(b) \operatorname{Op}_h(\widetilde b) + O(h^\infty)$ in $L^2 \to L^2$, and the estimate for $b$ follows from the estimate for $\widetilde b$ and the polynomial resolvent estimate.

Now, assume that the result fails. Then, there exist sequences $h_n \rightarrow 0$, $u_n$, $f_n$, so that $u_n$ is outgoing and 
\begin{equation} \label{eq:set_contr}
P_{h_n} u_n = \operatorname{Op}_{h_n}(a) f_n, \quad \Vert \operatorname{Op}_{h_n}(b) u_n\Vert _{L^2}\geq \frac{n}{h_n} \Vert f_n \Vert_{L^2}.
\end{equation}
By renormalizing, we can assume that
\begin{equation}\label{eq:res:renorm}
 \Vert \operatorname{Op}_{h_n}(b) u_n\Vert_{L^2} = 1,
\end{equation}
in such a way that
\begin{equation}\label{eq:res:dec_f}
\Vert f_n \Vert_{L^2} = o(h_n).
\end{equation}
We aim to show that $ \Vert \operatorname{Op}_{h_n}(b) u_n\Vert_{L^2} \rightarrow 0$, leading to a contradiction.

Let $R_0>0$ and $\mathcal I$ denote the incoming set:
$$
\mathcal I := \Big \{ (x,\xi) \in T^*\mathbb R^d, \hspace{0.3cm}|x| \geq R_0, \; x\cdot \xi \leq0 \Big\}.
$$
Using the properties of the free resolvent, one can obtain in a classical way from the outgoingness of $u$ (for example from \cite[Proposition 2.2]{B02} together with \cite[(3.6)]{B02}), fixing $R_0>0$ big enough, that
\begin{equation} \label{eq:outgo}
{}^b \operatorname{WF}_h u \cap \mathcal I = \emptyset.
\end{equation}

The following Lemma expresses the fact that the flow of any point away from $\Gamma$ meets a region where it is microlocally controlled.

\begin{lem} \label{lem:weak_esc}
Let $\rho \in \Sigma \backslash \Gamma$, and $S_1>S_2$ be so that $\Omega^c \subset B(0,S_2)$. Then, 
$$
\exists t_0 \in \mathbb R, \; \phi_{t_0}(\rho) \in \mathcal I \hspace{0.3cm} \text{ \emph{or} } \hspace{0.3cm}
\exists t_0 \in \mathbb R, \; \phi_{t_0}(\rho) \in S^* (B(0, S_1) \backslash B(0,S_2)).
$$ 
\end{lem}
\begin{proof}
If for any $t$, $\phi_{t}(\rho) \notin \mathcal I$, necessarily $\phi_{t}(\rho)$ is bounded in one time direction (and this can only be in a neighborhood of the obstacle), and if in addition for any $t$, $\phi_{t}(\rho) \notin S^* (B(0, S_1) \backslash B(0,S_2))$, then $\phi_{t}(\rho)$ is bounded in both time directions (because it would have to intersect $S^* (B(0, S_1) \backslash B(0,S_2))$ to go away from the neighborhood of the obstacle it is bounded in), which is not possible since $\rho \notin \Gamma$.
\end{proof}

As a first consequence of the above Lemma, we have

\begin{lem} \label{lem:bdd}
For any $q \in C^\infty_c({}^b T^* \Omega \backslash \Gamma)$, 
the sequence $(\operatorname{Op}_{h_n}(q)u_n)_{n \geq 1}$ is bounded in $L^2(\Omega)$. In particular, $(\langle \operatorname{Op}_{h_n}(q)u_n, u_n \rangle)_{n \geq 1}$ is bounded.
\end{lem}
\begin{proof}
We will say that $u_n$ is bounded near $\rho \in {}^b T^* \Omega$ if for any $q\in C^\infty_c( {}^b T^* \Omega)$ supported close enough to $\rho$, $\Vert \operatorname{Op}_{h_n} (q) u_n\Vert_{L^2} \lesssim 1$. We will show that $u_n$ is bounded near any point of ${}^b T^* \Omega \backslash \Gamma$.

Let $\rho \notin \Gamma$. First, observe that if $\rho \notin \Sigma$, $u_n$ is bounded near $\rho$ by ellipticity (Lemma \ref{lem:ell_b}). Let us therefore assume that $\rho \in \Sigma$.
By Lemma \ref{lem:weak_esc}, 
$$
\exists t_0 \in \mathbb R, \; \phi_{t_0}(\rho) \in \mathcal I \hspace{0.3cm} \text{ \emph{or} } \hspace{0.3cm}
\exists t_0 \in \mathbb R, \; \phi_{t_0}(\rho) \in S^* (B(0, R+1) \backslash B(0,R)).
$$ 

In both cases,  $u_n$ is bounded microlocally near $\phi_{t_0}(\rho)$: by (\ref{eq:outgo})  in the first case, and by (\ref{eq:res:renorm}) and (\ref{eq:ass_b}) in the second case. 
In addition, $h_n^{-1}f_n$ is bounded in $L^2$  by (\ref{eq:res:dec_f}), hence, by propagation of singularities (as $u_n$ is tempered thanks to the polynomial resolvent estimate (\ref{eq:polynomial})) given for example by Proposition \ref{prop:prop_sing_qt} (with $\sigma = 0$), it follows that $u_n$ is bounded microlocally near $\rho$, and the first claim follows.

Finally, let $\widetilde q \in C^\infty_c({}^b T^* \Omega \backslash \Gamma)$ be so that $\widetilde q = 1$ near $\operatorname{supp} q$. As $u_n$ is tempered,
$\langle \operatorname{Op}_{h_n}(q)u_n, u_n \rangle = 
\langle \operatorname{Op}_{h_n}(q)u_n, \operatorname{Op}_{h_n}(\widetilde q) u_n \rangle + O(h_n^\infty)$, and the second claim follows.
\end{proof}

From Lemma \ref{lem:bdd}, there exists a non-negative Radon measure $\mu$ on ${}^b T^* \Omega \backslash \Gamma$, so that, passing to a subsequence
$$
\forall q \in C^\infty_c({}^b T^* \Omega \backslash \Gamma), \hspace{0.3cm} 
\langle \operatorname{Op}_{h_n}(q)u_n, u_n \rangle \to \mu(q).
$$
By semiclassical ellipticity (see for example \cite[Section 4.2]{B02}),
$$
\operatorname{supp}\mu \subset \Sigma.
$$
Now, by \cite[Section 4.3]{B02}, applied locally away from $\Gamma$, we have
\begin{prop} \label{prop:res:inv}  
$\mu$ is invariant under the generalized bicharacteristic flow on $\Sigma \backslash \Gamma$.
\end{prop}

Let now $R_2>R_1>\max(R_0,R)$, and
$$
S := \{ R_1< |x| < R_2 \}.
$$
Further, let $\eta \in C^\infty_c(\mathbb R^d)$, $\eta(x) = \eta_0(|x|)$ be so that
$$
\eta_0(r) = 1 \text { for } r\leq R_1, \quad \eta_0(r) = 0  \text { for } r\geq R_2,
$$
and
\begin{equation}\label{eq:res1}
\eta_0'(r)<0 \text{ on } R_1<r<R_2.
\end{equation}
Now, observe that, by (\ref{eq:set_contr}) and the fact that $P_{h_n}$ is self-adjoint
$$
\langle \frac {i}{h_n}[P_{h_n}, \eta]u_n, u_n\rangle = \frac {2}{h_n} \operatorname{Im}\langle f_n, (\operatorname{Op}_{h_n}(a))^* \eta u_{n} \rangle,
$$
but from Lemma \ref{lem:bdd}, as $a$ is supported away from $\Gamma$,
$$
\Vert (\operatorname{Op}_{h_n}(a))^* \eta u_{n} \Vert_{L^2} \lesssim 1,
$$
therefore,
$$
\Big| \langle \frac {i}{h_n}[P_{h_n}, \eta]u_n, u_n\rangle\Big| \leq \frac {2}{h_n} \Vert f_{n}\Vert_{L^2 }\Vert (\operatorname{Op}_{h_n}(a))^* \eta u_{n} \Vert_{L^2} \to 0.
$$
We thus have, localizing near $\{ |\xi|^2 = 1 \}$ by semiclassical ellipticity (either by Lemma \ref{lem:ell_b} or directly with a parametrix in the interior given for example by \cite[Proposition E.32]{DZ}) to identify the limit with $\mu(H_p\eta)$,
\begin{equation}\label{eq:com0}
\mu(H_p\eta) = 0,
\end{equation}
which is well defined since $H_p \eta$ is supported in ${}^b T^* \Omega \backslash \Gamma$.
Now, from (\ref{eq:outgo})
\begin{equation}\label{eq:inc0}
\mu(\mathcal I) = 0.
\end{equation}
On the other hand
$$
H_p \eta = 2\xi \cdot \nabla \eta = 2\eta'_0(|x|)\frac{x\cdot\xi}{|x|},
$$
therefore, by (\ref{eq:res1}) and (\ref{eq:inc0}) 
\begin{equation} \label{eq:Hpneg}
H_p \eta < 0 \quad \mu\text{--almost everywhere on } T^* S.
\end{equation}
Combining (\ref{eq:com0}) with (\ref{eq:Hpneg}),
we get
\begin{equation}\label{eq:S0}
\mu(T^* S) = 0.
\end{equation}
As by Lemma \ref{lem:weak_esc} any generalized bicharacteristic in $\Sigma \backslash \Gamma$ meets either $T^*S$ or $\mathcal I$, we conclude by the invariance property of Proposition \ref{prop:res:inv} together with (\ref{eq:inc0}) and (\ref{eq:S0}) that $\mu({}^b T^* \Omega \backslash \Gamma) = 0$. In particular, $\mu(|b|^2) = 0$. This is a contradiction with (\ref{eq:res:renorm}), and (\ref{eq:pol}) follows.

To show (\ref{eq:pol_chi}), the proof is the same, localizing near $\pi_b (\{|\xi|^2 = 1 \})$ by semiclassical ellipticity (Lemma \ref{lem:ell_b}) when needed. 

\end{proof}
}
\section{Reduction to logarithmic times near the periodic ray}

The aim of this section is to show that Semiclassical Strichartz estimates on a logarithmic interval near
the periodic ray $\mathcal R \subset \mathbb R^3$ imply \thmref{main}:

\rev{\begin{prop} Let $\psi \in C^\infty_c(\mathbb R_+^*)$, $\psi \neq 0$.
Assume that 
\label{thm:main-1}\label{prop:semilog} there exists $\epsilon>0$ and a neighbourhood $\mathcal U$ of $\mathcal R$
such that for any $\chi\in C_{c}^{\infty}(\mathbb R^3)$ 
with $\operatorname{supp} \chi \subset \mathcal U$, there exists $C>0$ so that
\begin{equation}
\Vert\chi e^{it\Delta_{D}}\psi(-h^{2}\Delta_D)u_{0}\Vert_{L^{p}(0,\epsilon h|\log h|)L^{q}(\Omega)}\leq C\Vert u_{0}\Vert_{L^{2}}.\label{eq:butult}
\end{equation}
Then, \thmref{main} holds.
\end{prop}}
The rest of this section is devoted to the proof of Proposition \ref{thm:main-1}: hence, we will assume that (\ref{eq:butult})
holds with $\epsilon=1$ and prove \thmref{main}. The proof for other
values of $\epsilon>0$ is the same.

Let $\chi_{\text{obst}},\chi_{\text{ray}}\in C_{c}^{\infty}$ be such
that $\chi_{\text{obst}}=1$ in a neighborhood of $\Theta_{1}\cup\Theta_{2}\cup\mathcal{R}$,
and $\chi_{\text{ray}}\in C_{c}^{\infty}$ such that $\chi_{\text{ray}}=1$
in a neighborhood of $\mathcal{R}$ (sufficiently small that (\ref{eq:butult}) holds by assumption), and let $\widetilde{\psi}$ such
that $\widetilde{\psi}=1$ on the support of $\psi$. We decompose $\psi(-h^{2}\Delta_D)e^{it\Delta_{D}}u_{0}$
in a sum of three terms
\begin{multline}
\psi(-h^{2}\Delta_D)e^{it\Delta_{D}}u_{0}=\widetilde{\psi}(-h^{2}\Delta_D)(1-\chi_{\text{obst}})\psi(-h^{2}\Delta_D)e^{it\Delta_{D}}u_{0}\\
+\widetilde{\psi}(-h^{2}\Delta_D)\chi_{\text{obst}}(1-\chi_{\text{ray}})\psi(-h^{2}\Delta_D)e^{it\Delta_{D}}u_{0}\\
+\widetilde{\psi}(-h^{2}\Delta_D)\chi_{\text{obst}}\chi_{\text{ray}}\psi(-h^{2}\Delta_D)e^{it\Delta_{D}}u_{0}.\label{eq:dec-red}
\end{multline}

We can deal with the first two terms like in \cite{MR2672795} thanks to a smoothing effect \textit{without loss} away from the trapped ray, which is a direct consequence of Theorem \ref{th:res} (Proposition \ref{prop:smooth_wo} below). The third term can be handled with the
method of \cite{MR2720226}: the smoothing effect \textit{with logarithmic
loss} obtained in \cite{MR2066943} combined with Proposition \propref{semilog}
and the smoothing effect in the non trapping region is sufficient
to show Strichartz estimates without loss.

\subsection{The first two terms: away from the periodic ray}

In this section, we deal with the first two terms of (\ref{eq:dec-red}),
following \cite{MR2672795}. We will use

\begin{prop}[Local smoothing without loss in the non trapping region]
\label{prop:smooth_wo}For any $\chi_0\in C_{c}^{\infty}(\mathbb R^3)$ supported away from $\mathcal R$, we have
\begin{equation}
\Vert\chi_0 e^{it\Delta_{D}}u_{0}\Vert_{L^{2}(\mathbb{R},H^{1/2}(\Omega))}\lesssim\Vert u_{0}\Vert_{L^{2}}.\label{eq:lswlnt}
\end{equation}
\end{prop}
\begin{proof}
The smoothing estimate (\ref{eq:lswlnt}) follows from Theorem {\ref{th:res}} (with the polynomial bound assumption satisfied thanks to \cite{MR2066943}) in a classical way, see for example \cite{MR2068304}.  
\end{proof}

\subsubsection{The first term} \label{sss:11}

We set $w_{h}(x,t)=(1-\chi_{\text{obst}})\psi(-h^{2}\Delta_D)e^{it\Delta_{D}}u_{0}$,
which satisfies
\[
\begin{cases}
i\partial_{t}w_{h}+\Delta_{D}w_{h}&= -[\chi_{\text{obst}},\Delta]\psi(-h^{2}\Delta_D)e^{it\Delta_{D}}u_{0}\\
w_{h}(t=0)&=  (1-\chi_{\text{obst}})\psi(-h^{2}\Delta_D)u_{0}
\end{cases}
\]
\rev{As $\chi_{\rm obst}=1$ near $\partial\Omega$,
$w_{h}$ solve a problem in the full space, so the Duhamel formula
gives
\begin{equation} \label{eq:wh}
w_{h}(t,x) = e^{it\Delta}(1-\chi_{\text{obst}})\psi(-h^{2}\Delta_D)u_{0} - \int_0^t e^{i(t-s)\Delta}[\chi_{\text{obst}},\Delta]\psi(-h^{2}\Delta_D)e^{it\Delta_{D}}u_{0}(s) \, ds.
\end{equation}
Similarly to  \cite{MR2672795, StaffTata}, we will use the following result due to  Christ-Kiselev to bound the Duhamel term arising above:
\begin{lem}[\cite{CK}] \label{lem:CK}
Let $B_1, B_2$ two Banach spaces, $r<s$, and consider a bounded operator $T:L^{r}(\mathbb R, B_1) \mapsto L^{s}(\mathbb R, B_2)$ given by a locally integrable kernel $K(t,s)$ with values in bounded operators from $B_1$ to $B_2$. Then, the operator
$$
\widetilde T f(t) := \int_{s<t} K(t,s) f(s) \, ds
$$
is bounded in $L^{r}(\mathbb R, B_1) \mapsto L^{s}(\mathbb R, B_2)$ with
$$
\Vert \widetilde T\Vert_{L^{r}(\mathbb R, B_1) \mapsto L^{s}(\mathbb R, B_2)}
\leq C (1 - 2^{-(1/s-1/r)})^{-1} \Vert T\Vert_{L^{r}(\mathbb R, B_1) \mapsto L^{s}(\mathbb R, B_2)}.
$$ 
\end{lem}
From Lemma \ref{lem:CK}, Strichartz estimates for the free flow, and the smoothing effect in $\mathbb R^3$, the operator
$$
T : F \in L^{2}(\mathbb R, H_{\rm comp}^{-1/2}(\mathbb R^3)) \mapsto  \int_0^t e^{i(t-s)\Delta}F(s)\,ds  \in L^p(\mathbb R, L^q(\mathbb R^3))
$$
is bounded: indeed, by Lemma \ref{lem:CK} applied with $r=2$, $B_1 = H^{-\frac 12}_{\rm comp}$, $s=p$, $B_2 = L^q$ (here we use that $p>2$), it suffices to bound
$$
\widetilde T : F \in L^{2}(\mathbb R, H_{\rm comp}^{-1/2}(\mathbb R^3)) \mapsto  \int_0^\infty e^{i(t-s)\Delta}F(s)\,ds  \in L^p(\mathbb R, L^q(\mathbb R^3)).
$$
But $\widetilde T = U_0 U_0 ^*$ where $U_0(t) = e^{it\Delta}$ is the free flow, and $U_0^*$ is bounded $L^{2}(\mathbb R, H_{\rm comp}^{-1/2}(\mathbb R^3)) \mapsto L^2(\mathbb R^3)$ by the dual estimate to the smoothing effect in $\mathbb R^3$, and $U_0$ is bounded $L^2(\mathbb R^3) \mapsto L^p(\mathbb R, L^q(\mathbb R^3))$ by the Strichartz estimates for the free flow.
The bound on $T$ follows, and we get (as, as we will see shortly, $[\chi_{\text{obst}},\Delta]\psi(-h^{2}\Delta_D)e^{it\Delta_{D}}u_{0} \in L^2 H_{\rm comp}^{-1/2}$ by Proposition \ref{prop:smooth_wo})}
\begin{align*}
\Vert w_{h}\Vert_{L^{p}L^{q}}&\lesssim\Vert(1-\chi_{\text{obst}})\psi(-h^{2}\Delta_D)e^{it\Delta_{D}}u_{0}\Vert_{L^{2}(\mathbb{R}^{n})} \\
&\quad +\Vert[\chi_{\text{obst}},\Delta]\psi(-h^{2}\Delta_D)e^{it\Delta_{D}}u_{0}\Vert_{L^{2}H^{-1/2}}.
\end{align*}
Note that $[\chi_{\rm obst},\Delta]$ is supported in a compact set and outside
some neighborhood of $\Theta_{1}\cup\Theta_{2}\cup\mathcal{R}$. So,
there exists $\widetilde{\chi}\in C_{c}^{\infty}$ such that $\widetilde{\chi}=1$
on the support of $[\chi_{\text{obst}},\Delta]$ and vanishing
in a neighborhood of $\mathcal{R}$. Then, by the smoothing estimate
without loss in the non trapping region (Proposition \ref{prop:smooth_wo}), we get
\begin{align*}
\Vert[\chi_{\text{obst}},\Delta]\psi(-h^{2}\Delta_D)e^{it\Delta_{D}}u_{0}\Vert_{L^{2}H^{-1/2}} & = \Vert[\chi_{\text{obst}},\Delta]\widetilde{\chi}\psi(-h^{2}\Delta_D)e^{it\Delta_{D}}u_{0}\Vert_{L^{2}H^{-1/2}}\\
 & \leq \Vert\widetilde{\chi}\psi(-h^{2}\Delta_D)e^{it\Delta_{D}}u_{0}\Vert_{L^{2}H^{1/2}}\\
 & \lesssim \Vert\psi(-h^{2}\Delta_D)u_{0}\Vert_{L^{2}}.
\end{align*}
Finally, the estimate 
\[
\Vert\widetilde{\psi}(-h^{2}\Delta_D)w_{h}\Vert_{L^{r}}\leq\Vert w_{h}\Vert_{L^{r}}
\]
of \cite{Square} permits as in \cite{MR2672795} to conclude that
\[
\Vert\widetilde{\psi}(-h^{2}\Delta_D)(1-\chi_{\text{obst}})\psi(-h^{2}\Delta_D)e^{it\Delta_{D}}u_{0}\Vert_{L^{p}L^{q}}\leq\Vert\psi(-h^{2}\Delta_D)e^{it\Delta_{D}}u_{0}\Vert_{L^{2}}.
\]

\subsubsection{The second term}

Following again \cite{MR2672795}, we localise in time in the following
way: let $\varphi\in C_{c}^{\infty}(-1,2)$ equal to one on $[0,1]$,
and for $l\in\mathbb{Z}$
\begin{align*}
v_{h,l} & =  \varphi(t/h-l)\chi_{\text{obst}}(1-\chi_{\text{ray}})\psi(-h^{2}\Delta_D)e^{it\Delta_{D}}u_{0},\\
V_{h,l} & =  \left(\varphi(t/h-l)A+i\frac{\varphi'(t/h-l)}{h}\chi_{\text{obst}}(1-\chi_{\text{ray}})\right)\psi(-h^{2}\Delta_D)e^{it\Delta_{D}}u_{0},
\end{align*}
where
\[
A=\Delta\chi_{\text{obst}}(1-\chi_{\text{ray}})-\chi_{\text{obst}}\Delta\chi_{\text{ray}}-\nabla\chi_{\text{ray}}\cdot\nabla\chi_{\text{obst}}+(1-\chi_{\text{ray}})\nabla\chi_{\text{obst}}\cdot\nabla-\chi_{\text{obst}}\nabla\chi_{\text{ray}}\cdot\nabla.
\]
 These quantities verify
\[
\begin{cases}
i\partial_{t}v_{h,l}+\Delta_{D}v_{h,l}&=  V_{h,l},\\
v_{h,l|hl+2h>t>hl-h}&= 0.
\end{cases}
\]
Note that the support of $A$ is included in that of $\chi_{\text{obst}}(1-\chi_{\text{ray}})$.
Let $Q\subset\mathbb{R}^{3}$ be a cube sufficiently large to contain
$\Theta_{1}\cup\Theta_{2}$. We denote by $S$ the punctured torus
obtained by removing $\Theta_{1}\cup\Theta_{2}$ of $Q$. Now, let
$\widetilde{\chi}\in C_{c}^{\infty}$ be equal to one on the support of
$\chi_{\text{obst}}(1-\chi_{\text{ray}})$, supported outside a neighborhood
of $\mathcal{R}$ and in a neighborhood of $\partial\Omega$ such
that $\Delta_{D}$ and $\Delta_{S}$ coincides on its support (\figref{S}).
Writing $v_{h,l}=\widetilde{\chi}v_{h,l}$ and $V_{h,l}=\widetilde{\chi}V_{h,l}$,
one can show exactly as in \cite[p.284 - 287]{MR2672795}, using our smoothing estimate without loss outside
the trapping region of Proposition \ref{prop:smooth_wo} (instead of the smoothing estimate
of \cite{MR2066943} used in \cite{MR2672795}), that
\begin{equation}\label{eq:IE1}
\Vert\widetilde{\psi}(-h^{2}\Delta_D)\chi_{\text{obst}}(1-\chi_{\text{ray}})\psi(-h^{2}\Delta_D)e^{it\Delta_{D}}u_{0}\Vert_{L^{p}L^{q}}\leq\Vert\psi(-h^{2}\Delta_D)e^{it\Delta_{D}}u_{0}\Vert_{L^{2}}.
\end{equation}
\rev{We sketch the argument, and refer to \cite[p.284 - 287]{MR2672795} for the details.
As $v_{h,l}=\widetilde{\chi}v_{h,l}$ and $V_{h,l}=\widetilde{\chi}V_{h,l}$, these functions verify the corresponding problems in the punctured torus $S$:
\[
\begin{cases}
i\partial_{t}v_{h,l}+\Delta_{S}v_{h,l}&=  V_{h,l},\\
v_{h,l|hl+2h>t>hl-h}&= 0,
\end{cases}
\]
and writing
$$
\widetilde \psi(-h^2\Delta_D) \widetilde \chi = 
\chi_1 \widetilde \psi(-h^2\Delta_S) \widetilde \chi
+ (1-\chi_1) \widetilde \psi(-h^2\Delta_D) \widetilde \chi + \chi_1 (\psi(-h^2\Delta_D) - \psi(-h^2\Delta_S)) \widetilde \chi
$$
where $\chi_1 \in C^\infty_c(\mathbb R^3)$ is equal to one on the support of $\widetilde \chi$, we decompose
\begin{multline*}
\widetilde \psi(-h^2\Delta_D) v_{h,l} = \underbrace{\chi_1 \int_{hl-l}^t e^{i(t-s)\Delta_S} \widetilde \psi(-h^2\Delta_S) V_{h,l}(s) \, ds}_{=:v_{h,l,m}} \\
+  \underbrace{(1-\chi_1)\int_{hl-l}^t \widetilde \psi(-h^2\Delta_D) e^{i(t-s)\Delta_S}V_{h,l}(s)\,ds}_{=:v_{h,l,f}}+  \underbrace{\chi_1 (\psi(-h^2\Delta_D) - \psi(-h^2\Delta_S)) \widetilde \chi v_{h,l}}_{=:v_{h,l,s}}.
\end{multline*}
These three terms are estimated as follow:
\begin{itemize}
\item The term $v_{h,l,f}$ is supported away from the boundary, hence it can be estimated as in \S\ref{sss:11}, which gives
$$
\Vert v_{h,l,f} \Vert_{L^p(I^h_l, L^q(\Omega))} \lesssim \int_{I_h^l} \Vert  V_{h,l}(s) \Vert_{L^2}\, ds.
$$
\item The term $v_{h,l,m}$ enters the framework of the Strichartz estimates on compact Riemannian manifolds with strictly geodesically  convex boundary of \cite[Theorem 1.3]{MR2672795}, which gives
$$
\Vert v_{h,l,f} \Vert_{L^p(I^h_l, L^q(\Omega))} \lesssim \int_{I_h^l} \Vert \widetilde \psi(-h^2\Delta_S) V_{h,l}(s) \Vert_{L^2}\, ds,
$$
and one can eliminate $\widetilde \psi(-h^2\Delta_S)$ on the right-hand side using \cite[Theorem 1.1]{Square}.
\item The term $v_{h,l,s}$ is negligible in the sense that
$$
\Vert v_{h,l,s} \Vert_{L^p(I^h_l, L^q(\Omega))} 
\lesssim C_N h^N \Vert V_{h,l} \Vert_{L^2(I^h_l, H_0^{1-3/q})},
$$
this is \cite[Lemma 3.4]{MR2672795}.
\end{itemize}
Using Cauchy-Schwarz in time, the definition of $V_{h,l}$ and summing back in $l$, one obtain
\begin{align*}
&\Vert\widetilde{\psi}(-h^{2}\Delta_D)\chi_{\text{obst}}(1-\chi_{\text{ray}})\psi(-h^{2}\Delta_D)e^{it\Delta_{D}}u_{0}\Vert^p_{L^{p}L^{q}}\\
&\quad\lesssim h^{p/2} \Vert \widetilde \chi \psi(-h^2\Delta_D)e^{it\Delta_D} u_0 \Vert^p_{L^2(\mathbb R, H^1_0(\Omega))} \\
&\quad\quad+ 
h^{-p/2} \Vert \widetilde \chi \psi(-h^2\Delta_D)e^{it\Delta_D} u_0 \Vert^p_{L^2(\mathbb R, L^2(\Omega))}  
\\
&\quad\quad+ C_N h^{p(N+1/2)} \Vert \widetilde \chi \psi(-h^2\Delta_D)e^{it\Delta_D} u_0 \Vert^p_{L^2(\mathbb R, H_0^{3/2-3/q})} \\
&\quad\quad+ C_N h^{p(N-1/2)} \Vert \widetilde \chi \psi(-h^2\Delta_D)e^{it\Delta_D} u_0 \Vert^p_{L^2(\mathbb R, H_0^{1-3/q})}
\end{align*}
Each of these terms is estimated using our smoothing estimate without loss outside
the trapping region of Proposition \ref{prop:smooth_wo} (instead of the smoothing estimate
of \cite{MR2066943} used in \cite{MR2672795}), and we obtain (\ref{eq:IE1}) as announced.}
\begin{figure}
\includegraphics[scale=0.2]{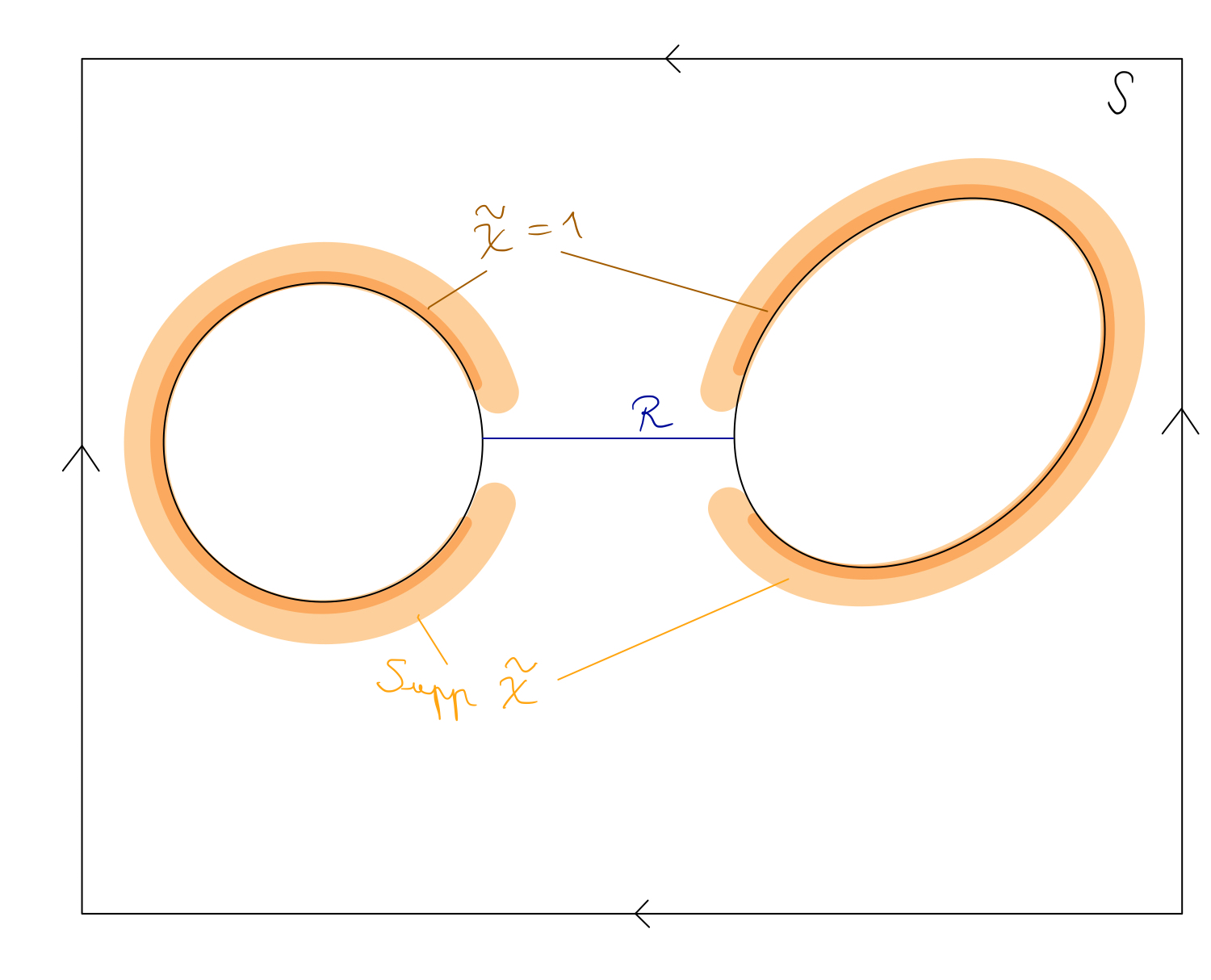}\caption{\label{fig:S}}

\end{figure}

\subsection{The third term: near the periodic ray}

We recall the smoothing effect with logarithmic loss obtained in \cite{MR2066943}
in the case of the exterior of many convex obstacles verifying the
Ikawa's condition, which is always verified in the present framework
of the exterior of two balls:
\begin{prop}
For any $\chi\in C_{c}^{\infty}(\mathbb{R}^{d})$ and any $u_{0}\in L^{2}(\Omega)$
such that $u_{0}=\psi(-h^{2}\Delta_D)u_{0}$ , we have
\begin{equation}
\Vert\chi e^{it\Delta_{D}}u_{0}\Vert_{L^{2}(\mathbb{R},L^{2})}\lesssim(h|\log h|)^{\frac{1}{2}}\Vert u_{0}\Vert_{L^{2}}\label{eq:smooth_logloss}
\end{equation}
\end{prop}
Let us denote, in this subsection, $\chi=\chi_{\text{osbt}}\chi_{\text{ray}}$
and $u=\psi(-h^{2}\Delta_D)e^{it\Delta_{D}}u_{0}$. We will here 
moreover use our assumption (\ref{eq:butult}) and follow the
method of \cite{MR2720226} to control the third term of (\ref{eq:dec-red}).

In the spirit of \cite{MR2720226}, we will localize in time intervals
of length $h|\log h|$, on which we can apply the semi-classical Strichartz
estimates of Proposition \propref{semilog}. Consider $\varphi\in C_{c}^{\infty}((-1,1))$
satisfying $\varphi\geq0$, $\varphi(0)=1$ and $\sum_{j\in\mathbb{Z}}\varphi(s-j)=1$.
We decompose
\[
\chi u=\sum_{j\in\mathbb{Z}}\varphi(\frac{t}{h|\log h|}-j)\chi u=:\sum_{j\in\mathbb{Z}}u_{j}.
\]
The $u_{j}$ satisfy the equation
\[
(i\partial_{t}+\Delta_{D})u_{j}=F_{j}+G_{j}
\]
with
\begin{align*}
F_{j} & =(h|\log h|)^{-1}\varphi'(\frac{t}{h|\log h|}-j)\chi u,\\
G_{j} & =2\varphi(\frac{t}{h|\log h|}-j)\big(\nabla\chi\cdot\nabla u-\rev{\frac 12}\Delta_{D}\chi u\big).
\end{align*}
Let us denote

\begin{align*}
v_{j}(t) & =\int_{(j-1)h|\log h|}^{t}e^{i(t-s)\Delta_{D}}F_{j}(s)ds,\\
w_{j}(t) & =\int_{(j-1)h|\log h|}^{t}e^{i(t-s)\Delta_{D}}G_{j}(s)ds.
\end{align*}
Clearly, $u_{j}=v_{j}+w_{j}$. If we define
\begin{align*}
\widetilde{v}_{j}(t) & =e^{it\Delta_{D}}\int_{(j-1)h|\log h|}^{(j+1)h|\log h|}e^{-is\Delta_{D}}F_{j}(s)ds,\\
\widetilde{w}_{j}(t) & =e^{it\Delta_{D}}\int_{(j-1)h|\log h|}^{(j+1)h|\log h|}e^{-is\Delta_{D}}G_{j}(s)ds,
\end{align*}
the Christ-Kiselev lemma allows us to estimate the $L^{p}L^{q}$ norms
of $\widetilde{v}_{j}$ and $\widetilde{w}_{j}$ instead of $v_{j}$ and $w_{j}$. 

We can use the semi-classical Strichartz estimate on logarithmic interval
of Proposition \propref{semilog} to estimate $\Vert\widetilde{v}_{j}\Vert_{L^{p}L^{q}}$:
\[
\Vert\widetilde{v}_{j}\Vert_{L^{p}L^{q}}\lesssim\Vert\int_{(j-1)h|\log h|}^{(j+1)h|\log h|}e^{is\Delta_{D}}F_{j}(s)ds\Vert_{L^{2}}
\]
We take $\widetilde{\chi}\in C_{c}^{\infty}$ equal to one on the support
of $\chi$ and use the dual version of (\ref{eq:smooth_logloss})
to get
\begin{align*}
\Vert\widetilde{v}_{j}\Vert_{L^{p}L^{q}} & \lesssim  \Vert\int_{(j-1)h|\log h|}^{(j+1)h|\log h|}e^{is\Delta_{D}}F_{j}(s)ds\Vert_{L^{2}}\\
 &   \quad =\frac{1}{h|\log h|}\Vert\int_{(j-1)h|\log h|}^{(j+1)h|\log h|}e^{is\Delta_{D}}\widetilde{\chi}\varphi'(\frac{t}{h|\log h|}-j)\chi uds\Vert_{L^{2}}\\
 & \lesssim  \frac{1}{h|\log h|}\times(h|\log h|)^{\frac{1}{2}}\Vert\varphi'(\frac{t}{h|\log h|}-j)\chi u\Vert_{L^{2}L^{2}},
\end{align*}
so we get
\[
\Vert\widetilde{v}_{j}\Vert_{L^{p}L^{q}}\lesssim\frac{1}{(h|\log h|)^{1/2}}\Vert\varphi'(\frac{t}{h|\log h|}-j)\chi u\Vert_{L^{2}L^{2}}.
\]

Let us now estimate $\Vert\widetilde{w}_{j}\Vert_{L^{p}L^{q}}$. Again,
because of Proposition \propref{semilog}, we have
\[
\Vert\widetilde{w}_{j}\Vert_{L^{p}L^{q}}\lesssim\Vert\int_{(j-1)h|\log h|}^{(j+1)h|\log h|}e^{is\Delta_{D}}G_{j}(s)ds\Vert_{L^{2}}.
\]
But $G_{j}$ is localized away from the periodic ray. We take $\widetilde{\chi}\in C_{c}^{\infty}$
equal to one on the support of $\nabla\chi$ and vanishing near the
periodic ray. Then $G_{j}=\widetilde{\chi}G_{j}$. Hence we can use the
dual estimate of the smoothing effect without loss in the non trapping
region of Proposition \propref{smooth_wo} to get
\begin{align*}
\Vert\widetilde{w}_{j}\Vert_{L^{p}L^{q}} & \lesssim  \Vert\int_{(j-1)h|\log h|}^{(j+1)h|\log h|}e^{is\Delta_{D}}G_{j}(s)ds\Vert_{L^{2}}=\Vert\int_{(j-1)h|\log h|}^{(j+1)h|\log h|}e^{is\Delta_{D}}\widetilde{\chi}G_{j}(s)ds\Vert_{L^{2}}\\
 & \lesssim  \Vert G_{j}\Vert_{L^{2}H^{-1/2}}\\
 & \lesssim  \Vert\varphi(\frac{t}{h|\log h|}-j)\rev{\widetilde \chi} u\Vert_{L^{2}H^{1/2}}.
\end{align*}

We use the Christ-Kiselev lemma (Lemma \ref{lem:CK}) twice, take the square and sum to
obtain 
\begin{align*}
\sum_{j\in\mathbb{Z}}\Vert u_{j}\Vert_{L^{p}L^{q}}^{2} & \lesssim  \sum_{j\in\mathbb{Z}}\left(\frac{1}{h|\log h|}\Vert\varphi'(\frac{t}{h|\log h|}-j)\chi u\Vert_{L^{2}L^{2}}^{2}+\Vert\varphi(\frac{t}{h|\log h|}-j)\rev{\widetilde \chi} u\Vert_{L^{2}H^{1/2}}^{2}\right)\\
 & \lesssim  \frac{1}{h|\log h|}\Vert\chi u\Vert_{L^{2}L^{2}}^{2}+\Vert \rev{\widetilde \chi} u\Vert_{L^{2}H^{1/2}}^{2},
\end{align*}
the first term is controlled using the smoothing estimate with logarithmic
loss (\ref{eq:smooth_logloss}), and the second the smoothing estimate
on the non-trapping region (\ref{eq:lswlnt}). Hence we get 
\[
\sum_{j\in\mathbb{Z}}\Vert u_{j}\Vert_{L^{p}L^{q}}^{2}\lesssim\Vert\psi(-h^{2}\Delta_D)u_{0}\Vert_{L^{2}}^{2}.
\]
But, because of the continuous embedding $l^{2}(\mathbb{Z})\hookrightarrow l^{p}(\mathbb{Z})$
for $p\geq2$ we know that
\[
\Vert\chi u\Vert_{L^{p}L^{q}}\sim\left(\sum_{j\in\mathbb{Z}}\Vert u_{j}\Vert_{L^{p}L^{q}}^{p}\right)^{1/p}\lesssim\left(\sum_{j\in\mathbb{Z}}\Vert u_{j}\Vert_{L^{p}L^{q}}^{2}\right)^{1/2}
\]
and we thus can conclude:
\[
\Vert\widetilde{\psi}(-h^{2}\Delta_D)\chi_{\text{ray}}\chi_{\text{obst}}\psi(-h^{2}\Delta_D)e^{it\Delta_{D}}u_{0}\Vert_{L^{p}L^{q}}\lesssim\Vert\psi(-h^{2}\Delta_D)u_{0}\Vert_{L^{2}}^{2}.
\]

\subsection{Conclusion}

We conclude from the two previous subsections that the estimate (\ref{eq:butult})
implies the estimate 
\[
\Vert e^{it\Delta_{D}}\psi(-h^{2}\Delta_D)u_{0}\Vert_{L^{p}(\mathbb{R},L^{q}(\Omega))}\leq\Vert\psi(-h^{2}\Delta_D)u_{0}\Vert_{L^{2}}.
\]
Like in \cite{MR2672795}, we can remove the frequency cut-off
to get
\[
\Vert e^{it\Delta_{D}}u_{0}\Vert_{L^{p}(\mathbb{R},L^{q}(\Omega))}\leq\Vert u_{0}\Vert_{L^{2}}.
\]
This ends the proof of Proposition \propref{semilog}: the estimate (\ref{eq:butult})
implies \thmref{main}. Thus,
the rest of the paper will be devoted to prove (\ref{eq:butult}). 
\qed

\rev{
\subsection{An a-priori Strichartz estimate with polynomial loss}

We end this Section by observing that, without making the assumption that we have a Strichartz estimate in logarithmic times near the trapped ray, we obtain the following ``a-priori'' Strichartz estimate with logarithmic loss:

\begin{prop} \label{prop:st_logloss}
Let $\psi \in C^\infty_c(\mathbb R_+^*)$, $\psi \neq 0$. There exists $C>0$ so that, for any $u_0 \in L^2$,\begin{equation}
\Vert e^{it\Delta_{D}}\psi(-h^{2}\Delta_D)u_{0}\Vert_{L^{p}(\mathbb R, L^{q}(\Omega))}\leq C|\log h| \Vert u_{0}\Vert_{L^{2}}.
\end{equation}
\end{prop}
\begin{proof}
The proof is a simplification of the above and we therefore only briefly sketch it. We cut the solution in only two parts, a part away from the obstacle and a part near the obstacle, for which we split in times interval of lenghts $h$ (instead of $h|\log h|$): the $|\log h|^{\frac 12}$ loss coming from the smoothing estimate of Proposition \ref{eq:smooth_logloss} is therefore not compensated and we obtain the a-priori estimate with loss above (observe that we have to use Proposition \ref{eq:smooth_logloss} twice, hence the $|\log h| = (|\log h|^{\frac 12})^2$ loss).
\end{proof}

\section{Reduction to data supported away from the boundary and microlocally near the trapped set}

As we will need various nested spectral cut-offs, we introduce the following notation to lighten the redaction:
\begin{defn}
For $\varphi_1, \varphi_2 \in C^\infty_c(\mathbb R_+^*)$, we will denote 
$$
\varphi_1 \prec \varphi_2 \text{ if }\varphi_2 = 1 \text{ on }\operatorname{supp} \varphi_1.
$$
\end{defn}

The purpose of this section is to reduce further to data supported away from the boundary and microlocally near the trapped set, in the following sense. We recall that $\mathcal K \subset {}^b T^* \Omega$ is the trapped set and $\mathcal R\subset \mathbb R^3$ its projection in space.

\begin{prop} \label{prop:mega_red}
Let $\psi \in C^\infty_c(\mathbb R_+^*)$, $\psi\neq 0$. Assume that there exists
\begin{itemize} 
\item $\chi \in C^\infty_c(\mathbb R^3)$ equal to one near $\mathcal R$,
\item a small neighborhood $\mathcal D \subset {}^b T^* \Omega$ of $\mathcal K$,
\item a small neighborhood $\mathcal V \subset \Omega$ of $\partial \Omega$,
\item $q\in C^\infty_c(T^*\mathbb R^3)$ equal to one in a neighborhood of $(\mathcal D \backslash T^* \mathcal V) \cap \{ |\xi|^2 \in \operatorname{supp}\psi\}$, and with $d(\pi_x{\operatorname{supp}q, \partial\Omega})\geq c > 0$,
\end{itemize}
and $C>0$ so that for any $u_0 \in L^2$,
$$
\Vert\chi e^{ith\Delta_{D}}\psi(-h^{2}\Delta_D)\operatorname{Op}_h (q) u_{0}\Vert_{L^{p}(0,\epsilon|\log h|)L^{q}(\Omega)}\leq Ch^{-1/p}\Vert u_{0}\Vert_{L^{2}}.
$$
Then, Theorem \ref{thm:main} holds.
\end{prop}

By Proposition \ref{prop:semilog}, Theorem \ref{thm:main} reduces to the estimate (\ref{eq:butult}).
After some preliminary, we will successively show that the estimate (\ref{eq:butult}) for data near the trapped set and away from the boundary implies the same estimate for
data in a ball and away from boundary ; which in turn implies (\ref{eq:butult}) 
for data in a ball ; which finally implies (\ref{eq:butult}) for all data.

\subsection{Preliminary: escape from the obstacle}

We need the following preparatory lemma, stating that if one starts far from the obstacles and in directions pointing away from the obstacles, one never observe the solution near the obstacles. Importantly, we don't assume that the symbol $a$ below is compactly supported in space (but only that it is uniformly bounded).

\begin{lem} \label{lem:away_is_away}
Let $\psi \in C^\infty_c(\mathbb R_+^*)$.
Let $R_0 > 0$ be so that $\Theta_1\cup\Theta_2 \subset B(0,R_0)$, and $a \in S^0(\mathbb R^3 \times \mathbb R^3)$  so that 
$$
\forall (x,\xi) \in \operatorname{supp}a, \quad \forall t \geq 0, \quad x + t\xi \notin B(0,R_0).
$$ 
Then, for any bounded family $(u_{0,h})$ of $L^2(\Omega)$, which we denote $u_0$, for any $\chi \in C^\infty_c$ so that $\operatorname{supp} \chi \subset B(0,R_0)$, we have in any $H^\sigma$
\begin{equation}\label{eq:escr1}
\sup_{t \in [0, \epsilon|\log h|]} \Vert \chi e^{ith\Delta_D}\psi(-h^2\Delta_D)\operatorname{Op}_h(a) u_0 \Vert_{H^\sigma} = O(h^\infty) \Vert u_0 \Vert_{L^2},
\end{equation}
hence in particular
\begin{equation}\label{eq:escr2}
\Vert \chi e^{ith\Delta_D}\psi(-h^2\Delta_D)\operatorname{Op}_h(a) u_0 \Vert_{L^p((0, \epsilon|\log h|), L^q)} = O(h^\infty) \Vert u_0 \Vert_{L^2}.
\end{equation}
The corresponding statement holds for negative times.
\end{lem}
\begin{proof}
The strategy is the following: we will show the analogous result for the solution in  the free space $\mathbb R^3$ associated with the same initial data. It will follow in particular that the solutions in $\mathbb R^3$ and in $\Omega$ are close near the support of $\chi$, and the result will follow.

First, as $\pi_x\operatorname{supp}a$ is away from the boundary, there is
$\chi_1\in C^\infty(\mathbb R^3)$, supported away from the boundary and
equal to one outside a ball, so that
$\operatorname{Op}_h(a)=\chi_1\operatorname{Op}_h(a)+O(h^\infty)$ in any
Sobolev operator norms. Take
$\chi_2\in C^\infty(\mathbb R^3)$ supported away from the boundary and
equal to one near $\operatorname{supp}\chi_1$. By Lemma
\ref{lem:psi_away}, we have for any $\sigma \geq 0$
$$
\psi(-h^2\Delta_D)\chi_1
=\chi_2 \operatorname{Op}_h(\psi(|\xi|^2)) \chi_1 +O(h^\infty)_{H^{-\sigma}(\Omega)\to H^{\sigma}(\Omega)}. 
$$
Let $b_h$ a symbol so that
$$
\chi_2 \operatorname{Op}_h(\psi(|\xi|^2)) \chi_1 \operatorname{Op}_h(a) = \operatorname{Op}_h(b_h) + O(h^\infty)_{H^{-\sigma}(\Omega)\to H^{\sigma}(\Omega)},
$$
with
\begin{equation}\label{eq:esc1}
\operatorname{supp}b_h \subset \operatorname{supp} a \cap \{|\xi|^2 \in \operatorname{supp}\psi\}.
\end{equation}
By the above, it suffices to show the conclusion with $e^{ith\Delta_D}\psi(-h^2\Delta_D)\operatorname{Op}_h(a) u_0$ replaced with
$$
u(t) := e^{ith\Delta_D} \operatorname{Op}_h(b_h) u_0.
$$
Let $v$ denote the corresponding free solution 
$$
v(t) :=  e^{ith\Delta_{\mathbb R^3}} \operatorname{Op}_h(b_h) u_0.
$$

Let $\gamma \in C^\infty_c(B(0,R_0))$. We first show that, for any $\sigma$
\begin{equation}\label{eq:esc2}
\sup_{t \in [0, \epsilon|\log h|]} \Vert \gamma v(t) \Vert_{H^\sigma} = O(h^\infty)\Vert u_0 \Vert_{L^2}.
\end{equation}
This will follow by non-stationary phase applied to the free propagator. To enjoy nicer kernel formula, we first switch to a right quantisation $\operatorname{Op}_h^r$: select $\widetilde b_h \in S^0$, with the same support properties (\ref{eq:esc1}) than $b_h$, so that $\operatorname{Op}_h b_h = \operatorname{Op}^r_h \widetilde b_h + O(h^\infty)$ in any $L^2 \to H^\sigma$. For 
\begin{equation}\label{eq:esc3}
\widetilde v := e^{ith\Delta_{\mathbb R^3}} \operatorname{Op}^r_h(\widetilde b_h) u_0 = v + O(h^\infty)_{H^\sigma}\Vert u_0\Vert_{L^2},
\end{equation}
we have
$$
\gamma \widetilde v(t,x) = (2\pi h)^{-3}\int\int
 e^{i\Phi(t,x,y,\xi)/h} \gamma(x) \widetilde b_h(y,\xi) u_0(y)\,dyd\xi,
 \quad
 \Phi=(x-y)\cdot\xi-t|\xi|^2.
$$
Denote the kernel arising above
$$
K(x,y,t) := (2\pi h)^{-3}\int e^{i\Phi(t,x,y,\xi)/h} \gamma(x) \widetilde b_h(y,\xi) \, d\xi.
$$
We will integrate by parts, writing
$$
e^{i\Phi(t,x,y,\xi)/h} =  h \mathcal L e^{i\Phi(t,x,y,\xi)/h}, \quad \mathcal L := \frac 1 i \frac{\nabla_\xi\Phi}{|\nabla_\xi \Phi|^2} \cdot \nabla_\xi.
$$
As $\nabla_\xi \Phi = x - y - {2t\xi} $, we have by assumption, uniformly in $t\geq 0$,
$$
\forall x \in \operatorname{supp}\gamma, \; \forall (y,\xi) \in \operatorname{supp}\widetilde b_h, \; | \nabla_\xi \Phi |  \geq c > 0.
$$
In addition, as $\partial_\xi \nabla_\xi \Phi = -2t I$,  we get for any multi-indice $\alpha$
$$
\partial^\alpha_\xi \Big(\frac{\nabla_\xi\Phi}{|\nabla_\xi \Phi|^2}\Big) \leq C_{|\alpha|}|t|^\alpha.
$$
Thus, integrating $N$ times by part
$$
K(x,y,t) = (2\pi h)^{-3} h^{N} \int
 e^{i\Phi(t,x,y,\xi)/h} (\mathcal L ^*)^N \widetilde b_h(y,\xi)\gamma(x)\,d\xi, 
$$
gives, using the fact that $\widetilde b_h \in S^0$ and is compactly supported in $\xi$,
\begin{equation}\label{eq:esc4}
|K(x,y,t)| \lesssim_N h^{-3+N} (1+t)^N.
\end{equation}
As $\widetilde b_h(y,\xi)$, and hence $K(x,y,t)$, is not necessarily compactly supported in $y$, the above is not enough to conclude. We will combine it with the decay properties of $K$ in $y$, as follows now.
Let $C>0$ sufficiently
large so that
$$
|y|>C(1+t), \; x\in\operatorname{supp}\gamma, \; \xi \in \pi_\xi \operatorname{supp}\widetilde b_h \implies |\nabla_\xi\Phi|=|x-y-2t\xi|\geq\tfrac12|y|,
$$
which is possible because $\operatorname{supp}\gamma$ and $\pi_\xi \operatorname{supp}\widetilde b_h$ are compact.
In this region,
$$
 \left|\partial_\xi^\alpha
 \left(\frac{\nabla_\xi\Phi}{|\nabla_\xi\Phi|^2}\right)\right|
 \lesssim_\alpha t^{|\alpha|}|y|^{-1-|\alpha|}
 \lesssim_\alpha\langle y\rangle^{-1},
$$
and the same integration by parts therefore gives
\begin{equation}\label{eq:esc5}
 |K_h(t,x,y)|\leq C_Nh^{N-3}\langle y\rangle^{-N},
 \qquad |y|>C(1+t).
\end{equation}
We now cut the integral in two
$$
\gamma \widetilde v(t,x) =  \int_{|y|\leq C(1+t)} K(x,y,t)u_0(y) \, dy+ \int_{|y|\geq C(1+t)} K(x,y,t)u_0(y) \, dy.
$$
On the first term, (\ref{eq:esc4}) together with Cauchy-Schwartz in $y$ gives
$$
\Big|  \int_{|y|\leq C(1+t)} \Big|  \lesssim_N h^{-3+N} (1+t)^N |\{ |y|\leq 1+Ct\}|^{\frac 12} \Vert u_0 \Vert_{L^2} \lesssim_N h^{-3+N} (1+t)^{N+\frac 32} \Vert u_0\Vert_{L^2}
$$
whereas on the second term, by (\ref{eq:esc5}), for $N\geq2$
$$
\Big| \int_{|y|\geq C(1+t)} \Big| \lesssim_N h^{-3+N}  \Vert \langle y\rangle^{-N} \Vert_{L^2} \Vert u_0 \Vert_{L^2} \lesssim_N h^{-3+N} \Vert u_0\Vert_{L^2}.
$$
As $\gamma$ is compactly supported, integrating in space gives (\ref{eq:esc2}) for $\sigma = 0$ (for $\widetilde v$, and hence for $v$ by (\ref{eq:esc3})). The same argument applied to the $\partial_x^\beta \gamma \widetilde v(t,x)$, which makes only appear supplementary $h^{-\sigma}$ factors, gives the $H^\sigma$ claim.

Let now $\beta \in C^\infty_c(B(0,R_0))$ be so that $\beta = 1$ near the obstacles and $\operatorname{supp}\chi$. We let
$$
z(t) := u(t) - (1-\beta) v(t).
$$
Observe that $z$ solves the following semiclassical Schr\"odinger equation with Dirichlet boundary conditions on $\Omega$
$$
\begin{cases}
 (i\partial_t + h\Delta_D)z=  h[\Delta,\beta]v = :F & \text{ in }\Omega,\\
  z(t) = 0 &\text{ on }\partial \Omega,\\
  z(0)= \beta \operatorname{Op}_h(b_h)u_0.
 \end{cases}
$$
First, as $\beta = 0$ on $\pi_x\operatorname{supp}$ a, 
$$
z(0) = \beta \operatorname{Op}_h(b_h)u_0 = \operatorname{Op}_h(\beta b_h)u_0 = 0.
$$
Second, taking $\gamma \in C^\infty_c(B(0,R_0))$ equal to one on $\operatorname{supp} \beta$, we have by (\ref{eq:esc2})
$$
\sup_{t \in [0, \epsilon|\log h|]} \Vert h[\Delta,\beta]v \Vert_{H^\sigma}  \lesssim \sup_{t \in [0, \epsilon|\log h|]} \Vert \gamma v \Vert_{H^{\sigma+1}} = O(h^\infty) \Vert u_0 \Vert_{L^2}.
$$
By Duhamel formula and Minkowski inequality, we thus obtain
\begin{equation} \label{eq:esc6}
\sup_{t \in [0, \epsilon|\log h|]} \Vert z(t) \Vert_{H^\sigma} = O(h^\infty)\Vert u_0 \Vert_{L^2}.
\end{equation}
Combining (\ref{eq:esc2}) with (\ref{eq:esc6}) gives (\ref{eq:escr1}). The second claim (\ref{eq:escr2}) follows by Sobolev embedding.

\end{proof}

\subsection{From data near the trapped set and away from the boundary, to data in a ball and away from boundary}

\begin{prop} \label{prop:red_to_ball_away}
Let $\psi, \psi_1 \in C^\infty_c(\mathbb R_+^*)$ with $\psi \prec \psi_1$.
Assume that there exists:
\begin{itemize} 
\item $\chi \in C^\infty_c(\mathbb R^3)$ equal to one near $\mathcal R$,
\item a small neighborhood $\mathcal D \subset {}^b T^* \Omega$ of $\mathcal K$,
\item a small neighborhood $\mathcal V \subset \Omega$ of $\partial \Omega$,
\item $q\in C^\infty_c(T^*\mathbb R^3)$ equal to one in a neighborhood of $(\mathcal D \backslash T^* \mathcal V) \cap \{ |\xi|^2 \in \operatorname{supp}\psi\}$, and with $d(\pi_x{\operatorname{supp}q, \partial\Omega})\geq c > 0$,
\end{itemize}
and $C>0$ so that for any $u_0 \in L^2$,
$$
\Vert\chi e^{ith\Delta_{D}}\psi_1(-h^{2}\Delta)\operatorname{Op}_h (q) u_{0}\Vert_{L^{p}(0,\epsilon|\log h|)L^{q}(\Omega)}\leq Ch^{-1/p}\Vert u_{0}\Vert_{L^{2}}.
$$
Then, there exists $\chi_0 \in C^\infty_c(\mathbb R^3)$ equal to one near $\mathcal R$ so that for any $\chi_{\rm int}^{\rm comp}\in C^\infty_c(\mathbb R^3)$ supported away from a small neighborhood of $\partial \Omega$, there exists $C>0$ so that for any $u_0 \in L^2$ 
$$
\Vert\chi_0 e^{ith\Delta_{D}}\psi(-h^{2}\Delta_D) \chi^{\rm comp}_{\rm int}u_{0}\Vert_{L^{p}(0,\epsilon|\log h|)L^{q}(\Omega)}\leq Ch^{-1/p}\Vert u_{0}\Vert_{L^{2}}.
$$
\end{prop}

We will need the following Strichartz estimates without loss for data with only a finite number of reflections, which will we use in the proof of Proposition \ref{lem:away_is_away}.

\begin{prop}[Strichartz for a finite number of reflections]
\label{prop:fr_strichartz}
Let $\psi \in C^\infty_c(\mathbb R_+^*)$.
Let $a\in C_c^\infty(T^*\mathbb R^d)$, supported in $\Omega$, be so that $d(\pi_x \operatorname{supp} a, \partial  \Omega)>0$ and so that every forward generalized broken ray from $\operatorname{supp} a$
has at most $N$ reflections. Then, for any $h>0$ small enough and any $u_0 \in L^2(\Omega)$
\begin{equation}
\label{eq:fr_strichartz}
\Vert e^{ith\Delta_D}\psi(-h^2\Delta_D)\operatorname{Op}_h(a) u_0 \Vert_{L^p((0,\epsilon|\log h|), L^q(\Omega))}
 \lesssim  Ch^{-1/p} \Vert u_0 \Vert_{L^2(\Omega)}.
\end{equation}
\end{prop}

\begin{proof}
We denote
$$
 u(t) : = e^{ith\Delta_D}\psi(-h^2\Delta_D)\operatorname{Op}_h(a)u_0.
$$

Let $0\leq\rho_j\leq1$ in $C_c^\infty(\mathbb R^3)$, $j=1,2$,
with disjoint supports, such that $\rho_j=1$ near $\Theta_j$, and set
$\rho_3 := 1-\rho_1-\rho_2$.
For $j=1,2$, the function $u_j:=\rho_ju$, is solution of the following semiclassical Schr\"odinger equation on $\mathbb R^3\setminus\Theta_j$
$$
\begin{cases}
 (i\partial_t+h\Delta_D)u_j=h[\Delta,\rho_j]u = :F_j & \text{ in }\mathbb R^3\setminus\Theta_j,\\
  u_j(t) = 0 &\text{ on }\partial \Theta_j, \\
  u_j(0)=\rho_j u(0).
 \end{cases}
$$
Therefore, by the global Strichartz estimates outside one strictly convex obstacle
of \cite{MR2672795}, we have, for $j=1,2$.
\begin{equation}
\label{eq:finite_model_est}
 \|u_j\|_{L^p((0,\epsilon|\log h|), L^q(\Omega_j))}
 \lesssim h^{-1/p}
 \left(\|\rho_ju(0)\|_{L^2(\Omega)}+\int_0^{\epsilon|\log h|}\|F_j(s)\|_{L^2(\Omega)}\,ds\right).
\end{equation}
For $j=3$, $u_3$ is similarly solution of a semiclassical Schr\"odinger equation with source term in $\mathbb R^3$, and by classical Strichartz estimates, the above also holds with $F_3 := h[\Delta,\rho_0]u$. Now, choose  $\kappa\in C_c^\infty(\Omega)$ so that  $\kappa = 1$ near $\operatorname{supp} \nabla \rho_0 \cup \operatorname{supp} \nabla \rho_1 \cup \operatorname{supp} \nabla \rho_2$. Summing (\ref{eq:finite_model_est}) for $j=1,2,3$, it follows that
\begin{equation}
\label{eq:finite_gluing}
 \|u\|_{L^p((0, \epsilon|\log h|), L^q(\Omega))}
 \lesssim Ch^{-1/p}
 \left(\|u_0\|_{L^2(\Omega)}+\int_0^{\epsilon|\log h|} h \|\kappa u(s)\|_{H^1}\,ds\right).
\end{equation}

To conclude, we will show that, for $h$ small enough
\begin{equation} \label{eq:fr_t1}
\int_0^{\epsilon|\log h|} h \|\kappa u(s)\|_{H^1}\,ds \lesssim \|u_0\|_{L^2(\Omega)}.
\end{equation}
To this end, let $R_0>0$ big enough so that
$$
B(0,R_0) \supset \pi_x \operatorname{supp}a \cup \operatorname{supp} \kappa \cup \Theta_1 \cup \Theta_2,
$$
and $S>0$ big enough so that every forward bicharacteristic from $\operatorname{supp} a \cap \{ |\xi|^2 \in \operatorname{supp}\psi\}$  lies in $T^*(\mathbb R^3 \backslash B(0,R_0+1))$ in times $s\geq S$. Such an $S$ exists thanks to our finite reflections assumption (and because $\psi$ is supported away from zero). We will show that
\begin{equation} \label{eq:fr_t2}
\int_S^{\epsilon|\log h|} h \|\kappa u(s)\|_{H^1}\,ds = O(h^\infty) \|u_0\|_{L^2(\Omega)}.
\end{equation}
As $S$ is fixed and $h \|\kappa u(s)\|_{H^1} \lesssim \Vert u_0 \Vert_{L^2}$, (\ref{eq:fr_t2}) immediately implies (\ref{eq:fr_t1}) and hence the result.
By Lemma \ref{lem:traces_away_bdd}, which we can use because $a$ is supported away from the boundary,
$$
{}^b\operatorname{WF}_h^xu(S) \subset \big\{ |x| \geq R_0+1, \frac{x\cdot\xi}{|x||\xi|}\geq c>0\big\} \cap \{ |\xi|^2 \in \operatorname{supp} \psi \}.
$$
Thus, there is $a_+ \in S^0$ so that
$$
\forall (x,\xi)\in\operatorname{supp}a_+, \; \forall t \geq 0, \; x+t\xi \not\in B(0,R_0)
$$
and
\begin{equation}\label{eq:Ref0WF}
{}^b\operatorname{WF}_h^x(1-\operatorname{Op}_h(a_+))u(S) = \emptyset.
\end{equation}
Now, using the fact that $a$ is compactly supported in space together with the semiclassical approximate finite speed of propagation of Lemma \ref{lem:semi_speed}, we can write, for some $\eta \in C^\infty_c$
$$
u(S) = \eta u(S) + O(h^\infty)_{H^s}\Vert u_0 \Vert_{L^2}.
$$
Now, take $\eta_1\in C^\infty_c(\mathbb R^3)$, equal to one on the support of $\eta$. We 
have, by the above, then Lemma \ref{ref:ell_space} together with (\ref{eq:Ref0WF}) (and using pseudo-locality to insert $\eta_1$)
\begin{align*}
\psi_1(-h^2\Delta_D)(1-\operatorname{Op}_h(a_+))u(S) &=  \psi_1(-h^2\Delta_D)\eta_1(1-\operatorname{Op}_h(a_+))\eta u(S) + O(h^\infty)_{H^s}\Vert u_0 \Vert_{L^2} \\
&= O(h^\infty)_{H^s}\Vert u_0 \Vert_{L^2},
\end{align*}
thus
$$
u(S) = \psi_1(-h^2\Delta_D) u(S) = \psi_1(-h^2\Delta_D)\operatorname{Op}_h(a_+)u(S) + O(h^\infty)_{H^s}\Vert u_0 \Vert_{L^2}.
$$
Lemma \ref{lem:away_is_away} gives (\ref{eq:fr_t2}), which ends the proof (using for example Proposition \ref{prop:st_logloss} to deal with the $O(h^\infty)$ reminder term).
\end{proof}

Before going to the proof of Proposition \ref{prop:red_to_ball_away}, we observe that a simplification of the above proof gives Strichartz estimates without loss in finite times, which will be useful in the sequel and which we record now:

\begin{prop}[Strichartz in finite times]
\label{prop:ft_strichartz}
Let $\psi \in C^\infty_c(\mathbb R_+^*)$.
Let $S>0$. We have for any $h>0$ and any $u_0 \in L^2(\Omega)$
\begin{equation}
\label{eq:fr_strichartz}
\Vert e^{ith\Delta_D}\psi(-h^2\Delta_D) u_0 \Vert_{L^p([0, S], L^q(\Omega))}
 \lesssim  Ch^{-1/p} \Vert u_0 \Vert_{L^2(\Omega)}.
\end{equation}
\end{prop}
\begin{proof}
We follow the beginning of the proof of Proposition \ref{prop:fr_strichartz}, replacing $\epsilon|\log h|$ by $S$, and simply notice that, as $S>0$ is finite
$$
\int_0^{S} h \|\kappa u(s)\|_{H^1}\,ds \lesssim_S \|u_0\|_{L^2(\Omega)}.
$$
\end{proof}

\begin{proof}[Proof of Proposition \ref{prop:red_to_ball_away}]
We take $\chi_0 \in C^\infty_c(\mathbb R^3)$, $\chi_0 = 1$ near $\mathcal R$, so that $\operatorname{supp} \chi_0 \subset \{ \chi = 1\}$.

Let $(b_i)_{0 \leq i \leq K}$, $b_i \in C^\infty_c(\mathbb R^3 \times \mathbb R^3)$, with $\pi_x \operatorname{supp}b_i$ away from the boundary for all $i$, be a partition of the unity of $T^*\operatorname{supp}\chi_{\rm int}^{\rm comp} \cap \{ |\xi|^2 \in \operatorname{supp}\psi\}$, that is, so that
$$
\sum_{i=0}^K b_i = 1 \text{ in a neighborhood of }T^*\operatorname{supp}\chi_{\rm int}^{\rm comp} \cap \{ |\xi|^2 \in \operatorname{supp}\psi\},
$$
and $N\geq 0$ be so that, for any $0\leq i \leq K$, 
\begin{align*}
\text{either: }& \exists t_i > 0, \; \phi_{t_i}(\operatorname{supp}b_i) \subset  {\mathcal D} \backslash  {T^* \mathcal V}, \\
\text{or: }& \text{any forward bicharacteristic from }\operatorname{supp}b_i\text{ intersects }\partial \Omega\text{ at most $N$ times}.
\end{align*}
This is possible by definition of the trapped set (and the fact if one enters a small microlocal neighborhood of the trapped set, one is both in a slightly bigger microlocal neighborhood of it and away from the boundary shortly after). In the first case above, we denote $i \in I$, in the second case, we denote $i \in J$. 

As $\chi_{\rm int}^{\rm comp}$ is supported away from the boundary, we have, by Lemma \ref{lem:psi_away}
$$
\psi(-h^2\Delta_D)\chi_{\rm int}^{\rm comp}
= \widetilde \chi_{\rm int}^{\rm comp} \operatorname{Op}_h(\psi(|\xi|^2))\chi_{\rm int}^{\rm comp}+O(h^\infty)_{L^2(\Omega)\to L^2(\Omega)},
$$
for $\widetilde \chi_{\rm int}^{\rm comp} \in C^\infty_c$, supported away from the boundary and equal to one on $\operatorname{supp} \chi_{\rm int}^{\rm comp}$.
As the full symbol of $\operatorname{Op}_h(\psi(|\xi|^2))\chi_{\rm int}^{\rm comp}$ is microlocally supported where 
$\sum_{i=0}^K b_i = 1$, we therefore have
\begin{align*}
\psi(-h^2\Delta_D)\chi_{\rm int}^{\rm comp}
&= \psi_1(-h^2\Delta_D) \widetilde \chi_{\rm int}^{\rm comp}  \operatorname{Op}_h(\psi(|\xi|^2))\chi_{\rm int}^{\rm comp}\sum_{i=0}^K \operatorname{Op}_hb_i \\
&\qquad+O(h^\infty)_{L^2(\Omega)\to L^2(\Omega)}.
\end{align*}
For conciseness, let us denote
\begin{align*}
C_i &:= \widetilde \chi_{\rm int}^{\rm comp}  \operatorname{Op}_h(\psi(|\xi|^2))\chi_{\rm int}^{\rm comp} \operatorname{Op}_h(b_i) \\
&= \psi(-h^2\Delta_D)\chi_{\rm int}^{\rm comp} \operatorname{Op}_h(b_i) + O(h^\infty),
\end{align*}
where the second line holds in any Sobolev operator norm.
By the decomposition above, it suffices to show that, for all $0\leq i \leq K$,
$$
\Vert\chi_0 \psi_1(-h^2\Delta_D) e^{ith\Delta_D}C_i  u_{0}\Vert_{L^{p}(0,\epsilon|\log h|)L^{q}(\Omega)} 
\leq Ch^{-1/p}\Vert u_{0}\Vert_{L^{2}}.
$$
By Proposition \ref{prop:fr_strichartz} and the definition of $J$, the above holds for any $i\in J$. It therefore remains to deal with the case $i\in I$.

Let $i\in I$. By Lemma \ref{lem:traces_away_bdd}, which we can use because $C_i$ coincides, modulo a $O(h^\infty)_{L^2\rightarrow L^2}$ term, with an operator with symbol supported away from the boundary, we get by 
 the definition of $I$
\begin{equation} \label{eq:jeud2}
{}^b\operatorname{WF}^x_h (e^{it_ih\Delta_D} C_i u_0) \subset ({\mathcal D \backslash T^* \mathcal V})\cap\{ |\xi|^2 \in \operatorname{supp}\psi\},
\end{equation}
thus
$$
{}^b\operatorname{WF}^x_h (\operatorname{Op}_h(1-q) e^{it_ih\Delta_D} C_i u_0) =\emptyset,
$$
And, similarly as in the proof of Proposition \ref{prop:fr_strichartz}, using the fact that $\widetilde \chi_{\rm int}^{\rm comp}$ is compactly supported together with the semiclassical approximate finite speed of propagation of Lemma \ref{lem:semi_speed}, we can use Lemma \ref{ref:ell_space} to get
\begin{equation} \label{eq:jeud3}
\psi_1(-h^2\Delta_D) e^{it_i h\Delta_D} C_i u_0 = \psi_1(-h^2\Delta_D) \operatorname{Op}_h(q) e^{it_i h\Delta_D} C_i u_0 + O(h^\infty)\Vert u_0\Vert_{L^2} \text{ in }L^2.
\end{equation}
 
Now, by the Strichartz estimates in finite times given by Proposition \ref{prop:ft_strichartz},
$$
\Vert \chi_0 \psi_1(-h^2\Delta_D) e^{ith\Delta_D} C_i u_0 \Vert_{L^p((0, t_i), L^q)}\lesssim h^{-1/p} \Vert u_0 \Vert_{L^2}.
$$
On the other hand
\begin{align*}
&\Vert \chi_0 \psi_1(-h^2\Delta_D) e^{ith\Delta_D} C_i u_0 \Vert_{L^p((t_i , \epsilon |\log h|), L^q)} \\
  &= \Vert \chi_0 \psi_1(-h^2\Delta_D) e^{ith\Delta_D} e^{it_i h\Delta_D} C_i u_0 \Vert_{L^p((0 , \epsilon |\log h| - t_i), L^q)} \\
&= \Vert \chi_0 e^{ith\Delta_D} \psi_1(-h^2\Delta_D) e^{it_i h\Delta_D} C_i u_0 \Vert_{L^p((0 , \epsilon |\log h| - t_i), L^q)}, 
\end{align*}
where, using (\ref{eq:jeud3}) and the assumption, we get
\begin{align*}
&\Vert \chi_0 e^{ith\Delta_D} \psi_1(-h^2\Delta_D) e^{it_i h\Delta_D} C_i u_0 \Vert_{L^p((0 , \epsilon |\log h| - t_i), L^q)} \\
&\qquad\lesssim\Vert \chi e^{ith\Delta_D} \psi_1(-h^2\Delta_D) e^{it_i h\Delta_D} C_i u_0 \Vert_{L^p((0 , \epsilon |\log h| - t_i), L^q)} \\
&\qquad\lesssim\Vert \chi e^{ith\Delta_D} \psi_1(-h^2\Delta_D) \operatorname{Op}_h(q) e^{it_i h\Delta_D} C_i u_0 \Vert_{L^p((0 , \epsilon |\log h| - t_i), L^q)} + O(h^\infty)\Vert u_0 \Vert_{L^2} \\
&\qquad\lesssim h^{-1/p}\Vert e^{it_i h\Delta_D} C_i u_0  \Vert_{L^2} + O(h^\infty)\Vert u_0 \Vert_{L^2} \\  
&\qquad\lesssim h^{-1/p} \Vert u_0 \Vert_{L^2} + O(h^\infty)\Vert u_0 \Vert_{L^2},  
\end{align*}
where we used the fact that $\chi_0 = \chi \chi_0$ in the first inequality, (\ref{eq:jeud3}) together with the a-priori Strichartz estimates with logarithmic loss of Proposition \ref{prop:st_logloss} to deal with the $O(h^\infty)$ error still in the second inequality, the assumption in the third inequality, and the fact that $e^{it_i h\Delta_D}$ is an $L^2$ isometry in the last one.
Summing up the above, we have 
\begin{equation}
\Vert\chi_0 \psi_1(-h^2\Delta_D) e^{ith\Delta_D}C_i  u_{0}\Vert_{L^{p}(0,\epsilon|\log h|)L^{q}(\Omega)} 
\leq Ch^{-1/p}\Vert u_{0}\Vert_{L^{2}}.
\end{equation}
which ends the proof.
\end{proof}

\subsection{From data in a ball and away from the boundary, to data in a ball}

\begin{prop}
\label{prop:reduc_away_bdd}
Let $\psi, \psi_1 \in C^\infty_c(\mathbb R_+^*)$ with $\psi \prec \psi_1$.
Assume that there exists $\chi_0 \in C^\infty_c(\mathbb R^3)$ equal to one near $\mathcal R$ so that for any $\chi_{\rm int}^{\rm comp}\in C^\infty_c(\mathbb R^3)$ supported away from a small neighborhood of $\partial \Omega$, there exists $C>0$ so that  for any $u_0 \in L^2$,
$$
\Vert\chi_0 e^{ith\Delta_{D}}\psi_1(-h^{2}\Delta) \chi^{\rm comp}_{\rm int}u_{0}\Vert_{L^{p}(0,\epsilon|\log h|)L^{q}(\Omega)}\leq Ch^{-1/p}\Vert u_{0}\Vert_{L^{2}}.
$$
Then, for any $\chi^{\rm comp} \in C^\infty_c(\mathbb R^3)$, there exists $C>0$ so that for any $u_0 \in L^2$,
$$
\Vert\chi_0 e^{ith\Delta_{D}}\psi(-h^{2}\Delta_D) \chi^{\rm comp} u_0 \Vert_{L^{p}(0,\epsilon|\log h|)L^{q}(\Omega)}\leq Ch^{-1/p}\Vert u_{0}\Vert_{L^{2}}.
$$
\end{prop}
\begin{proof}

We are here inspired by \cite{MR2672795}, Section 2, p.265-266. 

Let $\alpha_0,\beta_0$ so that $\operatorname{supp} \psi \subset [\alpha_0, \beta_0]$. 

Let us denote $d := d(\Theta_1, \Theta_2)$, and take $t_0 := \frac d{20\sqrt \beta_0}$, $t_1 := \frac{t_0}{8}$.
Let $\chi_b \in C^\infty_c(\mathbb R^3)$ be so that $\chi_b = 1$ near $\partial \Omega$, and supported close to $\partial \Omega$, 
and let $A \in C^\infty(\mathbb R^3)$ be so that $A = 0$ near $\partial \Omega$, $A=1$ away from $\partial \Omega$, and such that any broken space-time bicharacteristic starting from $\rho \in {}^bT^*(\operatorname{supp} \chi_b \times [0, 2t_1])$ with  $-\tau \in [\alpha_{0},\beta_{0}]$ ($\tau$ denoting the dual time variable) verifies
\begin{equation}\label{eq:red_bdd_bich__}
d(\pi_x\gamma_t(\rho), \operatorname{supp}(1-A))\geq c_0 > 0, \hspace{0.3cm} \forall t\in[-2t_{0},-t_{0}].
\end{equation}

By assumption, the estimate holds for data $\chi^{\rm comp}(1-\chi_b) u_0$, thus, it remains to show it for the data $\chi^{\rm comp} \chi_b u_0$. Therefore, let us denote
$$
u(t) := e^{ith\Delta_D}\psi(-h^{2}\Delta_D)\chi^{\rm comp} \chi_bu_{0}.
$$

Let $\Psi\in C^{\infty}(\mathbb{R})$ such that $\Psi(t)=0$ for $t\leq - 2t_0$,
$\Psi(t)=1$ for $t\geq -t_0$, and $\Psi'\geq0$.
We define 
$$
w(t,x):=\Psi(t)u(t).
$$
Then $w$ satisfies
$$
\begin{cases}
i\partial_{t}w+h\Delta w  =  i\Psi'(t)u(t),\\
w_{|\partial\Omega} =  0,\ w|_{t\leq-2t_0}=0.
\end{cases}
$$
By Duhamel formula, as $w(t) = u(t)$
for $t\geq-t_0$ and as $\Psi'$ is supported in $[-2t_0, -t_0]$,
we have, for $t\geq-t_0$ 
\begin{align*}
u(t,x)&=\int_{-2t_0}^{-t_0}e^{i(t-s)h\Delta_D}i\Psi'(s)e^{ish\Delta_D}\psi(-h^{2}\Delta_D)\chi^{\rm comp} \chi_bu_{0}ds \\
&=\int_{-2t_0}^{-t_0}e^{i(t-s)h\Delta_D}\widetilde \psi(-h^{2}\Delta_D)i\Psi'(s)e^{ish\Delta_D}\psi(-h^{2}\Delta_D)\chi^{\rm comp} \chi_bu_{0}ds,
\end{align*}
where we took $\widetilde \psi \in C^\infty_c(\mathbb R_+^*)$ with $\psi \prec \widetilde \psi \prec \psi_{1}$.
Denote, for $Q\in\{A, 1-A\}$
$$
u_Q(t,x)=\int_{-\infty}^{t}e^{i(t-s)h\Delta_D}\widetilde \psi(-h^{2}\Delta_D)\Psi'(s)Qe^{ish\Delta_D}\psi(-h^{2}\Delta_D)\chi^{\rm comp} \chi_bu_{0}ds,
$$
which is solution of
$$
\begin{cases}
i\partial_{t}u_{Q}+h\Delta u_{Q}=i\Psi'(t) \widetilde \psi(-h^{2}\Delta_D)Q e^{ith\Delta_D}\psi(-h^{2}\Delta_D)\chi^{\rm comp} \chi_bu_{0}, \\
{u_{Q}}_{|\partial\Omega} =  0,\ {u_{Q}}|_{t\leq-2t_0}=0,
\end{cases}
$$
in such a way that 
$$
u = u_A + u_{1-A} \quad \text{ for } t\geq -t_0.
$$

\noindent $\triangleright$ \textbf{Estimation of $u_{1-A}$.}

 We start by estimating $u_{1-A}$.  Let $\widetilde \alpha_0, \widetilde \beta_0$ so that $\operatorname{supp} \widetilde \psi \subset [\widetilde \alpha_0, \widetilde \beta_0]$, and $R >0$ be so that $ \operatorname{supp} \chi_b \cup  \operatorname{supp}(1-A) \subset B(0,R)$, and take 
$$
R_0 := R+4\sqrt{\widetilde \beta_0}(t_1+2t_0),
$$ 
and $\theta \in C^\infty_c(\mathbb R^3)$ so that $\theta = 1$ in  $B(0,R_0)$.

 Using the fact that $\chi_b = \theta \chi_b$, and that $u_{1-A} = \psi_1(-h^2\Delta) u_{1-A}$, we decompose
$$
u_{1-A}(t,x)= e^{i(t-t_1)h\Delta_D} \psi_1(-h^2\Delta_D) \Big((1-\theta) + \chi_b + \theta(1-\chi_b) \Big) u_{1-A}(t_1).
$$
The idea is that $(1-\theta)u_{1-A}(t_1)$ will be small by semiclassical approximate finite speed of propagation (Lemma \ref{lem:semi_speed}), $\chi_b u_{1-A}(t_1)$ will be small by construction of $A$ and propagation of singularities (Proposition \ref{prop:source1}), and $\theta(1-\chi_b)$ verifies the assumptions. 

\emph{The first term: $(1-\theta)u_{1-A}(t_1)$.}
Concerning this first term, we write
\begin{align*}
&(1-\theta)u_{1-A}(t_1) \\
&= \int_{-2t_0}^{-t_0} (1-\theta)e^{i(t_1-s)h\Delta_D}\widetilde \psi(-h^{2}\Delta_D)\Psi'(s)(1-A)e^{ish\Delta_D}\psi(-h^{2}\Delta_D)\chi^{\rm comp} \chi_bu_{0}ds.
\end{align*}
But, by semiclassical approximate finite speed of propagation given by Lemma \ref{lem:semi_speed} applied to $(1-\theta)e^{ish\Delta}\widetilde \psi(-h^{2}\Delta_D)\Psi'(s)(1-A)$ (=$\Psi'(s) (1-\theta)e^{ish\Delta}\widetilde \psi(-h^{2}\Delta_D)(1-A)$), together with the definition of $\theta$ and $R_0$, we get that for any $N$, there is $C_{N}>0$ so that for any $s\in(-2t_0, -t_0)$ and any $f\in L^2$
$$
\Vert (1-\theta)e^{i(t_1-s)h\Delta_D}\widetilde \psi(-h^{2}\Delta_D)\Psi'(s)(1-A)f \Vert_{L^2} \leq C_{N,\sigma}h^N \Vert f \Vert_{L^2}, 
$$
and applying this to the above integral with $f = e^{ish\Delta_D}\psi(-h^{2}\Delta_D)\chi^{\rm comp} \chi_bu_{0}$, we get (using the fact that $e^{ish\Delta_D}$ is an $L^2$--isometry)
\begin{equation}\label{eq:AB1a}
\Vert  (1-\theta)u_{1-A}(t_1)\Vert_{L^2} = O(h^\infty) \Vert u_0 \Vert_{L^2}.
\end{equation}

\emph{The second term: $\chi_b u_{1-A}(t_1)$.} To handle this term, we will use the propagation result with source term Proposition \ref{prop:source1}. We first write (removing the $\widetilde \psi(-h^2\Delta_D)$ with respect to $u_{1-A}$)
$$
\widetilde u_{1-A}(t,x):=\int_{-\infty}^{t}e^{i(t-s)h\Delta_D}\Psi'(s)(1-A)e^{ish\Delta_D}\psi(-h^{2}\Delta_D)\chi^{\rm comp} \chi_bu_{0}ds.
$$
We have, for any finite interval $I$
\begin{equation}\label{eq:tilde_1mA}
\widetilde u_{1-A} = u_{1-A} + O(h^\infty)\Vert u_0 \Vert_{L^2} \quad \text{ in } L^\infty(I, L^2(\Omega)),
\end{equation}
indeed,
\begin{align*}
&(1-\widetilde \psi(-h^2\Delta_D))(1-A)\psi(-h^2\Delta_D) \\ 
&\quad= -(1-\widetilde \psi(-h^2\Delta_D))A\psi(-h^2\Delta_D) = O(h^\infty)_{L^2 \to L^2}
\end{align*}
using the fact that $\psi\prec\widetilde\psi$ for the first identity,
and for the second identity, Lemma~\ref{lem:psi_away}
and symbolic calculus, since $A$ vanishes near the boundary
and $1-\widetilde\psi$ vanishes in a neighborhood of
$\operatorname{supp}\psi$. Now, by Proposition \ref{prop:source1} applied to $\widetilde u_{1-A}$, 
 if $\rho\in {}^b \operatorname{WF}^{t,x}_h(\widetilde u_{1-A})$,
then the broken characteristic starting from $\rho$ 
 must intersect 
$$
{}^b \operatorname{WF}^{t,x}_h((1-A)u)\cap\{t\in[-2t_0,- t_0]\}.\label{eq:wfs}
$$
Because of (\ref{eq:red_bdd_bich__}), there can be no such $\rho$ in ${}^bT^*(\operatorname{supp}\chi_b\times[0,2t_1])$, hence we have
\begin{equation*}
{}^b \operatorname{WF}^{t,x}_h(\chi_b \widetilde u_{1-A}) \cap \{ t \in [0, 2t_1] \} =\emptyset.\label{eq:wfs}
\end{equation*}
It follows from (\ref{eq:tilde_1mA}) and the above that
\begin{equation*}
{}^b \operatorname{WF}^{t,x}_h(\chi_b u_{1-A}) \cap \{ t \in [0, 2t_1] \} =\emptyset.\label{eq:wfs}
\end{equation*}
Therefore, thanks to the wavefront pull-back Lemma \ref{lem:WFpullback},
\begin{equation} \label{eq:PB__}
{}^b \operatorname{WF}^x_h(\chi_b u_{1-A}(t_1)) = \emptyset.
\end{equation}
 (this is now the wavefront set in space only).
Indeed, to see that we can apply Lemma \ref{lem:WFpullback}, take $v(t):=e^{i(t-t_1)h\Delta_D} u_{1-A}(t_1)$,
which agrees with $ u_{1-A}$ for $t>-t_0$.
Since $ u_{1-A}(t_1)=\psi_1(-h^2\Delta_D)u_{1-A}(t_1)$,
we have
$\varphi(hD_t)(\chi_bv)=\chi_bv$ for
$\varphi\in C_c^\infty(\mathbb R)$ equal to one near
$-\operatorname{supp}\psi_1$. Thus $\chi_bv$ satisfies the assumptions of
the pull-back Lemma \ref{lem:WFpullback}, and applying it to $\chi_bv$ at $t = t_1$ gives (\ref{eq:PB__}).
From (\ref{eq:PB__}), we get in particular, from Lemma \ref{ref:ell_space},
\begin{equation}\label{eq:AB1b}
\Vert \psi_1(-h^2\Delta) \chi_b u_{1-A}(t_1)\Vert_{L^2(\Omega)}=O(h^{\infty})\Vert u_0 \Vert_{L^2}.
\end{equation}

 \emph{The third term: $\theta(1-\chi_b)u_{1-A(t_1)}$.}
Finally, for the last term, we have by assumption, as $\theta(1-\chi_b)$ is supported in a compact and away from the boundary (and using Proposition \ref{prop:ft_strichartz} for the bounded intervals outside $[0,\epsilon |\log h|]$)
\begin{align}
&\Vert\chi_0 e^{i(t-t_1)h\Delta_D} \psi_1(-h^2\Delta) \theta(1-\chi_b) u_{1-A}(t_1)  \Vert_{L^{p}(0,\epsilon|\log h|)L^{q}(\Omega)} \nonumber \\
&\qquad= 
\Vert\chi_0 e^{ith\Delta_D} \psi_1(-h^2\Delta) \theta(1-\chi_b) u_{1-A}(t_1)  \Vert_{L^{p}(-t_1,\epsilon|\log h|-t_1)L^{q}(\Omega)} \nonumber \\
&\qquad \leq Ch^{-1/p}\Vert \theta(1-\chi_b) u_{1-A}(t_1)\Vert_{L^{2}} \lesssim 
Ch^{-1/p}\Vert u_0\Vert_{L^{2}}. \label{eq:AB1c}
\end{align}

Gathering (\ref{eq:AB1a}), (\ref{eq:AB1b}), and (\ref{eq:AB1c}) we can conclude for $u_{1-A}$, by noticing that the flow of the error terms produce negligible Strichartz norms for example by using the a-priori estimate with loss given by Proposition \ref{prop:st_logloss}, and we get
\begin{equation} \label{eq:AB1}
\Vert \chi_0 u_{1-A} \Vert_{L^p((0, \epsilon |\log h|), L^q)} \lesssim h^{-1/p}\Vert u_{0} \Vert_{L^2} + O(h^\infty) \Vert u_0\Vert_{L^2} \lesssim h^{-1/p}\Vert u_{0} \Vert_{L^2}.
\end{equation}

\noindent $\triangleright$ \textbf{Estimation of $u_{A}$.}
It remains therefore to estimate $u_A$. We have, by Minkowski inequality
\begin{align*}
&\Vert \chi_0 u_A \Vert_{L^p((0, \epsilon |\log h|), L^q)} \\
&\qquad\leq \int_{-2t_0}^{-t_0} \Vert \chi_0 e^{i(t-s)h\Delta_D}\widetilde\psi(-h^{2}\Delta_D)\Psi'(s)Ae^{ish\Delta_D}\psi(-h^{2}\Delta_D)\chi^{\rm comp} \chi_bu_{0} \Vert_{L^p((0, \epsilon |\log h|), L^q)} ds \\
&\qquad=\int_{-2t_0}^{-t_0} \Vert \chi_0 e^{ith\Delta_D}\widetilde\psi(-h^{2}\Delta_D)\Psi'(s)Ae^{ish\Delta_D}\psi(-h^{2}\Delta_D)\chi^{\rm comp} \chi_bu_{0} \Vert_{L^p((-s, \epsilon |\log h| -s), L^q)} ds \\
&\qquad\leq\int_{-2t_0}^{-t_0} \Vert \chi_0 e^{ith\Delta_D}\widetilde\psi(-h^{2}\Delta_D)\Psi'(s)Ae^{ish\Delta_D}\psi(-h^{2}\Delta_D)\chi^{\rm comp} \chi_bu_{0} \Vert_{L^p((t_0, \epsilon |\log h| +2t_0), L^q)} ds,
\end{align*}
but as above, by using the semiclassical approximate finite speed of propagation given by Lemma \ref{lem:semi_speed}, this writes (using again Proposition \ref{prop:st_logloss} to control the flow associated with error terms in Strichartz norm),
\begin{align*}
&\Vert \chi_0 u_A \Vert_{L^p((0, \epsilon |\log h|), L^q)} \\
&\qquad\leq \int_{-2t_0}^{-t_0} \Vert \chi_0 e^{ith\Delta_D}\widetilde\psi(-h^{2}\Delta_D)\Psi'(s)\theta A e^{ish\Delta_D}\psi(-h^{2}\Delta_D)\chi^{\rm comp} \chi_bu_{0} \Vert_{L^p((t_0, \epsilon |\log h| +2t_0), L^q)} ds \\
&\qquad\qquad+ O(h^\infty) \Vert u_0\Vert_{L^2}.
\end{align*}
But $\theta A$ is supported in a compact and away from the boundary, so we can use the assumption (and again Proposition \ref{prop:ft_strichartz} for the bounded intervals outside $[0,\epsilon |\log h|]$) to get
\begin{align*}
\Vert \chi_0 u_A \Vert_{L^p((0, \epsilon |\log h|), L^q)} &\lesssim \int_{-2t_0}^{-t_0} h^{-1/p} \Vert  \Psi'(s) e^{ish\Delta_D}\psi(-h^{2}\Delta_D)\chi^{\rm comp} \chi_bu_{0} \Vert_{L^2} ds \\
&\qquad + O(h^\infty) \Vert u_0\Vert_{L^2} \\
&\lesssim h^{-1/p}\Vert u_{0} \Vert_{L^2} + O(h^\infty) \Vert u_0\Vert_{L^2},
\end{align*}
where we used the fact that $e^{ish\Delta}$ is an $L^2$--isometry in the second inequality. Combined to (\ref{eq:AB1}), this finishes the proof.

\end{proof}

\subsection{From data in a ball, to all data}
\begin{prop}
\label{prop:reduc_comp}
Let $\psi \in C^\infty_c(\mathbb R_+^*)$.
Assume that there exists $\chi_0 \in C^\infty_c(\mathbb R^3)$ equal to one near $\mathcal R$  so that for any $\chi^{\rm comp} \in C^\infty_c(\mathbb R^3)$ there is $C>0$ so that for any $u_0 \in L^2$
$$
\Vert\chi_0 e^{ith\Delta_{D}}\psi(-h^{2}\Delta_D)\chi^{\rm comp} u_{0}\Vert_{L^{p}(0,\epsilon|\log h|)L^{q}(\Omega)}\leq Ch^{-1/p}\Vert u_{0}\Vert_{L^{2}}.
$$
Then, there is $C>0$ so that for any $u_0 \in L^2$
$$
\Vert\chi_0 e^{ith\Delta_{D}}\psi(-h^{2}\Delta_D)u_{0}\Vert_{L^{p}(0,\epsilon|\log h|)L^{q}(\Omega)}\leq Ch^{-1/p}\Vert u_{0}\Vert_{L^{2}}.
$$
\end{prop}

Before proving Proposition \ref{prop:reduc_comp}, we need the following lemma, which gives a decomposition of the identity in four parts: a compact-in-space part, an outgoing part, an incoming part, and a part which never sees the obstacles. 

\begin{lem}\label{lem:direct_exterior}
Let $\chi \in C^\infty_c(\mathbb R^3)$ and $\psi \in C^\infty_c(\mathbb R_+^*)$.
There exists $\eta \in C^\infty_c(\mathbb R^3)$ and uniformly bounded operators $B_+,B_-,B_0$ on
$L^2(\Omega)$ such that
\begin{equation}\label{eq:direct_partition}
 \psi(-h^2\Delta_D)=\psi(-h^2\Delta_D)\big(\eta^2+B_++B_-+B_0\big),
\end{equation}
with in addition
\begin{equation}\label{eq:direct_energy}
 \forall v \in L^2, \quad \|B_+v\|_{L^2}^2+\|B_-v\|_{L^2}^2\le(1+Ch)\|v\|_{L^2}^2,
\end{equation}
and, $\forall v \in L^2$,
\begin{align}
 & \|\chi e^{ith\Delta_D} \psi(-h^2\Delta_D) B_+ v \|_{L^p((0,\epsilon|\log h|), L^q)} = O(h^\infty)\|v\|_{L^2}, \label{eq:direct_escape_p} \\
 &    \|\chi e^{ith\Delta_D} \psi(-h^2\Delta_D) B_-v \|_{L^p((-\epsilon|\log h|, 0), L^q)}  = O(h^\infty)\|v\|_{L^2},\label{eq:direct_escape_m}\\
 & \|\chi e^{ith\Delta_D} \psi(-h^2\Delta_D) B_0v \|_{L^p((-\epsilon|\log h|, \epsilon|\log h|), L^q)}
 = O(h^\infty)\|v\|_{L^2}.\label{eq:direct_escape_zero}
\end{align}
\end{lem}

\begin{proof}
We will show an approximate decomposition
\begin{equation} \label{eq:dec_out_app}
 \psi(-h^2\Delta_D)I=\psi(-h^2\Delta_D)\big(\eta^2+B_++B_-+B_0\big) + O(h^\infty)_{H^{-s}\to H^{s}},
\end{equation}
where $B_{0,+,-}$ verify (\ref{eq:direct_energy})--(\ref{eq:direct_escape_p})--(\ref{eq:direct_escape_m})--(\ref{eq:direct_escape_zero}) and the decay of the remainder holds in any Sobolev operator norm. Indeed, setting $E:= I - \eta^2-B_+-B_--B_0$, and replacing $B_0$ by $B_0+E$, this proves the Lemma, as $E$ is not contributing in (\ref{eq:direct_escape_zero}) by Sobolev embedding. 

Let $R_0 > 0$ be so that
$$
\operatorname{supp}\chi \cup \Theta_1 \cup \Theta_2 \subset B(0,R_0),
$$
and fix $R_2>R_1>R_0$. Let $\eta \in C^\infty_c(\mathbb R^3)$ and $\zeta \in C^\infty(\mathbb R^3)$ so that 
$$
\eta^2 + \zeta^2 = 1
$$
and $\operatorname{supp}\eta \subset B(0,R_2)$, $\eta = 1$ on $B(0,R_1)$. Observe that, in particular, $\zeta$ is supported in $|x|\geq R_1$. Let $c>0$ be sufficiently small that
$$
\forall |x|\geq R_1, \; \forall \xi\neq0, \; \Big|\frac{x\cdot \xi}{|x||\xi|}\Big| \leq c \implies \forall t \in \mathbb R, \; x+t\xi \notin B(0,R_0).
$$
We let $\lambda \in C^\infty(\mathbb R)$, $|\lambda|\leq1$ be supported in $[ \frac c2, \infty)$ and so that $\lambda = 1$ on $[ c, \infty)$. In addition, let $\widetilde \psi \in C^\infty_c(\mathbb R_+^*)$, $0\leq\widetilde\psi\leq1$, so that $\psi \prec \widetilde \psi$. We let, for $|x|\geq R_1$,
$$
 a_\pm(x,\xi) := \widetilde \psi(|\xi|^2) \lambda\big(\pm\frac{x\cdot \xi}{|x||\xi|}\big), \quad a_0(x,\xi) := \widetilde \psi(|\xi|^2)-a_+(x,\xi) -   a_-(x,\xi),
$$
then smoothly extended to $x$ in the whole space, and define
$$
B_{+,-,0} := \zeta\operatorname{Op}_{h}(a_{+,-,0})\zeta.
$$
As $\eta^2+\zeta^2 = 1$, taking $\widetilde \zeta \in C^\infty(\mathbb R^3)$ so that $\widetilde \zeta = 0$ on $B(0,R_0)$ and $\widetilde \zeta = 1$ near $\operatorname{supp} \zeta$ and applying Lemma \ref{lem:psi_away} we get, in any Sobolev operator norm,
\begin{align*}
&\psi(-h^2\Delta_D)\big(I - \eta^2 - B_+ - B_- - B_0\big) = \psi(-h^2\Delta_D)\zeta \operatorname{Op}_{h}(1-\widetilde \psi(|\xi|^2)) \zeta \\
&\quad= \widetilde \zeta \operatorname{Op}_h(\psi(|\xi|^2))\zeta \operatorname{Op}_{h}(1-\widetilde \psi(|\xi|^2)) \zeta + O(h^\infty),
\end{align*}
and (\ref{eq:dec_out_app}) follows by symbolic calculus since $1-\widetilde\psi$ vanishes in a neighborhood of
$\operatorname{supp}\psi$. $B_+^* B_+ + B_-^*B_-$ has principal symbol 
$$
\zeta^4(a_+^2 + a_-^2) \leq 1,
$$
so that by the semiclassical Garding inequality
$$
\langle (I - B_+^* B_+ - B_-^*B_-) v, v \rangle \geq - Ch\Vert v \Vert^2_{L^2},
$$
and (\ref{eq:direct_energy}) holds. Observe that $a_{+,-,0}$ are in $S^0$ (thanks to $\widetilde\psi$ cutting near $\xi = 0$). Taking $b_+,b_-,b_0 \in S^0$ so that $B_{\pm,0}= \operatorname{Op}_h b_{\pm,0} + O(h^\infty)$ in any Sobolev operator norm with
$$
\operatorname{supp}b_{+,-,0} \subset T^*\operatorname{supp}\zeta \cap \operatorname{supp}a_{+,-,0},
$$
we then have
$$
\forall (x,\xi) \in \operatorname{supp}b_+, \; \forall t \geq 0,\; x+t\xi \not\in B(0,R_0),
$$
$$
\forall (x,\xi) \in \operatorname{supp}b_-, \; \forall t \leq 0,\; x+t\xi \not\in B(0,R_0),
$$
$$
\forall (x,\xi) \in \operatorname{supp}b_0, \; \forall t \in \mathbb R,\; x+t\xi \not\in B(0,R_0),
$$
thus Lemma \ref{lem:away_is_away} gives
 (\ref{eq:direct_escape_p}), (\ref{eq:direct_escape_m}), and (\ref{eq:direct_escape_zero}). 
\end{proof}

\begin{proof}[Proof of Proposition \ref{prop:reduc_comp}]
The overall idea is the following. Starting from
$[0,\epsilon|\log h|]$, we choose a splitting time $\tau$ and
decompose the observed flow into three contributions, up to a
negligible error: the contribution of the data localized by
$\eta^2$ at time $\tau$, which is compactly supported and thus which we know how to estimate, a contribution from the incoming part, which will be active only on the right of $\tau$, and a contribution from the outgoing part, which was active only on the left of $\tau$. We will then iterate this construction on the outgoing and incoming contributions on their respective left and right sub-intervals, choosing the splitting time so that the mass of each remaining data contracts, and show that this process converges.

Let $\eta, B_+, B_-, B_0$ be given by Lemma \ref{lem:direct_exterior} applied with $\chi := \chi_0$.
Recalling that $p>2$, fix a number $\rho$ (this will be our contraction factor) so that
\[
 2^{-1/2}<\rho<2^{-1/p},
\]
and take $h_0$ sufficiently small that 
$$
(1+Ch)/2\le\rho^2,
$$ with
$C$ the constant in (\ref{eq:direct_energy}).

We now explain the splitting procedure from an arbitrary $(I,g)$, where $I = [a,b] \subset [0,\epsilon|\log h|]$ is an interval, and $g\in L^2$. The functions
\[
 E_\pm(s) :=\|B_\pm e^{ish\Delta_D}g\|_{L^2}^2
\]
are continuous and satisfy
$$
E_+(s)+E_-(s)\leq 2\rho^2\|g\|_{L^2}^2.
$$
If they coincide somewhere in $I$, choose such a time $\tau$ and define
\[
 I_L:=[a,\tau],\quad g_L:=e^{-i\tau h\Delta_D}B_+e^{i\tau h\Delta_D}g,
\] 
\[I_R=[\tau,b], \quad
 g_R:=e^{-i\tau h\Delta_D}B_-e^{i\tau h\Delta_D}g.
\]
If $E_+(s)>E_-(s)$ for all $s\in I$, take $\tau=a$, and
\[
 I_L:=\emptyset,\quad I_R=[\tau,b],\qquad
 g_L:=0, \quad
 g_R:=e^{-i\tau h\Delta_D}B_-e^{i\tau h\Delta_D}g.
\]
Finally, if $E_+(s)<E_-(s)$ for all $s\in I$, take $\tau=b$, and
\[
 I_L:=[a,\tau],\quad I_R=\emptyset,\qquad
 g_L:=e^{-i\tau h\Delta_D}B_+e^{i\tau h\Delta_D}g,\quad
 g_R:=0.
\]
In the three cases, we have 
\begin{equation}\label{eq:direct_contraction}
 \|g_L\|_{L^2}\le\rho\|g\|_{L^2},
 \qquad \|g_R\|_{L^2}\le\rho\|g\|_{L^2}.
\end{equation}
Indeed, in the first case above, $ \|g_L\|^2_{L^2} = E_+(\tau) = E_-(\tau) =  \|g_R\|^2_{L^2}$ and the estimates are immediate, whereas in the second case ($\tau = a$), we have $g_L = 0$ and $E_+(s)>E_-(s)$ for all $s\in I$, from which 
$$
\|g_R\|^2_{L^2} = E_-(a) < \frac{E_-(a)+E_+(a)}{2}\leq \rho^2\|g\|_{L^2}^2,
$$
and similarly in the third case ($\tau = b$).
We decompose, according to Lemma~\ref{lem:direct_exterior}
\begin{align*}
\psi(-h^2\Delta_D) g &= e^{-i\tau h\Delta_D}\psi(-h^2\Delta_D)e^{i\tau h\Delta_D}g \\
&=e^{-i\tau h\Delta_D} \psi(-h^2\Delta_D) (\eta^2 + B_+ + B_- + B_0)e^{i\tau h\Delta_D}g, 
\end{align*}
and we obtain, on $I$
\begin{align}
 \chi_0 e^{ith\Delta_D}\psi(-h^2\Delta_D) g
 ={}&\chi_0 e^{i(t-\tau)h\Delta_D}\psi(-h^2\Delta_D) \eta^2e^{i\tau h\Delta_D}g\notag\\
 &+\mathbf1_{I_L}(t)\chi_0 e^{ith\Delta_D}\psi(-h^2\Delta_D) g_L \notag\\
 &+\mathbf1_{I_R}(t)\chi_0 e^{ith\Delta_D}\psi(-h^2\Delta_D) g_R \notag\\
 & +r_I(t),
 \label{eq:direct_recursion}
\end{align}
where
$$
r_I(t) := \chi_0 e^{i(t-\tau)h\Delta_D}\psi(-h^2\Delta_D)(\mathbf 1_{t>\tau}B_+ + \mathbf 1_{t<\tau}B_- + B_0)e^{i\tau h\Delta_D}g,
$$
which verifies, by Lemma~\ref{lem:direct_exterior}, (\ref{eq:direct_escape_p})--(\ref{eq:direct_escape_m})--(\ref{eq:direct_escape_zero}), 
\[
 \|r_I\|_{L^p(I;L^q)}
  = O(h^\infty)\Vert g \Vert_{L^2}.
\]
We let
$$
M_{I,g}(t) := \chi_0 e^{i(t-\tau)h\Delta_D}\psi(-h^2\Delta_D) \eta^2e^{i\tau h\Delta_D}g + r_I(t).
$$
Observe that, by assumption (which we extend to negative times by complex conjugation)
\begin{align*}
 & \|\chi_0 e^{i(t-\tau)h\Delta_D}\psi(-h^2\Delta_D) \eta^2e^{i\tau h\Delta_D}g\|_{L^p(I,L^q)} =  \|\chi_0 e^{ith\Delta_D}\psi(-h^2\Delta_D) \eta^2e^{i\tau h\Delta_D}g\|_{L^p([a-\tau, b-\tau],L^q)} \\
 &\quad \leq  \|\chi_0 e^{ith\Delta_D}\psi(-h^2\Delta_D) \eta^2e^{i\tau h\Delta_D}g\|_{L^p([a-b, b-a],L^q)} \\
 &\quad\leq  \|\chi_0 e^{ith\Delta_D}\psi(-h^2\Delta_D) \eta^2e^{i\tau h\Delta_D}g\|_{L^p([-\epsilon |\log h|, \epsilon |\log h|],L^q)} \lesssim h^{-1/p}\|g\|_{L^2},
\end{align*}
hence
\begin{equation}\label{eq:direct_main}
 \|M_{I,g}\|_{L^p(I,L^q)}\lesssim h^{-1/p}\|g\|_{L^2}.
\end{equation}

Starting with $(I,g) = \big([0,\epsilon|\log h|], u_0\big)$, we iterate the above  construction: we let 
$$
\mathcal I_0 := \Big\{ \big([0,\epsilon|\log h|], u_0\big) \Big\},
$$
and for $j\geq 0$
$$
\mathcal I_{j+1} := \Big\{ (I_R, g_R), \; (I,g) \in \mathcal I_j  \Big\}\cup \Big\{ (I_L, g_L), (I,g) \in \mathcal I_j  \Big\},
$$
(where we take only intervals of non-zero length).
Hence $\mathcal I_j$ denotes the set of intervals-data pairs obtained at generation $j$.
Observe already that, by (\ref{eq:direct_contraction})
\begin{equation} \label{eq:red_comp_contractt}
\forall (I,g) \in \mathcal I_j, \qquad \Vert g \Vert_{L^2} \leq \rho^j\|u_0\|_{L^2}.
\end{equation}
By (\ref{eq:direct_recursion}), we get, at generation $J$
\begin{equation} \label{eq:red_comp_dec}
 \chi_0 e^{ith\Delta_D}\psi(-h^2\Delta_D) u_0 = \sum_{j=0}^{J-1} \sum_{(I,g)\in\mathcal I_j} \mathbf1_{I}(t) M_{I,g} + \underbrace{\sum_{(I,g)\in\mathcal I_J} \mathbf1_{I}(t)\chi_0 e^{ith\Delta_D}\psi(-h^2\Delta_D) g}_{=:R_J}.
\end{equation}

Using the fact that, for a given $j$, the intervals in $\mathcal I_j$ have disjoint interiors, then (\ref{eq:direct_main}), we get 
$$
\left\|\sum_{(I,g)\in\mathcal I_j}\mathbf1_I(t)M_{I,g}(t)
  \right\|_{L^{p}((0,\epsilon|\log h|),L^{q})}^p =
   \sum_{(I,g)\in\mathcal I_j}
   \int_I\|M_{I,g}(t)\|_{L^q}^p\,dt 
   \lesssim h^{-1}\sum_{(I,g)\in\mathcal I_j}\|g\|_{L^2}^p.
$$
But, the number of intervals in $\mathcal I_j$ is at most $2^j$, and their associated data have norm at most
$\Vert g \Vert_{L^2} \leq \rho^j\|u_0\|_{L^2}$ by (\ref{eq:red_comp_contractt}), thus
\begin{equation}\label{eq:direct_generation}
 \left\|\sum_{(I,g)\in\mathcal I_j} \mathbf1_{I}(t) M_{I,g}\right\|_{L^{p}((0,\epsilon|\log h|),L^{q})}
 \lesssim h^{-1/p}2^{j/p}\rho^j\|u_0\|_{L^2}.
\end{equation}
This is summable because $2^{1/p}\rho<1$, thus by the triangle inequality, we obtain that for any $J\geq 0$,
$$
 \left\| \sum_{j=0}^{J-1} \sum_{(I,g)\in\mathcal I_j} \mathbf1_{I}(t) M_{I,g} \right\|_{L^{p}((0,\epsilon|\log h|),L^{q})}
\lesssim h^{-1/p} \|u_0\|_{L^2}.
$$

On the other hand, by Sobolev embedding
\[
 \sup_{t\in\mathbb R}\|\chi_0  e^{ith\Delta_D}\psi(-h^2\Delta_D) g\|_{L^q}\le D_h\|g\|_{L^2},
\]
for some  $D_h > 0$ depending on $h$. Therefore, using disjointness again,
the  remainder after $J$ generations  has norm
at most
\begin{align} \label{eq:direct_generation_rem}
 \|R_J\|_{L^{p}((0,\epsilon|\log h|),L^{q})}^p
 &=\sum_{(I,g)\in\mathcal I_J}
   \int_I
   \|\chi_0 e^{ith\Delta_D}\psi(-h^2\Delta_D)g\|_{L^q}^p\,dt\nonumber\\
 &\leq D_h^p
   \sum_{(I,g)\in\mathcal I_J}|I|\|g\|_{L^2}^p\nonumber\\
 &\leq D_h^p\rho^{Jp}\|u_0\|_{L^2}^p
   \sum_{(I,g)\in\mathcal I_J}|I|\nonumber\\
 &\leq D_h^p \rho^{Jp}\|u_0\|_{L^2}^p \epsilon|\log h|,
\end{align}
where we used (\ref{eq:red_comp_contractt}) again to obtain the third line.

Gathering (\ref{eq:direct_generation_rem}), (\ref{eq:direct_generation}), and (\ref{eq:red_comp_dec}), it follows that, for any $0<h<h_0$, for any $J\geq 0$,
$$
\Vert\chi_0 e^{ith\Delta_{D}}\psi(-h^{2}\Delta_D)u_{0}\Vert_{L^{p}((0,\epsilon|\log h|),L^{q})}\leq Ch^{-1/p}\Vert u_{0}\Vert_{L^{2}} + D_h(\epsilon|\log h|)^{1/p}\rho^J\|u_0\|_{L^2}.
$$
Letting $J\to\infty$ ($h$ being fixed) gives the result.
\end{proof}

We can now conclude the section:

\begin{proof}[Proof of Proposition \ref{prop:mega_red}]
The Proposition follows from Proposition \ref{prop:semilog} together with the combination of Propositions \ref{prop:red_to_ball_away}, \ref{prop:reduc_away_bdd}, and \ref{prop:reduc_comp}.
\end{proof}
}

\section{Approximation of the solution near the trapped set} \label{sec:approx}

\subsection{Phase functions}

Following the works of Iwaka \cite{Ikawa2,IkawaMult} and Burq \cite{plaques},
we define the phase functions and their reflected phases in the following
way. 

We call $\varphi:\mathcal{U}\rightarrow\mathbb{R}$ a phase function
on the open set $\mathcal{U}\subset\mathbb{R}^{3}$ if $\varphi$
is $C^{\infty}$ on $\mathcal{U}$ and verifies $|\nabla\varphi|=1$.
We say that $\varphi$ verifies $(P)$ on $\partial\Theta_{p}$ if
\begin{enumerate}
\item The principal curvatures of the level surfaces of $\varphi$ with
respect to $-\nabla\varphi$ are non-negative in every point of $\mathcal{U}$,
\item We have, for $j\neq p$
\[
\Theta_{j}\subset\big\{y+\tau\nabla\varphi(x) \hspace{0.3cm} \text{s.t.} \hspace{0.3cm} \tau\geq0, \; y\in\mathcal{U}\cap\partial\Theta_{p}, \;\nabla\varphi(y)\cdot n(y)\geq0\big\},
\]
\item For all $A\in\mathbb{R}$, the set $\{\varphi\leq A\}$ is empty or
convex.
\end{enumerate}
Let $\delta_{1}\geq0$ and $\varphi$ be a phase function. We set

\begin{align*}
\Gamma_{p}(\varphi) & =\{x\in\partial\Theta_{p}\ \text{s.t.}\ -n(x)\cdot\nabla\varphi(x)\geq\delta_{1}\},\\
\mathcal{U}_{p}(\varphi) & =\underset{X^{1}(x,\nabla\varphi(x))\in\Gamma_{p}(\varphi)}{\bigcup}\{X^{1}(x,\nabla\varphi(x))+\tau\Xi(x,\nabla\varphi(x)),\ \tau\geq0\}.
\end{align*}
Then, there exists $\delta_{1}\geq0$ such that, if $\varphi$ is
a phase function verifying $(P\text{)}$ on $\partial\Theta_{p}$,
we can define the reflected phase $\varphi_{j}$ on the obstacle $\Theta_{j}$
on the open set $\mathcal{U}_\infty$, verifying $(P)$ on $\partial\Theta_{j}$,
by the following relation, for $X^{1}(x,\nabla\varphi(x))\in\Gamma_{p}(\varphi)$:
\[
\varphi_{j}(X^{1}(x,\nabla\varphi)+\tau\Xi^{1}(x,\nabla\varphi))=\varphi(X^{1}(x,\nabla\varphi))+\tau.
\]
Note in particular the following simple but fondamental fact, simply
differentiating this relation with respect to $\tau$, for all $x$
such that $X^{1}(x,\nabla\varphi(x))\in\Gamma_{p}(\varphi)$:
\[
\nabla\varphi_{j}(X^{1}(x,\nabla\varphi)+\tau\Xi^{1}(x,\nabla\varphi))=\Xi^{1}(x,\nabla\varphi).
\]

We call a finite sequence $J=(j_{1},\cdots,j_{n})$, $j_{i}\in\{1,2\text{\}}$
with $j_{i}\neq j_{i+1}$ a story of reflections, and will denote
$\mathcal{I}$ the set of all the stories of reflection. By induction,
we can define the phases $\varphi_{J}$ for any $J\in\mathcal{I}$,
on the sets $\mathcal{U}_{J}(\varphi)$.

For $f\in C^{\infty}(\mathcal{U})$ and $m\in\mathbb{N}$, let 
\[
|f|_{m}(\mathcal{U})=\max_{0\leq r \leq m}\max_{(a_{i})\in(\mathcal{S}^{2})^{r}}\sup_{\mathcal{U}}|(a_{1}\cdot\nabla)\cdots(a_{r}\cdot\nabla)f|.
\]
We recall the following estimate due to \cite{Ikawa2,IkawaMult,plaques}:
\begin{prop}
\label{prop:contrder}For every $m\geq0$ we have 
\[
|\nabla\varphi_{J}|_{m}\leq C_{m}|\nabla\varphi|_{m}.
\]
\end{prop}
\rev{Moreover, according to \cite[Remark 3.17']{plaques}, starting with a phase verifying $(P)$, its reflection after a sufficiently long story of reflections can always be defined in a fixed neighborhood of the trapped ray:
\begin{prop}
\label{prop:opens}
There exist $M\in\mathbb N$ and an open neighborhood
$\mathcal U_0\subset\mathbb R^3$ of $\mathcal R$ such that
the following holds.
For every $p\in\{1,2\}$, every phase $\varphi$ verifying
$(P)$ on $\partial\Theta_p$, and every
$J=(j_1,\ldots,j_n)\in\mathcal I$ with $j_1\neq p$
and $|J|\geq M$, the reflected phase $\varphi_J$
can be defined on $\mathcal U_0\cap\Omega$.  
\end{prop} }

\rev{
\subsection{The microlocal cut-off}

By Proposition \ref{prop:mega_red}, in order to show Theorem \ref{thm:main}, it suffices to construct
\begin{itemize} 
\item $\psi \in C^\infty_c(\mathbb R_+^*)$,
\item $\chi \in C^\infty_c(\mathbb R^3)$ equal to one near $\mathcal R$,
\item a small neighborhood $\mathcal D \subset {}^b T^* \Omega$ of $\mathcal K$,
\item a small neighborhood $\mathcal V \subset \Omega$ of $\partial \Omega$,
\item $q\in C^\infty_c(T^*\mathbb R^3)$ equal to one in a neighborhood of $(\mathcal D \backslash T^* \mathcal V) \cap \{ |\xi|^2 \in \operatorname{supp}\psi\}$, and so that $d(\pi_x\operatorname{supp}q, \partial\Omega)\geq c > 0$,
\end{itemize}
so that for any $u_0 \in L^2$,
$$
\Vert\chi e^{ith\Delta_{D}}\psi(-h^{2}\Delta_D)\operatorname{Op}_h (q) u_{0}\Vert_{L^{p}(0,\epsilon|\log h|)L^{q}(\Omega)}\leq Ch^{-1/p}\Vert u_{0}\Vert_{L^{2}}.
$$

Let $\psi \in C^\infty_c(\mathbb R_+^*)$, $\psi \neq 0$. We let $\mathcal D$, $\mathcal V$ of and $q \in C^\infty_c(\mathbb R^3_x \times \mathbb R^3_\xi)$, be as above. We allow ourselves to reduce the size of $\mathcal D$ and hence of $\operatorname{supp}q$ so that the above still holds, a finite number of times in the sequel. 

By Lemma \ref{lem:psi_away}, as $\pi_x \operatorname{supp}q$ is away from the boundary, taking  $\chi_1,\chi_2\in C_c^\infty(\mathbb R^3)$ vanishing in
a neighborhood of the obstacles, so that $\chi_1 = 1$ on  $\pi_x \operatorname{supp}q$, and satisfying $\operatorname{dist}\bigl(\operatorname{supp}\chi_1, \operatorname{supp}(1-\chi_2))>0$, replacing $q$ by the symbol $\widetilde q_h \in S^{-\infty}(\mathbb R^3_x \times \mathbb R^3_\xi)$ with the same support localisation, so that, for any $\sigma \geq 0$,
$$
\operatorname{Op}_h(\widetilde q_h) = \chi_2\operatorname{Op}_h\bigl(\psi(|\xi|^2)\bigr)\chi_1\operatorname{Op}_h( q)
 +O(h^\infty)_{H^{-\sigma}(\Omega)\to H^\sigma(\Omega)}, 
$$
it suffices to show (still denoting $q$ instead of $\widetilde q_h$) 
\begin{equation} \label{eq:ult_reduc}
\Vert\chi e^{ith\Delta_{D}}\operatorname{Op}_h (q) u_{0}\Vert_{L^{p}(0,\epsilon|\log h|)L^{q}(\Omega)}\leq Ch^{-1/p}\Vert u_{0}\Vert_{L^{2}}.
\end{equation}

Let $\eta \in C^\infty_c(\mathbb R^3)$, supported away from the boundary, and equal to one on $\pi_x \operatorname{supp}q$. Observe that we have, by non-stationary phase, in any $H^\sigma$
$$
\operatorname{Op}_h(q)u_0(x) = \frac{1}{(2\pi h)^3} \int \int e^{i(x-y)\cdot \xi/h} \eta(y)q(x,\xi)u_0(y) d\xi dy + O(h^\infty)\Vert u_0 \Vert_{L^2},
$$
thus, for any $\sigma \geq 0$
\begin{equation}\label{eq:eta_y1}
e^{ith\Delta_{D}}\operatorname{Op}_h (q) u_{0} = \int e^{ith\Delta_{D}} \delta^y u_0(y)\, dy + O(h^\infty)\Vert u_0 \Vert_{L^2} \quad \text{ in }L^\infty([0, \epsilon|\log h|], H^\sigma),
\end{equation}
where
\begin{equation}\label{eq:eta_y2}
\delta^y : =  \frac{1}{(2\pi h)^3} \int e^{i(x-y)\cdot \xi/h} \eta(y) q(x,\xi) d\xi.
\end{equation}

In order to show (\ref{eq:ult_reduc}), we will construct a parametrix -- that is, an approximate
solution -- in time $0\leq t\leq\epsilon|\log h|$ for the semi-classical
Schr\"{o}dinger equation with data $\delta^{y}$. The first
step will be to construct an approximate solution of the semi-classical
Schr\"{o}dinger equation with data 
\[
e^{i(x-y)\cdot\xi/h}q(x,\xi)
\]
where $\xi\in\pi_\xi\operatorname{supp}q $ and $y \in \operatorname{supp}\eta$
are fixed and considered as parameters. This is the aim of this section.

We let $0<\alpha_0<\beta_0$ so that ${\pi_\xi}\operatorname{supp}q \subset \{ |\xi|^2 \in [\alpha_0, \beta_0]\}$.

\subsection{The spatial neighborhood $\mathcal U_\infty$ where the parametrix will be constructed}

In order to construct our parametrix, starting with the phase
$\varphi(x ; \xi,y)=\frac{(x-y)\cdot\xi}{|\xi|}$, we will need the corresponding  reflected phases.
Our parametrix construction will only hold in a spatial neighbourhood of the trapped ray where all these reflected phases can be defined. This neighbourhood, denoted $\mathcal U_\infty$, is given by the following lemma. Observe that the construction of this Section will then need to be truncated around $\mathcal U_\infty$, see Section \ref{sec:proof}. 
 
\begin{lem}\label{lem:common_phase_domain}
Up to shrinking $\operatorname{supp}q$, there exists a fixed tubular neighborhood
$\mathcal U^\infty \subset \mathbb R^3$ of $\mathcal R$ such that,
for every $\xi\in\pi_\xi\operatorname{supp}q$,
$y\in\operatorname{supp}\eta$, and any $J \in \mathcal I$, the $J$-th reflected phase $\varphi_J$ starting with
\[
\varphi(x ; \xi,y)=\frac{(x-y)\cdot\xi}{|\xi|}
\]
is well defined for $x\in\mathcal U^\infty$.
\end{lem}

\begin{proof}
As
\[
\varphi_J(x,\xi,y)
=\varphi_J(x,\xi,0)-\frac{y\cdot\xi}{|\xi|},
\]
it suffices to consider $y=0$.

By splitting $q$ into two pieces, we can assume that every
forward ray from $\operatorname{supp}q$ first meets the same
obstacle $\Theta_i$.

Then, we first construct the
phase reflected by the first obstacle encountered as in
\cite[Lemma~4.6]{plaques}. This phase satisfies $(P)$, so the
next reflected phases are constructed recursively by
\cite[Proposition~3.2]{plaques}.

Proposition~\ref{prop:opens} then provides a common neighborhood of $\mathcal R$ where all the reflected phases are defined for all
sufficiently long reflection stories. 
For the
finitely many shorter stories, it suffices to shrink the support of $q$ so that they are all well defined in a common
neighborhood of the trapped ray.

Taking a smaller tube contained in all these neighborhoods,
we obtain $\mathcal U^\infty$, on which all the relevant
reflected phases are defined.
\end{proof}

For $J \in \mathcal I$, $\xi\in\operatorname{supp}q$ and $y \in \operatorname{supp}\eta$, we now denote $\varphi_J(x,y,\xi)$ the reflected phase starting from $\varphi(x ; \xi,y)=:\frac{(x-y)\cdot\xi}{|\xi|}$. All these phases are well defined for $x$ in the fixed spatial neighborhood of the trapped ray $\mathcal U_\infty$ by Lemma \ref{lem:common_phase_domain}. We will also denote
$$
\varphi_J(x,\xi):=\varphi_J(x,\xi,0).
$$
}
\subsection{Approximate solution}

We look for the solution in positives times of the equation
\[
\begin{cases}
(i\partial_{t}w+h\Delta w) & =0\ \text{in }\Omega\\
w(t=0)(x) & =e^{i(x-y)\cdot\xi/h}q(x,\xi) \\
w_{|\partial\Omega} & =0
\end{cases}
\]
as the Neumann serie \emph{
\[
w=\sum_{J\in\mathcal{I}}(-1)^{|J|}w^{J}
\]
}where
\[
\begin{cases}
(i\partial_{t}w^{\emptyset}+h\Delta w^{\emptyset}) & =0\ \text{in }\mathbb{R}^{n}\\
w(t=0)(x) & = e^{i(x-y)\cdot\xi/h}q(x,\xi) \\ 
\end{cases}
\]
 and, for $J\neq\emptyset$, $J=(j_{1},\cdots,j_{n})$, $J'=(j_{1},\cdots,j_{n-1})$
\begin{equation}
\begin{cases}
(i\partial_{t}w^{J}+h\Delta w^{J}) & =0\ \text{in }\mathbb{R}^{n}\backslash\Theta_{j_{n}}\\
w(t=0) & =0\\
w_{|\partial\Theta_{j_{n}}}^{J} & =w_{|\partial\Theta_{j_{n}}}^{J'}.
\end{cases}\label{eq:wJ}
\end{equation}
We will look for the $w^{J}$'s as power series in $h.$ In the sake
of conciseness, these series will be considered at a formal level
in this section, and we will introduce their expression as a finite
sum plus a reminder later, in the last section. \rev{In addition, recall that this construction takes place in $\mathcal U_\infty$, where all the phases are well defined. In the last section, multiplication by a spatial cut-off supported in $\mathcal U_\infty$ will give a global object.}

\subsubsection{Free space solution}

We look for $u^{\emptyset}$ as
\begin{align*}
w^{\emptyset} & =\sum_{k\geq0}h^{k}w_{k}^{\emptyset}e^{i((x-y)\cdot\xi-t\xi^{2})/h}\\
w_{0}^{\emptyset}(t & =0)=q(x,\xi)\\
w_{k}^{\emptyset}(t & =0)=0
\end{align*}
solving the transport equations gives immediately
\begin{align*}
w_{0}^{\emptyset} & =q(x-2t\xi,\xi)\\
w_{k}^{\emptyset} & =i\int_{0}^{t}\Delta w_{k-1}(x-2{(t-s)}\xi,s)ds\quad k\geq1
\end{align*}

\subsubsection{Reflected solutions}

We look for $w^{J}$ as
\begin{align*}
w^{J} & =\sum_{k\geq0}h^{k}w_{k}^{J}e^{i(\varphi_{J}(x,\xi, y)|\xi|-t\xi^{2})/h},\\
w_{k}^{J}|_{t\leq0} & =0,\\
w_{k|\partial\Theta_{j_{n}}}^{J} & =w_{k|\partial\Theta_{j_{n}}}^{J'}.
\end{align*}
For $x\in\mathcal{U}_\infty$, we have 
\[
\begin{cases}
(\partial_{t}+2|\xi|\nabla\varphi_{J}\cdot\nabla+|\xi|\Delta\varphi_{J})w_{0}^{J} & =0\\
w_{0|\Theta_{j_{n}}}^{J} & =w_{0|\Theta_{j_{n}}}^{J'}\\
w_{0}^{J}|_{t\leq0} & =0
\end{cases}
\]
and
\[
\begin{cases}
(\partial_{t}+2|\xi|\nabla\varphi_{J}\cdot\nabla+|\xi|\Delta\varphi_{J})w_{k}^{J} & =i\Delta w_{k-1}^{J}\\
w_{k|\Theta_{jn}}^{J} & =w_{k|\Theta_{j_{n}}}^{J'}\\
w_{k}^{J}|_{t\leq0} & =0.
\end{cases}
\]

We can solve the transport equation along the rays:
\begin{lem}
\label{lem:allongrays}Let $\psi\in C^{\infty}(D)$ be such that $|\nabla\psi|=1$.
Then, when $\{x+2\tau|\xi|\nabla\psi(x);\;\tau\in[0,\tau_{0}]\}\subset D$,
the solution of the equation
\[
(\partial_{t}+2|\xi|\nabla\psi\cdot\nabla+|\xi|\Delta\psi)v=f
\]
is represented as
\begin{multline*}
v(x+2\tau|\xi|\nabla\psi(x),t+\tau)=\left(\frac{G\psi(x+2\tau|\xi|\nabla\psi(x))}{G\psi(x)}\right)^{1/2}v(x,t)\\
+\int_{0}^{\tau}\left(\frac{G\psi(x+2\tau|\xi|\nabla\psi(x))}{G\psi(x+2s|\xi|\nabla\psi(x))}\right)^{1/2}f(x+2s|\xi|\nabla\psi(x),t+s)ds.
\end{multline*}
where $G\psi$ denote the gaussian curvature of the level surfaces
of $\psi$.
\end{lem}
\begin{proof}
If we take $w(\tau)=v(x+2\tau|\xi|\nabla\psi,t+\tau)$, $w$ solves
the following ordinary differential equation
\[
\partial_{\tau}w=-|\xi|\Delta\psi(x+2\tau|\xi|\nabla\psi(x))w + f(x+2\tau|\xi|\nabla\psi,t+\tau).
\]
But -- see for example \cite{Luneberg}:
\[
\exp(-|\xi|\int_{\tau_{0}}^{\tau}\Delta\psi(x+2s|\xi|\nabla\psi(x))ds)=\left(\frac{G\psi(x+2\tau|\xi|\nabla\psi(x))}{G\psi(x+2\tau_{0}|\xi|\nabla\psi(x))}\right)^{1/2},
\]
so the formula holds.
\end{proof}
Now, we can find $w_{0}^{J}$. Indeed, as $w_{0|\partial\Theta_{j_{n}}}^{J}=w_{0|\partial\Theta_{j_{n}}}^{J'}$,
because of the previous formula we have
\[
w_{0}^{J}(x+2\tau|\xi|\nabla\varphi_{J}(x),t+\tau)=\left(\frac{G\varphi_{J}(x+2\tau|\xi|\nabla\varphi_{J}(x))}{G\varphi_{J}(x)}\right)^{1/2}w_{0}^{J'}(x,t)
\]
for all $x\in\Gamma_{j_{n}}$. So, for all $x\in\mathcal{U}_\infty$
we get
\begin{multline*}
w_{0}^{J}(x,t)=\left(\frac{G\varphi_{J}(x)}{G\varphi_{J}(X^{-1}(x,|\xi|\nabla\varphi_{J}))}\right)^{1/2}\\
\times w_{0}^{J'}\left(X^{-1}(x,|\xi|\nabla\varphi_{J}),t-\frac{\varphi_{J}(x)-\varphi_{J}(X^{-1}(x,|\xi|\nabla\varphi_{J}))}{2|\xi|}\right).
\end{multline*}

Iterating the process and doing the same for $k\geq1$ we get the
following expressions of $w_{k}^{J}$ for $x\in\mathcal{U}_\infty$:
\begin{prop}
\label{prop:solref}We denote by $\widehat{X}_{-2t}(x,|\xi|\nabla\varphi_{J})$
the backward spatial component of the flow starting from $(x,|\xi|\nabla\varphi_{J})$,
defined in the same way as $X_{-2t}(x,|\xi|\nabla\varphi_{J})$, at
the difference that we ignore the first obstacle encountered if it's
not $\Theta_{j_{n}},$ and we ignore the obstacles after $|J|$ reflections.
Moreover, for $J=(j_{1}=i_{1},\dots,j_{n})\in\mathcal{I}_{i_{1}}$,
denote by
\[
J(x,t,\xi)=\begin{cases}
(j_{1},\cdots,j_{k}) & \text{if }\widehat{X}_{-2t}(x,|\xi|\nabla\varphi_{J})\text{ has been reflected \ensuremath{n-k} times,}\\
\emptyset & \text{if }\widehat{X}_{-2t}(x,|\xi|\nabla\varphi_{J})\text{ has been reflected \ensuremath{n} times}.
\end{cases}
\]
 Then, the $w_{k}^{J}$'s are given by, for $t\geq0$ and $x\in\mathcal{U}_\infty$,
\[
w_{0}^{J}(x,t)=\Lambda\varphi_{J}(x,\xi)q(\widehat{X}_{-2t}(x,|\xi|\nabla\varphi_{J}),\xi)
\]
where $\Lambda \varphi_{\emptyset} := 1$, and, for $|J|\geq 1$,
\[
\Lambda\varphi_{J}(x,\xi):=\left(\frac{G\varphi_{J}(x)}{G\varphi_{J}(X^{-1}(x,|\xi|\nabla\varphi_{J}))}\right)^{1/2}\times\cdots\times\left(\frac{G\varphi_{(j_1)}(X^{-(|J|-1)}(x,|\xi|\nabla\varphi_{J}))}{G\varphi_{(j_1)}(X^{-|J|}(x,|\xi|\nabla\varphi_{J}))}\right)^{1/2},
\]
and, for $k\geq1$, and $x\in\mathcal{U}_\infty$
\[
w_{k}^{J}(x,t)=i\int_{0}^{t}g_{\varphi_{J}}(x,t-s,\xi)\Delta w_{k-1}^{J(x,\xi,t-s)}(\widehat{X}_{-2(t-s)}(x,|\xi|\nabla\varphi_{J}),s)ds
\]
where

\[
g_{\varphi_{J}}(x,\xi,t)=\left(\frac{G\varphi_{J}(x)}{G\varphi_{J}(X^{-1}(x,|\xi|\nabla\varphi_{J}))}\right)^{1/2}\times\cdots\times\left(\frac{G\varphi_{J(x,t,\xi)}(X^{-|J(x,t,\xi)|-1}(x,|\xi|\nabla\varphi_{J}))}{G\varphi_{J(x,t,\xi)}(\widehat{X}_{-2t}(x,|\xi|\nabla\varphi_{J}))}\right)^{1/2}.
\]
\end{prop}
We now study the support properties of the $w_{k}^{J}$'s.
\begin{lem}
\label{lem:supprought}For $x\in\mathcal{U}_\infty$,
\begin{equation}
w_{k}^{J}(x,t)\neq0\implies(\widehat{X}_{-2t}(x,|\xi|\nabla\varphi_{J}),\xi)\in\operatorname{supp}q.\label{eq:UinfJ}
\end{equation}
And moreover, still for $x\in\mathcal{U}_\infty$,
\begin{equation}
\operatorname{supp}w_{k}^{J}(x,\cdot,\cdot)\subset\left\{ J(x,\xi,t)=\emptyset\right\} .\label{eq:suppJ}
\end{equation}
\end{lem}
\begin{proof}
We prove both properties by induction. The first, (\ref{eq:UinfJ}),
is obviously true for $k=0$. Now, suppose that it is true at the
rank $k-1$. If $w_{k}^{J}(x,t)\neq0$, there exists $s\in[0,t]$
such that $w_{k-1}^{J(x,\xi,t-s)}(\widehat{X}_{-2(t-s)}(x,|\xi|\nabla\varphi_{J}),s)\neq0$.
By the induction hypothesis, we deduce that
\[
\left(\widehat{X}_{-2s}(\widehat{X}_{-2(t-s)}(x,|\xi|\nabla\varphi_{J}),|\xi|\nabla\varphi_{J(x,\xi,t-s)}),\xi\right)\in\operatorname{supp}q.
\]
But $\widehat{X}_{-2s}(\widehat{X}_{-2(t-s)}(x,|\xi|\nabla\varphi_{J}),|\xi|\nabla\varphi_{J(x,\xi,t-s)})=\widehat{X}_{-2t}(x,|\xi|\nabla\varphi_{J})$,
thus (\ref{eq:UinfJ}) is shown for all $k\geq0$.

Let us now prove (\ref{eq:suppJ}) . For $J\in\mathcal{I}_{i_{0}}$,
$w_{k}^{J}=0$ thus it suffices to consider $J\in\mathcal{I}_{i_{1}}$.
We prove again the formula by induction on $k$.

We have , with $X^{-|J|}:=X^{-|J|}(x,\nabla\varphi_{J}(x)|\xi|)$:
\[
w_{0}^{J}=\Lambda\varphi_{J}(x,\xi)q\left(X^{-|J|}-\left(t-(\varphi_{J}(x)-\varphi_{J}(X^{-|J|})\right)\xi,\xi\right).
\]
If $J(x,\xi,t)\neq\emptyset$, then $t-(\varphi_{J}(x)-\varphi_{J}(X^{-|J|})\leq0$.
Because $X^{-|J|}\in\partial\Theta_{i_{1}}$, $\xi$ points toward
$\Theta_{i_{1}}$, and $q$ is spatially supported in a neighborhood
of $\mathcal{R}$ and away of the boundary, we deduce that $X^{-|J|}-\left(t-(\varphi_{J}(x)-\varphi_{J}(X^{-|J|})\right)\xi$
does not belong to the spatial support of $q$, thus $w_{0}^{J}(x,t)=0$
and the formula holds for $k=0$.

Now, let $k\geq1$ and suppose that (\ref{eq:suppJ}) is true for
all $J\in\mathcal{I}$ for $w_{k-1}$. Suppose that $w_{k}^{J}(x,\xi,t)\neq0$.
Then, there exists $s\in[0,t]$ such that
\[
w_{k-1}^{J(x,\xi,t-s)}(\widehat{X}_{-2(t-s)}(x,|\xi|\nabla\varphi_{J}),s)\neq0.
\]
We denote $J=(j_{1},\dots,j_{n})$ and $J(x,\xi,t-s)=(j_{1},\dots,j_{k})$.
By the induction hypothesis, $J(\widehat{X}_{-2(t-s)}(x,|\xi|\nabla\varphi_{J}),\xi,t-s)=\emptyset$.
That means that for $t-s\leq t'\leq t$, $\widehat{X}_{-2t'}(x,|\xi|\nabla\varphi_{J})$
is reflected $k$ times. But note that by definition of $J(x,\xi,t-s)$,
for $0\leq t'\leq t-s$,$\widehat{X}_{-2t'}(x,|\xi|\nabla\varphi_{J})$
is reflected $n-k$ times; so, for $0\leq t'\leq t$, $\widehat{X}_{-2t'}(x,|\xi|\nabla\varphi_{J})$
is reflected $n$ times, therefore $J(x,\xi,t)=\emptyset$ and (\ref{eq:suppJ})
holds.
\end{proof}

\begin{lem}
\label{lem:support}There exists $c_{1},c_{2}>0$ such that for every
$J\in\mathcal{I}$, for $x\in \mathcal U^\infty$, the support of 
$w_{k}^{J}(x, \cdot)$ is included in $\left\{ c_{1}|J|\leq t\leq c_{2}(|J|+1)\right\} $.
\end{lem}
\begin{proof}
Recall that $q$ is supported away from $\partial\left(\Theta_{1}\cup\Theta_{2}\right)$.
Let 
\[
\delta_{0}=d(\pi_x \operatorname{supp}q,\partial\left(\Theta_{1}\cup\Theta_{2}\right))>0.
\]
In order to be reflected (backward) $|J|$ times and then be in
$\pi_x\operatorname{supp}q$, a ray with speed in $[\alpha_{0},\beta_{0}]$ needs a time a
least $\frac{d}{2\beta_{0}}(|J|-1)+\frac{\delta_{0}}{2\beta_{0}}$.
Moreover, a ray starting in $\mathcal U^\infty$ needs a time
a most $2\frac{D}{\alpha_{0}}(|J|+1)$ to be reflected $|J|$ times,
where $D$ is the maximal distance between the obstacles, and then
it leaves the support of $q$ in a time at most $2\frac{d}{\alpha_{0}}$.
Therefore, by \lemref{supprought}, the time support of $w_{k}^{J}(x, \cdot)$
is included in 
\[
\frac{d}{2\beta_{0}}(|J|-1)+\frac{\delta_{0}}{2\beta_{0}}\leq t\leq2\frac{D}{\alpha_{0}}(|J|+1)+2\frac{d}{\alpha_{0}},
\]
and the lemma holds with (say) $c_{1}=\frac{\delta_{0}}{2\beta_{0}}$
and $c_{2}=4\frac{D}{\alpha_{0}}$.
\end{proof}

\subsection{The $\xi$ derivatives}

We now investigate the derivatives of the phases with respect
to $\xi$. 

\begin{lem}
\label{lem:nondeg}Let $J\in\mathcal{I}$ and $\mathcal{S}_{J}(x,y,t,\xi):=\varphi_{J}(x,y,\xi)|\xi|-t\xi^{2}$. Then \rev{
\begin{enumerate}
\item 
For all $t>0$, there exists at most one $s_{J}(x,y,t)$ such that
$\nabla_{\xi}\mathcal{S}_{J}(x,t,s_{J}(x,t))=0$. If it exists, such a critical point $\xi=s_{J}(x,y,t)$ is the vector allowing reaching $x$ from
 $y$ in time $2t$ by following (exhaustively) the story of reflexions J.
 In addition, denoting
\begin{align*}
 l_{J}(x,\xi)= &d(x,X^{-1}(x,\nabla\varphi_{J}(x,\xi))) \\
               &+d(X^{-1}(x,\nabla\varphi_{J}(x,\xi)),X^{-2}(x,\nabla\varphi_{J}(x,\xi))) \\
               &+\cdots +d(X^{-|J|-1}(x,\nabla\varphi_{J}(x,\xi)),X^{-|J|}(x,\nabla\varphi_{J}(x,\xi))),
\end{align*}
we have
\begin{equation}\label{eq:DSJ}
\forall t \geq \frac{\ell_J(x, \xi)}{2|\xi|}, \hspace{0.3cm} \nabla_{\xi}\mathcal{S}_{J}(x,\xi,t)=\widehat{X}_{-2t}(x,|\xi|\nabla\varphi_{J}(x,\xi))-y.
\end{equation}

\item The $\xi$-Hessian of the phase writes
\begin{equation}\label{eq:detpo}
D_{\xi}^{2}\mathcal{S}_{J}(x,\xi,t)= - \int_0^{2t} \mathcal L_{J^{(s)}} \big(X_{-s}(x,|\xi|\nabla\varphi_{J}(x,\xi)), \xi \big) \; ds,
\end{equation}
where we denoted $J^{(s)}=J(x,\xi,s/2)$ with the notations of Proposition
\propref{solref}, and
\begin{equation}\label{eq:defLJ}
\big(\mathcal L_J (x, \xi)\big)_{i,j} :=  \partial_{\xi_i}\nabla\psi_{J}(x,\xi)\cdot \partial_{\xi_j} \nabla \psi_{J}(x,\xi).
\end{equation}
\end{enumerate} }
\end{lem}
\begin{proof}
For $J=\emptyset$, an explicit computation gives the result: indeed, observe in particular that $\mathcal L_\emptyset = \operatorname{I}_3$.
Thus we are reduced
to the case $|J|\geq1$. Let $J=(j_{1},\cdots,j_{n})$.

\textbf{The Hessian.} We first show
(\ref{eq:detpo}). Let $\psi_{J}(x,\xi)=\varphi_{J}(x,\xi)|\xi|$.
Because $|\nabla\psi_{J}|^{2}=|\xi|^{2}$, $\partial_{\xi_{i}\xi_{j}}^{2}\psi_{J}$
verifies the transport equation with source term
\[
\nabla\left(\partial_{\xi_{i}\xi_{j}}^{2}\psi_{J}\right)\cdot\nabla\psi_{J}=\delta_{ij}-\partial_{\xi_{i}}\nabla\psi_{J}\cdot\partial_{\xi_{j}}\nabla\psi_{J}.
\]
For $x\in\Gamma_{j_{,n}}$, we have $\nabla\psi_{J}(x+\tau\nabla\psi_{J}(x,\xi))=\nabla\psi_{J}(x,\xi)$
and therefore, $g(s)=\partial_{\xi_{i}\xi_{j}}^{2}\psi_{J}(x+s\nabla\psi_{J}(x,\xi),\xi)$
verifies
\[
g'(s)=\delta_{ij}-\partial_{\xi_{i}}\nabla\psi_{J}(x+s\nabla\psi_{J}(x,\xi),\xi)\cdot\partial_{\xi_{j}}\nabla\psi_{J}(x+s\nabla\psi_{J}(x,\xi),\xi)
\]
and thus
\begin{multline*}
\partial_{\xi_{i}\xi_{j}}^{2}\psi_{J}(x+\tau\nabla\psi_{J})=\partial_{\xi_{i}\xi_{j}}^{2}\psi_{J}(x)+\tau\delta_{ij}\\
-\int_{0}^{\tau}\partial_{\xi_{i}}\nabla\psi_{J}(x+s\nabla\psi_{J}(x,\xi),\xi)\cdot\partial_{\xi_{j}}\nabla\psi_{J}(x+s\nabla\psi_{J}(x,\xi),\xi)ds.
\end{multline*}
Noting that $\partial_{\xi_{i}\xi_{j}}^{2}\psi_{J}=\partial_{\xi_{i}\xi_{j}}^{2}\psi_{J'}$
on $\Gamma_{j_{n}}$ and iterating this argument down to $\partial_{\xi_{i}\xi_{j}}^{2}\psi_{\emptyset}=0$
on $\Gamma_{j_{1}}$ we get, for all $x\in\mathcal{U}_\infty$
\begin{multline*}
\partial_{\xi_{i}\xi_{j}}^{2}\psi_{J}(x,\xi)=\frac{l}{|\xi|}\delta_{ij}\\
-\int_{0}^{\frac{l}{|\xi|}}\partial_{\xi_{i}}\nabla\psi_{J^{(s)}}(X_{-s}(x,|\xi|\nabla\varphi_{J}(x,\xi)),\xi)\cdot\partial_{\xi_{j}}\nabla\psi_{J^{(s)}}(X_{-s}(x,|\xi|\nabla\varphi_{J}(x,\xi)),\xi)ds
\end{multline*}
where we denoted $J^{(s)}=J(x,\xi,s/2)$ with the notations of Proposition
\propref{solref} and $l=l_{J}(x,\xi)$ given by
\begin{multline*}
l_{J}(x,\xi)=d(x,X^{-1}(x,\nabla\varphi_{J}(x,\xi)))+d(X^{-1}(x,\nabla\varphi_{J}(x,\xi)),X^{-2}(x,\nabla\varphi_{J}(x,\xi)))\\
+\cdots+d(X^{-|J|-1}(x,\nabla\varphi_{J}(x,\xi)),X^{-|J|}(x,\nabla\varphi_{J}(x,\xi))).
\end{multline*}

Thus we have
$$
D_{\xi}^{2}\mathcal{S}_{J}(x,\xi,t)=(\frac{l}{|\xi|}-2t)\operatorname{I}_3
-\int_{0}^{\frac{l}{|\xi|}} \mathcal L_{J^{(s)}} \big(X_{-s}(x,|\xi|\nabla\varphi_{J}(x,\xi)), \xi \big) \; ds,
$$
and the fact that $\mathcal L_\emptyset = \operatorname{I}_3$ gives the result.

\textbf{Critical points. }We now show the first part of the statement.
Differentiating $|\nabla\psi_{J}|^{2}=|\xi|^{2}$ we get in the same
way than before, for $x\in\Gamma_{j_{n}}$ 
\[
\partial_{\xi_{i}}(\varphi_{J}|\xi|)(x+\tau\nabla\varphi_{J}(x,\xi),\xi)=\partial_{\xi_{i}}(\varphi_{J'}|\xi|)(x,\xi)+\tau\frac{\xi_{i}}{|\xi|},
\]
and, iterating this argument up to $D_{\xi}(\varphi_{\emptyset}|\xi|)(x,\xi)=x-y$:
\[
D_{\xi}(\varphi_{J}|\xi|)(x,\xi)=X^{-|J|}(x,\nabla\varphi_{J}(x,\xi))-y+l_{J}(x,\xi)\frac{\xi}{|\xi|}.
\]
Therefore
\[
D_{\xi}\mathcal{S}_{J}(x,\xi,t)=X^{-|J|}(x,\nabla\varphi_{J}(x,\xi))-(2t-\frac{l_{J}(x,\xi)}{|\xi|})\xi-y,
\]
that is, if $w_{J}(x,t)\neq0$, by \lemref{supprought}
\[
D_{\xi}\mathcal{S}_{J}(x,\xi,t)=\widehat{X}_{-2t}(x,|\xi|\nabla\varphi_{J}(x,\xi))-y.
\]
Therefore, if $D_{\xi}\mathcal{S}_{J}(x,\xi,t)=0$ then
\[
\widehat{X}_{2t}(y,\xi)=x.
\]
and thus $\xi=s_{J}(x,t)$ is the vector allowing reaching $x$ from
the vector $y$ in time $2t$ beginning with a reflection on $\Theta_{j_{1}}$.
\end{proof}

In addition, we will need to control the directional derivatives of $w_{k}^{J}$.
To this purpose we show 
\begin{prop}
\label{prop:decder}For all multi-indices $\alpha,\beta$ there exists
a constant $D_{\alpha,\beta}>0$ such that the following estimate
holds on $\mathcal{U}_{\infty}$:
\[
|D_{\xi}^{\alpha}D_{x}^{\beta}\nabla\varphi_{J}|\leq D_{\alpha,\beta}^{|J|}.
\]
\end{prop}
\begin{proof}
Remark that $|D_{x}^{\beta}\varphi_{J}|\leq C$ for all multi-indices
$\beta$. Thus, we will show the estimate by induction on the size
of $\alpha$. Let $\alpha$ be a multi-indice. We assume that $|D_{\xi}^{\gamma}D_{x}^{\eta}\varphi_{J}|\leq C^{|J|}$
for $|\gamma|\leq|\alpha|-1$ and all multi-indices $\eta$. Differentiating
$|\nabla\varphi_{J}|^{2}=1$ with respect to $D_{\xi}^{\alpha}$ ,
we get an equation of the form
\[
\nabla(D_{\xi}^{\alpha}\varphi_{J})\cdot\nabla\varphi_{J}=R_{J}
\]
where
\[
R_{J}=\sum_{|\gamma|,|\gamma'|\leq|\alpha|-1}a_{\gamma,\gamma'}D_{\xi}^{\gamma}\nabla\varphi_{J}\cdot D_{\xi}^{\gamma'}\nabla\varphi_{J}
\]
 with $a_{\gamma,\gamma'}\in\mathbb{Z}$. Let $x\in\Gamma_{j_{n}}$.
We denote $g(s)=D^{\alpha}\varphi_{J}(x+s\nabla\varphi_{J}(x,\xi),\xi)$.
Because $\nabla\varphi_{J}(x+s\nabla\varphi_{J}(x,\xi),\xi)=\nabla\varphi_{J}(x,\xi)$,
$g$ verifies $g'(s)=R(x+s\nabla\varphi_{J})$. Therefore we have
\[
D_{\xi}^{\alpha}\varphi_{J}(x+\tau\nabla\varphi_{J}(x,\xi),\xi)=D_{\xi}^{\alpha}\varphi_{J}(x,\xi)+\int_{0}^{\tau}R_{J}(x+s\nabla\varphi_{J}(x,\xi),\xi)ds.
\]
But $D_{\xi}^{\alpha}\varphi_{J}=D_{\xi}^{\alpha}\varphi_{J'}$ on
$\Gamma_{j_{n}}$, so iterating this process we get, for all $x\in\mathcal{U}_{\infty}$
\begin{gather*}
D_{\xi}^{\alpha}\varphi_{J}(x,\xi)=D_{\xi}^{\alpha}\varphi(X^{-|J|}(x,\nabla\varphi_{J}(x,\xi)),\xi,\xi)\\
-\int_{0}^{d_{1}}R_{J}(X_{-s}(x,\nabla\varphi_{J}(x,\xi)),\xi)ds-\int_{d_{1}}^{d_{2}}R_{J^{'}}(X_{-s}(x,\nabla\varphi_{J}(x,\xi)),\xi)ds\\
-...-\int_{d_{|J|-1}}^{d_{|J|}}R_{(j_{1})}(X_{-s}(x,\nabla\varphi_{J}(x,\xi)),\xi)ds.
\end{gather*}
with
\[
d_{1}=d(x,X^{-1}(x,\nabla\varphi_{J}(x,\xi)),d_{i}=d(X^{-i}(x,\nabla\varphi_{J}(x,\xi),X^{-i-1}(x,\nabla\varphi_{J}(x,\xi)).
\]
Now, we differentiate this identity with respect to $D_{x}^{\beta}$,
$\beta$ been arbitrary. Remark that
\begin{itemize}
\item The derivatives of $\nabla\varphi_{J}$ with respect to $x$ are bounded
uniformly with respect to $J$ by Proposition \propref{contrder},
\item for $x\in\mathcal{U}_{\infty}$ and $\xi\in\pi\operatorname{supp}q$, $(x,\nabla\varphi_{J})$
is in a compact set away from $W_{\text{tan}}$ introduced in \lemref{hold12}.
Therefore, by \lemref{hold12}, $X^{-1}$ is $C^{\infty}$ on this
compact set so his derivatives will be bounded by a constant $C$,
and theses of $X^{-i}=X^{-1}\circ\cdots\circ X^{-1}$ by $C^{i}$,
and all of them by $C^{|J|}$,
\item $d(x,X^{-1}(x,\nabla\varphi_{J}(x,\xi))$ is nothing but $t(x,\nabla\varphi_{J}(x,\xi))$
where $t$ is the function introduced \lemref{hold12}, which is $C^{\infty}$
away from $W_{\text{tan}}$,
\item $X_{-s}$ consists to follow the straight line between $X^{-i}$ and
$X^{-i-1}$ for $d_{i}\leq s\leq d_{i-1}$ (with the convention $d_{0}=0$)
\item $(x,y)\in\Theta_{1}\cup\Theta_{2}\rightarrow d(x,y)$ is $C^{\infty}$,
\item By the induction hypothesis the derivatives with respect to $x$ of
$D_{\xi}^{\gamma}\nabla\varphi_{(j_{1},\cdots,j_{k})}$ for $|\gamma|\leq|\alpha|-1$
are bounded by $C^{|J|}$.
\end{itemize}
So, the left hand side will be bounded by
\[
|D_{\xi}^{\alpha}D_{x}^{\beta}\varphi_{J}|\leq C+|J|C^{6|J|}\lesssim D^{|J|}
\]
with (say) $D=C^{7}$, and the lemma holds.
\end{proof}
Thus we have
\begin{cor}
\label{cor:boundsdirect}The following bounds hold on $\mathcal{U}_{\infty}$
\[
|D_{\xi}^{\alpha}w_{k}^{J}|\lesssim C_{\alpha}^{|J|}. 
\]
\end{cor}
\begin{proof}
We have
\[
\frac{G\varphi_{J}(x)}{G\varphi_{J}(X^{-1}(x,\nabla\varphi_{J}(x,\xi))}=\frac{1}{1+lH(x,\xi)+l^{2}G(x,\xi)}
\]
where $l=d(x,X^{-1}(x,\nabla\varphi_{J}(x,\xi))$, $H$ is the trace
and $G$ the determinant of the second fundamental form of the level
surface of $\varphi_{J}$, that is
\[
D_{x}^{2}\varphi_{J}(X^{-1}(x,\nabla\varphi_{J}(x,\xi)).
\]
Therefore, differentiating the explicit expressions of Proposition
\propref{solref}, together with the estimates of Proposition \propref{decder},
and Proposition \ref{prop:contrder}, 
gives the
result. 
\end{proof}
\subsection{Decay of the reflected solutions}

The principal result allowing to estimate the decay of the reflected
solutions is the convergence of the product of the Gaussian curvatures
$\Lambda\varphi_{J}$ obtained by \cite{Ikawa2,IkawaMult} and \cite{plaques}.
In the present framework of two obstacles, it writes:
\begin{prop}
\label{prop:convL}Let $0<\lambda<1$ be the product of the two eigenvalues
lesser than one of the Poincar\'{e} map associated with the periodic trajectory.
Then, there exists $0<\alpha<1$, and for $I=(1,2)$ and $I=(2,1)$,
for every $l\in\{\{1\},\{2\},\emptyset\}$, there exists a $C^{\infty}$
function $a_{I,l}$ defined in $\mathcal{U}_{\infty}$, such that,
for all $J=(\underset{r\text{ times}}{\underbrace{I,\dots,I}},l)$,
we have
\[
\underset{\mathcal{U}_{\infty}}{\sup}|\Lambda\varphi_{J}-\lambda^{r}a_{I,l}|_{m}\leq C_{m}\lambda^{r}\alpha^{|J|}.
\]
\end{prop}
We deduce the following bounds:
\begin{prop}
\label{prop:essbouds}The following bounds hold on $\mathcal{U}_{\infty}$:
\[
|w_{k}^{J}|_{m}\leq C_{k}\lambda^{\rev{\frac{|J|}{2}}}. 
\]
\end{prop}
\begin{proof}
Making use of the explicit expressions for the $w_{k}^{J}$ given
by Proposition \ref{prop:solref}, the result follows by Proposition \propref{convL},
using the remarks made in the proof of Corollary \corref{boundsdirect}.

\end{proof}

In addition, we can also control the time derivatives:

\rev{\begin{cor}
\label{cor:boundsdirect_t}The following bounds hold on $\mathcal{U}_{\infty}$
\[
|\partial_{t}^{\alpha} \partial_{x}^{\beta} w_{k}^{J}|\lesssim C_{\alpha, \beta, k} \lambda^{{\frac{|J|}{2}}}.
\]
\end{cor}
\begin{proof}
We write, from the definition of $w_k^J$ (recalling that $\Delta w_{-1}^J := 0$) 
$$
\partial_t w_k^J = - 2|\xi|\nabla \varphi_J \cdot \nabla w_k^J - |\xi|\Delta \varphi_J w_k^J -i \Delta w_{k-1}^J.
$$
Differentiating the above in time and space, we therefore get the result by a double recurrence thanks to Proposition \ref{prop:contrder}, with the base case ($\alpha = 0$) given by Proposition \ref{prop:essbouds}.
\end{proof} }

\rev{
\subsection{The retarded phases}

Finally, we study the critical point of the following retarded phases, which will arise in the duality argument needed to conclude.
\begin{lem} \label{lem:ret_phases}Let 
$$
\Psi_{J_1, J_2}(x,z,y,\xi, \xi', t,s) :=  {\mathcal S}_{J_1}(x, y,\xi, t) - {\mathcal S}_{J_2}(z, y, \xi', s). 
$$
Then:
\begin{enumerate}
\item 
\begin{enumerate}
\item  For any fixed $(x, z, t, s)$ with $s\neq t$, there exists at most one triple $(y, \xi, \xi')$ so that $\nabla_{(y,\xi, \xi')}\Psi_{J_1, J_2}(x,z,y,\xi, \xi', t,s) = 0$. 
\item If $(y, \xi, \xi')$ is such a critical point, then $\xi = \xi'$ is the direction allowing to reach both $x$ from $y$ following (exhaustively) the story $J_1$, and $z$ from $y$ following (exhaustively) the story $J_2$. 
\item  In particular, $(x, z, t, s)$ being fixed, there exists at most one couple $(J_1, J_2)$ so that $\Psi_{J_1, J_2}$ has a critical point in $(y,\xi, \xi')$.
\end{enumerate}
\item There exists $c>0$, depending only of the geometry, so that when evaluated in  a critical point, the $(y,\xi, \xi')$--Hessian of $\Psi_{J_1, J_2}$ verifies
$$
\big| \det D^2_{(y,\xi, \xi')} \Psi_{J_1, J_2} \big| \geq \min(|t-s|, 1)^3 \lambda^{2\min({|J_1|}, {|J_2|})},
$$
where $0<\lambda<1$ was introduced in Proposition \propref{convL}.
\end{enumerate}

\end{lem}
\begin{proof}
By definition of the phases, $\varphi_J(x,y,\xi) = \varphi_J(x,0,\xi) - y \cdot \frac{\xi}{|\xi|}$,
and hence
$$
(\nabla_y {\mathcal S}_{J})(x, y, \xi, t) = - \xi.
$$
It follows that $\nabla_y\Psi_{J_1, J_2}(x,z,y,\xi, \xi', t,s) = \xi' - \xi$, and therefore, from Lemma \ref{lem:nondeg}, (1),
$$
\nabla_{(\xi, \xi',y)}\Psi_{J_1, J_2}(x,z,y,\xi, \xi', t,s) = 0
\iff
\begin{cases}
\xi = s_{J_1}(x,y,t), \\
\xi' = s_{J_2}(z,y,s), \\
\xi = \xi'.
\end{cases}
$$
Observe that for $(x, z, t, s)$ fixed with $s \neq t$, this system has at most one solution. Indeed, as $\xi = \xi'$ is uniquely determined by $x, z, t, s$ and $y$, it suffices to show that there is at most one $y$ so that 
$$
s_{J_1}(x,y,t) = s_{J_2}(z,y,s).
$$
Assume that $y$ is such as above.
Then, one can go both from $y$ to $x$ in time $t$ with $|J_1|$ reflexions beginning on $\Theta_{(J_1)_1}$ and from $y$ to $z$ in time $s$ with $|J_2|$ reflexions beginning on $\Theta_{(J_2)_1}$, 
\emph{beginning in both cases in the same unique direction}. As this same unique direction is given, we have necessary $\Theta_{(J_1)_1} = \Theta_{(J_1)_2}$, ie both path begin with a reflexion on the same obstacle. Next, observe that if $x=z$, necessarily $s=t$. Thus, in the case where $s\neq t$, we have $x \neq z$, hence the couple $(x, z)$ fixes the ray. Finally, $s$ (or $t$) fixes $y$ on this ray. Hence, for $s \neq t$, there is at most one critical point of the retarded phase.

We now need to study the Hessian of this phase. Observe that
$$
(D^2_{y, \xi} {\mathcal S}_{J})(x, y, \xi, t) = \operatorname{I}_3, \quad (D^2_{y, y} {\mathcal S}_{J})(x, y, \xi, t) = 0,
$$
therefore, $D^2_{\xi, \xi', y}\Psi_{J_1, J_2}$ is given in $3 \times 3$ blocks by
$$
D^2_{\xi, \xi', y}\Psi_{J_1, J_2}(x,z,y,\xi, \xi', t,s) =
\left(
\begin{array}{c|c|c}
D^2_\xi \mathcal{S}_{J_1}(x,\xi,t) & 0 & \operatorname{I}_3 \\
\hline
0 & - D^2_{\xi'} \mathcal{S}_{J_2}(z,\xi',s) & - \operatorname{I}_3\\
\hline
\operatorname{I}_3 & - \operatorname{I}_3& 0
\end{array}
\right).
$$
We can simplify the determinant of such a block matrix:
\begin{lem}
Let $X, Y \in GL_n (\mathbb R)$ be invertible $n\times n$ matrices and $M \in \mathcal M_{3n}(\mathbb R)$ be given in block form by
$$
M := \left(
\begin{array}{ccc}
X & 0 & \operatorname{I}_n \\
0 & Y & - \operatorname{I}_n\\
\operatorname{I}_n & - \operatorname{I}_n& 0
\end{array}
\right).
$$
Then
$$
\det M = (-1)^n\det(X+Y).
$$
\end{lem}
\begin{proof}
Denoting 
$$
A : = \left(
\begin{array}{cc}
X & 0 \\
0 & Y
\end{array}
\right),
\hspace{0.3cm}
B := 
\left(
\begin{array}{c}
\operatorname{I}_n \\
- \operatorname{I}_n
\end{array}
\right),
\hspace{0.3cm} C:= \left(
\begin{array}{cc}
\operatorname{I}_n & - \operatorname{I}_n
\end{array}
\right), \hspace{0.3cm} 
D:= 0_n,
$$ 
and using the following determinant formula for block matrices
$$
\det \left(\begin{array}{cc}
A B \\
CD 
\end{array}\right) = \det (A) \det (D - C A^{-1} B),
$$
one obtain
$$
\det M =(-1)^n \det (XY) \det\big( X^{-1} + Y^{-1}\big).
$$
This simplifies to
\begin{align*}
\det M &=(-1)^n \det\big( XYX^{-1} + X\big) = (-1)^n\det\big( X(Y+X)X^{-1}) \\
&= (-1)^n\det(X+Y). 
\end{align*}
\end{proof}
We therefore have
$$
\det D^2_{\xi, \xi', y}\Psi_{J_1, J_2}(x,z,y,\xi, \xi', t,s)  = - \det\big( D^2_\xi \mathcal{S}_{J_1}(x,\xi,t) -  D^2_{\xi'} \mathcal{S}_{J_2}(z,\xi',s)\big).
$$
Let us fix a critical point and evaluate $\det D^2_{\xi, \xi', y}\Psi_{J_1, J_2}$ at this point. We have either $t>s$ and $J_1 \supset J_2$ or $t<s$ and $J_2 \supset J_1$. Let us for example assume that we are in the first situation.
Thanks to Lemma \ref{lem:nondeg}, the above writes, because $x$, $z$ and $y$ are on the same ray
$$
\det D^2_{\xi, \xi', y}\Psi_{J_1, J_2}  =  \det \int_{2s}^{2t} \mathcal L_{J^{(\tau)}} \big(X_{-\tau}(x,|\xi|\nabla\varphi_{J_1}(x,\xi)), \xi \big) \; d\tau.
$$
Observe that the $\mathcal L_J$'s are symmetric non-negative matrices. Thus, from Minkowski determinant inequality, we have
$$
\big(\det D^2_{\xi, \xi', y}\Psi_{J_1, J_2} \big)^{\frac 13} \geq 
\int_{2s}^{2t} \big( \det \mathcal L_{J^{(\tau)}} \big(X_{-\tau}(x,|\xi|\nabla\varphi_{J_1}(x,\xi)), \xi \big) \big)^{\frac 13} \; d\tau.
$$
Let $\mathfrak d := d(\Theta_1, \Theta_2) > 0$. We have in particular
$$
\big(\det D^2_{\xi, \xi', y}\Psi_{J_1, J_2} \big)^{\frac 13} \geq 
\int_{2s}^{\min(2t, 2s + \mathfrak d)} \big( \det \mathcal L_{J^{(\tau)}} \big(X_{-\tau}(x,|\xi|\nabla\varphi_{J_1}(x,\xi)), \xi \big) \big)^{\frac 13} \; d\tau,
$$
and it follows that
\begin{equation*}
\det D^2_{\xi, \xi', y}\Psi_{J_1, J_2} 
 \geq c \big(\min\big( t -s , 1)\big)^3 \inf_{\substack{2s \leq \tau \leq \min(2t, 2s+\mathfrak d)\\x, \xi}} \det \mathcal L_{J^{(\tau)}} (x, \xi).
\end{equation*}
Now, a direct computation gives, from the definition of $\mathcal L_J$, 
$$
\det \mathcal L_J = \det (D_x \nabla_\xi \psi_J) ^ 2,
$$
but $D_x \nabla_\xi \psi_J = D_x \nabla_\xi \mathcal S_J$, and therefore, from Lemma \ref{lem:nondeg},
$$
\det \mathcal L_J (x, \xi) = \big(\det D_x \widehat X_{-2t}(x,|\xi|\nabla\varphi_{J}(x,\xi))\big)^2. 
$$
This is nothing but the product of the Gaussian curvatures
$$
\det D_x \widehat X_{-2t} (x,|\xi|\nabla\varphi_{J}(x,\xi)) = \Lambda\varphi_{J}(x,\xi) ^2,
$$
indeed the above follows in the same lines as \cite[proof Lemma 5.2]{plaques}: in the case $|J| = 1$, $\widehat X_{-t} (x,|\xi|\nabla\varphi_{J}(x,\xi)) = \gamma_t \circ g^{-1}$, where
$$
g : (y, s) \in \partial \Theta_{j_1} \times \mathbb R \mapsto  y + s \nabla\varphi_{J}(x,\xi) \in \mathbb R^3,
$$
$$
\gamma_t : (y, s) \in \partial \Theta_{j_1} \times \mathbb R \mapsto y - (t-s) \xi \in \mathbb R^3,
$$
but, from \cite[proof Lemma 5.2]{plaques}, denoting $(x_0, s_0) :=g^{-1}(x)$, we have $\det dg(x_0, s_0) = \Lambda\varphi_{J}(x,\xi) ^{-2} n(x_0) \cdot \nabla\varphi_{J}(x_0, \xi)$, and in addition $\det \gamma_t(x_0, s_0) = n (x_0) \cdot \xi$, from which the result follows as $n(x_0) \cdot \xi = - n(x_0) \cdot \nabla\varphi_{J}(x_0, \xi)$ (because $|J|=1$). For $|J| \geq 2$, we write $\widehat X_{-t} (x,|\xi|\nabla\varphi_{J}(x,\xi)) = \gamma^{J'}_t \circ g^{-1}$ with $g$ as above and
$$
\gamma^{J'}_t : (y, s) \in \partial \Theta_{j_1} \times \mathbb R \mapsto \widehat X^{J'}_{-(t-s)} (y)\in \mathbb R^3,
$$
where $\widehat X^{J'}_{-s}$ is defined in the same way as $X_{-s}$ but ignoring the obstacles after $|J'|$ reflexions, and we proceed similarly by induction on $|J|$.

From Proposition \ref{prop:convL}, it follows that, if $J = rI + l$,
$$
\det \mathcal L_J (x, \xi) \geq c \lambda^{4r}.
$$
Now, observe that, for $0 \leq \tau \leq \min(2t, 2s+\mathfrak d)$, $J^{(\tau)} \subset J_2^+$, where $J_2^+$ is constructed by adding one reflexion to $J_2$. The result follows.
\end{proof}

\begin{lem} \label{lem:diff_times}
There is $c, C>0$ so that, for any $J_1, J_2$ and $(x, z, s, t)$ so that the phase $\Psi_{J_1, J_2}$ has a unique critical point $(\xi_c,\xi_c, y_c)(x, z, s, t)$, if both $w_{k_1}^{J_1}(x, t, \xi_c) \neq0$ and $w_{k_2}^{J_2}(z, s, \xi_c)\neq 0$, then
$$
 ||J_1| - |J_2|| \geq  c|t-s| - C.
$$  
\end{lem}
\begin{proof}
This follows from Lemma \ref{lem:supprought} (transport of supports) together with point (1)-(b) of Lemma \ref{lem:ret_phases}.
\end{proof}
\section{Proof of Theorem 1} \label{sec:proof}

Let $\chi_{\rm tr} \in C^\infty_c(\mathbb R^3)$ so that $\operatorname{supp} \chi_{\rm tr} \cap \Omega \subset \mathcal U^\infty$, and $\chi_{\rm tr} = 1$ in a neighborhood of $\mathcal R \cup \pi_x\operatorname{supp}q$.
Let $\mathcal W_{K}$ denote the parametrix at precision $K$ constructed in the previous section, after truncation of our symbols by $\chi_{\rm tr}$, so that we get a well defined expression:

\begin{multline*}
(\mathcal W_{K} f)(x,t) \\
:= (2 \pi h)^{-3} \int_{y \in \Omega} \int_{\xi \in \mathbb R^3} \sum_{J \in \mathcal J} \sum_{k= 0}^K e^{i(\varphi_J(x,y, \xi)|\xi|-t|\xi|^2)/h} h^k \chi_{\rm tr}(x) w_{k}^J(x,\xi,t) \eta(y) f(y) \; d\xi \, dy.
\end{multline*}

In addition, we will now fix $\chi \in C^\infty_c(\mathbb R^3)$ so that $\chi = 1$ near $\mathcal R$, with $\operatorname{supp} \chi \Subset \{ \chi_{\rm tr} = 1\}$.

\subsection{The parametrix approximation}
We first show that our parametrix indeed approximates $e^{ith\Delta_D} \operatorname{Op}_h(q) $ in logarithmic times near the trapped ray, as intended.

\begin{prop} \label{prop:param_approx}
For any $m \geq 0$ and $N \geq 0$, there exists $K \geq 1$, $C>0$ and $h_0>0$, so that for any $0<\epsilon \leq 1$, any $u_0 \in L^2(\Omega)$ and any $0<h\leq h_0$, 
$$
\Big\Vert \chi(\mathcal W_{K} u_0 - e^{ith\Delta_D} \operatorname{Op}_h(q) u_0) \Big\Vert_{L^\infty \left((0, \epsilon |\log h|), H^m\right)} \leq C h^N \Vert u_0 \Vert_{L^2}.
$$
\end{prop}

\subsubsection{Control of the source term arising from the truncation}

The goal of this subsection is to show that the truncation $\chi_{\rm tr}$ is harmless for the approximation of the solution near the trapped ray.

To this end, define
\begin{multline}
(\mathcal C_{K} f)(x,t) \\
:= (2 \pi h)^{-3} \int_{y \in \Omega} \int_{\xi \in \mathbb R^3} \sum_{J \in \mathcal J} \sum_{k= 0}^K [\Delta, \chi_{\rm tr}]\Big(e^{i(\varphi_J(x,y, \xi)|\xi|-t|\xi|^2)/h} h^k  w_{k}^J(x,\xi,t)\Big) \eta(y) f(y) \; d\xi \, dy. \label{eq:end_def_comu}
\end{multline}
We will show
\begin{prop} \label{prop:trunc_harmless}
Let $K\geq 0$ be fixed. For any $N \geq 0$, there exists $h_0>0$, so that for any $u_0 \in L^2(\Omega)$, any $0<h\leq h_0$, and any $0<\epsilon\leq 1$,  we have,
$$
\Big\| \chi \int_0^t e^{i(t-s)h\Delta_D} (\mathcal C_{K} u_0)(s) \, ds \Big\|_{L^\infty \left((0, \epsilon |\log h|), H^m\right)} \leq C h^N \Vert u_0 \Vert_{L^2}.
$$
\end{prop}
By homogeneity, observe that it suffices to prove Proposition \ref{prop:trunc_harmless} for  $\Vert u_0 \Vert_{L^2} = 1$.

We will denote in this subsubsection, to lighten the notations
$$
F := \mathcal C_{K} u_0.
$$
Observe for future reference that $F$ writes as
$$
F(t,x)=(2\pi h)^{-3}\sum_{J,k}\int \int
 e^{i(|\xi|\varphi_J(x,\xi)-y\cdot\xi-t|\xi|^2)/h}
 h^kg_k^J(x,\xi,t)
\eta(y){u_0}(y)\,d\xi\,dy,
 $$
(recall that $|\xi|\varphi_J(x,y,\xi) = |\xi|\varphi_J(x,\xi)-y\cdot\xi$), where we denoted
$$
 g_k^J(x,\xi,t):=e^{-i|\xi|\varphi_J(x,\xi)/h}h[\Delta,\chi_{\rm tr}](e^{i|\xi|\varphi_J(x,\xi)/h}w_k^J).
$$
By the transport of supports property from Lemma \ref{lem:supprought}, as $\pi_x\operatorname{supp}q$ is away from the boundary, and $\operatorname{supp}\nabla \chi_{\rm tr}$ is away from $\pi_x\operatorname{supp}q$, there is $t_0 >0$ such that $g_k^J =0$ and $F=0$ on $[0, t_0]$. Therefore, we can uniformly extend $F$ and all the $g_k^J$ smoothly to zero on $(-\infty, 0]$. It is theses extensions which we consider in the sequel.

The idea behind Proposition \ref{prop:trunc_harmless} is that the commutator term $F$, which is supported away from $\mathcal R$, is outgoing, hence the flow it produces as a source is never observed near $\mathcal R$.
This outgoingness property is expressed in the following lemma, where the time-translation family is related to the fact that we will need estimates uniform on bounded sub-intervals of $[0, \epsilon |\log h|]$, as will be apparent in the proof of Proposition \ref{prop:trunc_harmless}.

In the following, we let $\chi_0 \in C^\infty_c(\mathbb R^3)$ so that $\chi_0 = 1$ near $\operatorname{supp}\chi$ (and hence near $\mathcal R$), with $\operatorname{supp} \chi_0 \Subset \{ \chi_{\rm tr} = 1\}$.

\begin{lem} \label{lem:trunc_escape}
For any $R_0>0$ so that $B(0,R_0)\Supset \Theta_1 \cup \Theta_2 \cup \operatorname{supp}\chi_0$, there is $S>0$ so that, for any semiclassical family of times $(\ell(h))_{0 < h \leq h_0}$, $\ell(h)\in [0, \lceil\epsilon|\log h|\rceil]$, and any family $(u_{0,h})$ with $\Vert u_{0,h}\Vert_{L^2}=1$ (still denoted $u_0$), we have
\begin{equation} \label{eq:src1}
\forall (x,t,\xi,\tau) \in {}^b \operatorname{WF}^{t,x}_h(F(t+\ell(h))), \quad
\begin{cases}
 \forall s \geq 0, \quad  \varphi_s(x,\xi) \notin {}^bT^*\operatorname{supp}\chi_0, \text { and }\\
  \forall s \geq S, \quad  \varphi_s(x,\xi) \in T^*(B(0,R_0)^c). 
 \end{cases}
\end{equation} 
In addition,
\begin{equation} \label{eq:src2}
{}^b \operatorname{WF}^{t,x}_h(F(t+\ell(h))) \subset \pi_b (\{ \tau = -|\xi|^2, \; |\xi|^2 \in [\alpha_0, \beta_0] \}).
\end{equation}
\end{lem}

\begin{proof}
Let $\chi_{\rm com} \in C^\infty_c(\mathbb R^3)$ be so that $\chi_{\rm com} = 1$ near $\operatorname{supp}\nabla\chi_{\rm tr}$ and $\chi_{\rm com} = 0$ near $\pi_x \operatorname{supp}q \cup \operatorname{supp}\chi_0$.
We will show by non-stationary phase that ${}^b \operatorname{WF}^{t,x}_h(F(t+\ell(h)))$ is contained in 
\begin{multline}\label{eq:phasecovectors}
 \overline{\bigcup_{0<h \leq h_0}\bigcup_{J\in\mathcal J, k \leq K} \Big\{ \pi_b \big(x,t,|\xi|\nabla_x\varphi_J(x,\xi),-|\xi|^2\big),} \\
 \overline{ (x,\xi,t) \text{ s.t. }t+\ell(h)\geq 0 \text{ and }\chi_{\rm com} w_k^J(x, \xi, t + \ell(h))\neq 0 \Big\}},
\end{multline}
from which the result follows by definition of the reflected phases and the transport of supports property from Lemma \ref{lem:supprought}: indeed, if $(x,t,\xi,\tau)$ is in the union above, then $(x,\xi)$ is the endpoint of a ray coming from $\pi_x\operatorname{supp} q$. As in addition $x\in\operatorname{supp}\chi_{\rm com}$ and $\pi_x \operatorname{supp} q \cup \operatorname{supp}\chi_0$ is contained in a tube around $\mathcal R$ with $\operatorname{supp}\chi_{\rm com}$ at distance $\geq c > 0$ of this tube, this implies the stated uniform escape property by convexity.

Let $\vartheta \in C^\infty_c(\mathbb R)$. We have
\begin{align*}
 \vartheta(t)F(t+\ell(h),x)
 &=(2\pi h)^{-3}\sum_{J,k}\int \int
 e^{i(|\xi|\varphi_J(x,\xi)-y\cdot\xi-t|\xi|^2)/h}
 \vartheta(t)h^kg_k^J(x,\xi,t+\ell(h))\notag\\
 &\hspace{38mm}\times
 e^{-i\ell(h)|\xi|^2/h}\eta(y){u_0}(y)\,d\xi\,dy.
\end{align*}
Let now $b \in C^\infty_c({}^bT^* (\Omega\times \mathbb R))$ supported in a single coordinate chart: for example, in a coordinate chart near the boundary. Then we can write
\begin{align*}
 &(\operatorname{Op}_h(b)\vartheta(\cdot)F(\cdot+\ell(h),\cdot))(t,x) \\
 &=(2\pi h)^{-7}\sum_{J,k} \int_{z_1>0}\int_{z', s, \tau, \eta, \xi, y}
 e^{i((x-z)\cdot\eta+(t-s)\tau)/h} 
 e^{i(|\xi|\varphi_J(z,\xi) - y\cdot \xi-s|\xi|^2)/h} \\
 &
 \hspace{38mm}\times b(x,t,x_1\eta_1,\eta',\tau) \vartheta(s)h^kg_k^J(z,\xi,s+\ell(h))\notag\\
 &\hspace{38mm}\times
 e^{-i\ell(h)|\xi|^2/h}\eta(y){u_0}(y)\,d\xi\,dy\,dz\,ds\,d\eta\,d\tau.
\end{align*}
By change of variable $z_1=x_1r$ and  $\eta_1=\omega/x_1$ (recall that $x_1, z_1>0$), so that $dz_1\,d\eta_1=dr\,d\omega$, this writes
\begin{align*}
 &(\operatorname{Op}_h(b)\vartheta(\cdot)F(\cdot+\ell(h),\cdot))(t,x) \\
 &=(2\pi h)^{-7}\sum_{J,k} \int_{r>0}\int_{z', s, \tau, \omega, \eta', \xi, y}
 e^{i\Phi_J/h} b(x,t, \omega,\eta',\tau) \vartheta(s)h^kg_k^J(z,\xi,s+\ell)\notag\\
 &\hspace{38mm}\times
 e^{-i\ell(h)|\xi|^2/h}\eta(y){u_0}(y)\,d\xi\,dy\,dz'\,ds\,d\eta'\,d\tau \, dr \, d\omega,
\end{align*}
where 
\[
 \Phi_J:=(1-r)\omega+(x'-z')\cdot\eta'+(t-s)\tau
          +|\xi|\varphi_J((x_1r,z'),\xi)-y\cdot\xi-s|\xi|^2.
\]

Let $\rho\in C_c^\infty((1/2,3/2))$, equal to one near $1$. Integrating by part an arbitrary number of times in $\omega$ on the part with $1-\rho(r)$ by using $e^{i\Phi_J/h}=\frac{h}{i(1-r)}\partial_\omega e^{i\Phi_J/h}$, we obtain, in any $H^\sigma(\mathbb R \times \Omega)$
\begin{align*}
 &(\operatorname{Op}_h(b)\vartheta(\cdot)F(\cdot+\ell(h),\cdot))(t,x) \\
 &=(2\pi h)^{-7}\sum_{J,k} \int_{r>0}\int_{z', s, \tau, \omega, \eta', \xi, y}
 e^{i\Phi_J/h} b(x,t, \omega,\eta',\tau) \rho(r) \vartheta(s)h^kg_k^J((rx_1,z'), \xi,s+\ell(h))\notag\\
 &\hspace{38mm}\times
 e^{-i\ell(h)|\xi|^2/h}\eta(y){u_0}(y)\,d\xi\,dy\,dz'\,ds\,d\eta'\,d\tau \, dr \, d\omega, \\
 &\quad+O(h^\infty)\Vert u_0 \Vert_{L^2}.
\end{align*}
Indeed, to get $O(h^\infty)$ in $L^2$ after the integrations by parts in $\omega$, we observe that all variables except in $r$ (for which we get a bound in ${h^M}{\langle r \rangle^{-M}}$, $M \gg 1$)  are in compact sets, and note that, from \lemref{support},
$$
 \operatorname{card} \{ J, \; \vartheta(s)g^k_J(\cdot,\ell+s,\xi) \neq 0 \} \lesssim 1+\ell \lesssim 1+\epsilon |\log h|.
$$
To have the $H^\sigma$ claim, we differentiate first and use Proposition \ref{prop:contrder} and Corollary \ref{cor:boundsdirect_t} to control the derivatives. Now, observe that
$$
 \nabla_{(r,z',s,\omega,\eta',\tau)}\Phi_J=
 \big(-\omega+x_1|\xi|\partial_{x_1}\varphi_J,
     -\eta'+|\xi|\partial_{z'}\varphi_J,
     -\tau-|\xi|^2,\ 1-r,\ x'-z',\ t-s\big).
$$
This is zero on
$$
\begin{cases}
s = t ,\\
z' = x', \\
r = 1, \\
\tau = - |\xi|^2, \\
\eta' = |\xi|\partial_{z'} \varphi_J, \\
\omega = x_1 |\xi| \partial_{x_1} \varphi_J.
\end{cases}
$$
Thus, in particular, if $b$ is supported away from the set in (\ref{eq:phasecovectors}),  there is a constant $c>0$ so that
\begin{gather*}
\forall (x,t, \omega,\eta',\tau) \in \operatorname{supp} b, \quad \forall r \in \operatorname{supp}\rho,\quad \forall (z',s) \text{ s.t. } g_k^J((rx_1,z'),\xi,s+\ell(h))\neq 0, \\  |\nabla \Phi_J|\geq c > 0.
\end{gather*}
Integrating by parts an arbitrary number of times in $(r,z',s,\omega,\eta',\tau)$ by using
\[
 \mathcal L_J=\frac hi\frac{\nabla\Phi_J}{|\nabla\Phi_J|^2}
                   \cdot\nabla,
 \qquad \mathcal L_Je^{i\Phi_J/h}=e^{i\Phi_J/h},
\]
we therefore obtain, similarly as before
$$
(\operatorname{Op}_h(b)\vartheta(\cdot)F(\cdot+\ell(h),\cdot))  = O(h^\infty)\Vert u_0 \Vert_{L^2}
$$
in any $H^\sigma(\mathbb R \times \Omega)$, and the result follows.
\end{proof}

Before concluding with the proof of Proposition \ref{prop:trunc_harmless}, we show that the non-spectrally localised part of the Duhamel term is negligible near $\mathcal R$:
 
 \begin{lem} \label{lem:Duh_spec_loc}
 Let $\vartheta\in C_c^\infty((-1,1))$, and, for any $0\leq \ell \leq \lceil\epsilon|\log h|\rceil$,
 $$
F_\ell(t):=\vartheta(t-\ell) F,
 $$
 $$
v_\ell(t):=-i\int_{-\infty}^t e^{i(t-s)h\Delta_D}F_\ell(s)\,ds.
$$
For any $\psi_1$ so that $\psi \prec \psi_1$, for any $\sigma \geq 0$
$$
\sup_{t\in [-1, \epsilon |\log h|+2]}\Vert\chi(1-\psi_1(-h^2\Delta_D))v_\ell(t) \Vert_{H^\sigma} = O(h^\infty)\Vert u_0\Vert_{L^2}, 
$$
uniformly in $0\leq \ell \leq \lceil\epsilon|\log h|\rceil$.
In addition, there exists $\varphi \in C^\infty_c((-\infty,0))$ so that for any  $m, \sigma \geq 0$, 
\begin{equation}\label{eq:duh_loc_time}
 \|\partial_t^m(1-\varphi(hD_t))F_\ell\|_{L^1(\mathbb R, H^\sigma(\Omega))}
 = O(h^\infty)\|u_0\|_{L^2},
\end{equation}
uniformly in  $0\leq \ell \leq \lceil\epsilon|\log h|\rceil$.

 \end{lem}
 \begin{proof}
 As a preliminary, observe that we have the following a-priori crude bound:
\begin{equation} \label{eq:com_crude}
\forall \sigma \geq 0, \; \forall j \geq 0, \; \exists M \geq 0, \quad\Vert \partial_t^j F \Vert_{L^1([-1, \epsilon |\log h|+2], H^\sigma)} \lesssim h^{-M} \Vert u_0 \Vert_{L^2},
\end{equation}
indeed
$$
\Vert \partial_t^j F \Vert_{L^1([-1, \epsilon |\log h|+2], H^\sigma)} \lesssim |\log h| \Vert \partial_t^j F \Vert_{L^\infty([-1, \epsilon |\log h|+2], H^\sigma)},
$$
and (\ref{eq:com_crude}) follows from Corollary \ref{cor:boundsdirect_t} to control the spatial and time derivatives (without need to use the $\lambda$ decay here),
\lemref{support} to show that $|\left\{ J\in\mathcal{I},\text{ s.t }g_{k}^{J}\neq0\right\} | \lesssim 1+t \lesssim 1+\epsilon |\log h|$, and the fact that $F$ has fixed compact spatial support.
 
We first show (\ref{eq:duh_loc_time}). Let $\varphi \in C^\infty_c((-\infty,0))$ so that $\varphi = 1$ near $[-\beta, -\alpha]$ and $\psi_1 = 1$ near $\operatorname{supp}\varphi(-\cdot)$ (i.e $\psi \prec \varphi(-\cdot)\prec\psi_1$).
For $|t-\ell|\geq 2$,
by writing
$$
(1-\varphi(hD_t))F_{\ell} (t) = -\frac{1}{2\pi h} \int_{\ell-1}^{\ell+1}\int e^{i(t-s)\tau/h}(1-\varphi(\tau)) F_\ell(s) d\tau ds,
$$
and integrating by parts in $\tau$ an arbitrary number of times using $|t-s|\geq \frac 12|t-\ell|$ gives, using (\ref{eq:com_crude}) in the second identity 
$$
 \|\partial_t^m(1-\varphi(hD_t))F_\ell\|_{L^1(|t-\ell|\geq 2, H^\sigma(\Omega))}
 =O(h^\infty)\|F_\ell\|_{L^1 H^\sigma} = O(h^\infty)\|u_0\|_{L^2},
$$
uniformly in $\ell$. Now for $|t-\ell|\leq 2$, we write
after Fourier inversion in time
\begin{align*}
&(1-\varphi(hD_t))F_{\ell} (t,x)\\
&=(2\pi h)^{-4} \sum_{J \in \mathcal J} \int
 e^{i(t-\ell)\tau/h}e^{i(|\xi|\varphi_J(x,\xi)-y\cdot\xi-\ell|\xi|^2)/h} e^{-is(\tau+|\xi|^2)/h}
(1-\varphi(\tau))
              \\
             & \hspace{4cm}\times  \sum_{k= 0}^K \vartheta(s) h^k g^k_J(\ell+s,x,\xi)\eta(y){u_0}(y)\,ds\,d\tau\,dy\,d\xi,
\end{align*}
where we recall that $g^k_J := e^{-i|\xi|\varphi_J/h}h[\Delta,\chi_{\rm tr}] (e^{i|\xi|\varphi_J/h}w^k_J)$.
Now, observe that
\begin{itemize}
\item if $\tau \in \operatorname{supp}(1-\varphi)$ and $\xi \in \pi_\xi\operatorname{supp} g_J \subset \{ |\xi|^2 \in [\alpha, \beta] \}$, we have 
$$
|\tau + |\xi|^2| \gtrsim \langle \tau \rangle,
$$
\item for  $s\in \operatorname{supp}\vartheta$, we have $-1\leq \ell + s \leq \epsilon |\log h| + 2$, thus it follows from \lemref{support} that 
 $$
 \operatorname{card} \{ J, \; \vartheta(s)g^k_J(\ell+s,x,\xi) \neq 0 \} \lesssim 1+\ell \lesssim 1+\epsilon |\log h|,
 $$
\item the $\vartheta(s)g^k_J(\ell+s,x,\xi)$ are  supported in $x$ and $s$ in fixed compact sets, and we need only the estimate in the compact $|t-\ell| \leq 2$,
\end{itemize} 
hence, differentiating in $t$ and $x$ then integrating by parts an arbitrary number in times in $s$ in the above and using Proposition \ref{prop:contrder} and Corollary \ref{cor:boundsdirect_t} (without need to use the decay in $\lambda^{|J|}$ here) to control the derivatives in $x$ and $s$, gives (\ref{eq:duh_loc_time}) uniformly in  $0\leq \ell \leq \lceil\epsilon|\log h|\rceil$.

It follows from (\ref{eq:duh_loc_time}) that, for any $\sigma \geq 0$  
 $$
 v_\ell = \varphi(hD_t)v_\ell + O(h^\infty)_{L^\infty([-1, \epsilon |\log h|+2], H^\sigma)}\Vert u_0 \Vert_{L^2},
 $$
 uniformly in $\ell$.
 Hence, it suffices to show the result for $v_\ell$ replaced with $\varphi(hD_t)v_\ell$.
 Now, observe that, taking the (semiclassical) Fourier transform in time
 $$
 (-h^2\Delta_D + \tau) (1-\psi_1(-h^2\Delta))\varphi(\tau)\widehat v_\ell(\tau) = - h(1-\psi_1(-h^2\Delta))\varphi(\tau) \widehat F_\ell(\tau),
 $$
 so that, as $\frac{1-\psi_1(\lambda)}{\lambda + \tau}$ is well defined near $\tau \in \operatorname{supp}\varphi$
\begin{align*}
 \chi(1-\psi_1(-h^2\Delta))\varphi(hD_t) v_\ell(t) &= - \frac{1}{2\pi} \int e^{it\tau/h}\varphi(\tau)\chi\frac{1-\psi_1(-h^2\Delta_D)}{-h^2\Delta_D + \tau}\widehat F_\ell(\tau) d\tau \\
 &= - \frac{1}{2\pi} \int e^{it\tau/h}\varphi(\tau)\chi\frac{1-\psi_1(-h^2\Delta_D)}{-h^2\Delta_D + \tau}\chi_{\rm com} \widehat F_\ell(\tau) d\tau,
\end{align*}
where $\chi_{\rm com} \in C^\infty_c$ is so that $\chi_{\rm com} = 1$ near $\operatorname{supp}\nabla \chi_{\rm tr}$ and $d(\operatorname{supp}\chi_{\rm com}, \operatorname{supp}\chi)>0$ (reducing $\operatorname{supp}\chi$ if necessary). Applying Lemma \ref{lem:offdiag_fc} with $\chi_1 = \chi$, $\chi_2 = \chi_{\rm com}$ and $f(\lambda) = \widetilde \varphi(\tau)\frac{1-\psi_1(\lambda)}{\lambda + \tau}$, where  $\widetilde \varphi \in C^\infty_c$ is so that $\widetilde \varphi = 1$ on $\operatorname{supp} \varphi$, then integrating in $\tau$, gives 
$$
\sup_{t\in \mathbb R}\Vert\chi(1-\psi_1(-h^2\Delta_D)) \varphi(hD_t)v_\ell(t) \Vert_{H^\sigma} = O(h^\infty) \sup_{\tau} \Vert \widehat F_\ell(\tau) \Vert_{L^2},
$$
and to conclude, we observe that, for $0\leq \ell \leq \lceil\epsilon|\log h|\rceil$, by (\ref{eq:com_crude})
$$
\Vert \widehat F_\ell(\tau) \Vert_{L^2} \leq \Vert F_\ell \Vert_{L^1L^2} \leq h^{-M} \Vert u_0 \Vert_{L^2}.
$$
 \end{proof}
 
 We can now conclude this subsection:
 
\begin{proof}[Proof of Proposition \ref{prop:trunc_harmless}]
Let $\vartheta\in C_c^\infty((-1,1))$ such that
$\sum_{\ell\in\mathbb Z}\vartheta(t-\ell)=1$. Recall the definitions of $v_\ell$ and $F_\ell$ from Lemma \ref{lem:Duh_spec_loc}.
The Duhamel term we want to estimate writes
$$
 \int_{0}^t e^{i(t-s)h\Delta_D}F(s) \, ds = i \sum_{\ell = 0}^{\lceil \epsilon |\log h| \rceil}v_\ell(t), \quad 0\leq t \leq \epsilon |\log h|.
$$
Thus, 
recalling that by homogeneity,  it suffices to prove Proposition \ref{prop:trunc_harmless} for  $\Vert u_0 \Vert_{L^2} = 1$,
 it suffices to show that, for any $u_0$ with $\Vert u_0 \Vert_{L^2} = 1$,
\begin{equation} \label{eq:duh_cont_unif}
\sup_{0\leq \ell \leq \lceil\epsilon|\log h|\rceil} \Vert \chi v_\ell \Vert_{L^\infty([0, \epsilon |\log h|], H^\sigma)} = O(h^\infty). 
\end{equation}
We will show that, for any semiclassical families $(\ell(h))_{0< h \leq h_0}$ with $\ell(h)$ integers in $\{ 0,\ldots,\lceil\epsilon|\log h|\rceil \}$ and $(u_{0,h})_{0< h \leq h_0}$ with $\Vert u_{0,h} \Vert_{L^2} = 1$ (still denoted $u_0$), we have
\begin{equation} \label{eq:duh_cont_fam}
\Vert \chi v_{\ell(h)} \Vert_{L^\infty([0, \epsilon |\log h|], H^\sigma)} = O(h^\infty). 
\end{equation}
Indeed, (\ref{eq:duh_cont_unif}) then follows by contradiction. To show (\ref{eq:duh_cont_fam}), let now  $(\ell(h))_{0\leq h \leq h_0}$ and $(u_{0,h})_{0\leq h \leq h_0}$(still denoted $u_0$) be fixed families. We denote 
$$
v_h := v_{\ell(h)}, \quad \widetilde v_h(t) := v_{\ell(h)}(t+\ell(h)).
$$
Observe that, by change of variable, then using the fact that $\widetilde v_h(t) = 0$ for $t \leq -1$, we get
\begin{align}
\sup_{0\leq t \leq \epsilon |\log h|}\Vert \chi v_h \Vert_{H^\sigma} &=
\sup_{-\ell(h)\leq t \leq \epsilon |\log h|-\ell(h)}\Vert \chi \widetilde v_h \Vert_{H^\sigma} \nonumber \\
&\leq \sup_{-1\leq t \leq \epsilon |\log h|}\Vert \chi \widetilde v_h
\Vert_{H^\sigma}, \label{eq:com_relT}
\end{align}
therefore it suffices to control $\sup_{-1\leq t \leq \epsilon |\log h|}\Vert \chi \widetilde v_h \Vert_{H^\sigma}$. Now, $\widetilde v_h$ is solution to
$$
\begin{cases}
i\partial_{t}\widetilde v_h+h\Delta_D \widetilde v_h=\vartheta(t)F(t+\ell(h)),\\
\widetilde v_h = 0 \text{ on }\partial\Omega, \\
\widetilde v_h = 0 \text{ for }t\leq -1.
\end{cases}
$$
We apply Proposition \ref{prop:source2} to $\widetilde v_h$, with the assumption on the source satisfied thanks to (\ref{eq:com_crude}). By the conclusion of Proposition \ref{prop:source2}, using (\ref{eq:src2}) from Lemma \ref{lem:trunc_escape}, if $\rho_{0} = (x_0, \xi_0, t_0, \tau_0) \in {}^b \operatorname{WF}^{t,x} \widetilde v_h$,
then the broken characteristic starting from $\rho_0$ must intersect in its past
$$
{}^b \operatorname{WF}^{t,x}_h F(t+\ell(h)) \cap \{ t \in [-1, 1] \}.\label{eq:wfs}
$$
By Lemma \ref{lem:trunc_escape}, (\ref{eq:src1}), we therefore have, for some $S\geq 0$ and $R_0 \geq 1$ so that $B(0, R_0)$ contains the obstacles and $\operatorname{supp}\chi_0$, (recall that $\chi_0$ was fixed with $\chi_0 = 1$ on $\operatorname{supp}\chi$)
\begin{equation} \label{eq:comWF1}
{}^b \operatorname{WF}^{t,x}_h \chi_0 \widetilde v_h = \emptyset,
\end{equation}
and
\begin{equation} \label{eq:comWF2}
\forall(x,t,\xi,\tau)\in{}^b \operatorname{WF}^{t,x}_h \widetilde v_h \cap \{ t \geq S+1\}, \quad \forall s \geq 0, \quad \text{ we have } |x+s\xi| \geq R_0.
\end{equation}
By the pull back Lemma \ref{lem:WFpullback} (with the ellipticity at $\tau = \infty$ coming from (\ref{eq:duh_loc_time})), we therefore get from (\ref{eq:comWF2})
\begin{equation} \label{eq:comWF3}
\forall(x,\xi)\in {}^b \operatorname{WF}_h^{x} \widetilde v_h(S+2),  \quad \forall s \geq 0, \quad \text{ we have }|x+s\xi| \geq R_0.
\end{equation}

Similarly as in the proof of Proposition \ref{prop:fr_strichartz}, using that $F$ has fixed compact spatial support, together with the Duhamel formula and semiclassical approximate finite speed of propagation of Lemma \ref{lem:semi_speed} in times bounded by $S+3$, we can use Lemma \ref{ref:ell_space} to get, from (\ref{eq:comWF3}), taking $\psi \prec \psi_1$
\begin{equation}\label{eq:duh_pqck_out}
 \psi_1(-h^2\Delta_D) \widetilde v_h(S+2)= \psi_1(-h^2\Delta_D)\operatorname{Op}_h(a_{\rm out})\widetilde v_h(S+2)
          +O(h^\infty)_{H^\sigma}\|u_0\|_{L^2},
\end{equation}
where $a_{\rm out}\in C_c^\infty(T^*\mathbb R^3)$,
is supported outside a ball containing the obstacles and $\operatorname{supp}\chi$,
with every forward straight ray from its support avoiding that ball.
Thus, applying Lemma \ref{lem:away_is_away} we have
\begin{align*}\label{eq:escr2_co}
&\Vert \chi e^{ith\Delta_D}\psi_1(-h^2\Delta_D) \widetilde v_h (S+2) \Vert_{L^\infty([0, \epsilon|\log h|], H^\sigma)} \\
&\qquad= O(h^\infty) \Vert  \widetilde v_h(S+2) \Vert_{L^2} + O(h^\infty) \Vert u_0 \Vert_{L^2}= O(h^\infty) \Vert u_0 \Vert_{L^2},
\end{align*}
where we used (\ref{eq:com_crude}) to obtain the second line.
Observing that $e^{ith\Delta_D}\psi_1(-h^2\Delta_D) \widetilde v_h (S+2) = \psi_1(-h^2\Delta_D) \widetilde v_h(t+S+2)$, we get after time change of variable \begin{equation} \label{eq:comR1}
\Vert \chi  \psi_1(-h^2\Delta_D) \widetilde v_h \Vert_{L^\infty([S+2, \epsilon|\log h|+S+2], H^\sigma)} =  O(h^\infty)\Vert u_0 \Vert_{L^2}.
\end{equation}
Finally, on the fixed time interval $[-1, S+2]$, Lemma \ref{lem:WFpullback} together with Lemma \ref{ref:ell_space} (and Lemma \ref{lem:offdiag_fc} to remove $\chi_0$ at an $O(h^\infty)$ cost) give from (\ref{eq:comWF1})
\begin{equation} \label{eq:comR2}
\Vert \chi   \psi_1(-h^2\Delta_D) \widetilde v_h \Vert_{L^\infty([-1, S+2], H^\sigma)} =  O(h^\infty)\Vert u_0 \Vert_{L^2},
\end{equation}
indeed, for the uniformity in time in (\ref{eq:comR2}) above, observe that the proof of Lemma \ref{lem:WFpullback} gives an $H^1_tL^2_x$ ($\hookrightarrow C^0_t L^2_x$) bound on a whole time interval.
Collecting (\ref{eq:comR1}) and (\ref{eq:comR2}), together with Lemma \ref{lem:Duh_spec_loc} to remove $\psi_1(-h^2\Delta_D)$, gives (\ref{eq:duh_cont_fam}) thanks to (\ref{eq:com_relT}), which ends the proof.
\end{proof}

\subsubsection{Proof of Proposition \ref{prop:param_approx}}
By construction, $w(t) := \mathcal W_{K} u_0$ verifies the equation, with Dirichlet boundary conditions
$$
i\partial_{t} w+h\Delta_D w= R_{K} +h\mathcal C_{K}u_0,
$$
where we recall that $C_{K}u_0$, defined by (\ref{eq:end_def_comu}), contains the commutator terms with $\chi_{\rm tr}$, and
$$
R_{K} := h^{K+1}\chi_{\rm tr}\frac{1}{(2\pi h)^{3}}\sum_{J\in\mathcal{I}}\int\Delta w_{K}^{J}(x,t,\xi)e^{i(\varphi_{J}(x,\xi)|\xi|-y\cdot\xi-t|\xi|^{2})/h}\eta(y)u_0(y)dy d\xi,
$$
and with data 
\begin{align*}
w(0) &= \chi_{\rm tr}\int \delta^y u_0(y) dy = \chi_{\rm tr} \frac{1}{(2\pi h)^3} \int e^{i(x-y)\cdot \xi/h} \eta(y) q(x,\xi) u_0(y) \, d\xi\,dy \\
&= \frac{1}{(2\pi h)^3} \int e^{i(x-y)\cdot \xi/h} \eta(y) q(x,\xi) u_0(y) \, d\xi\,dy,
\end{align*}
where we used on the second line the fact that $\chi_{\rm tr} = 1$ on $\pi_x\operatorname{supp}q$.

Let now $\widetilde w$ be solution to the same equation than $w$ but without the source term $\mathcal C_{K}u_0$ coming from the commutators, that is
$$
\begin{cases}
i\partial_{t} \widetilde w+h\Delta_D \widetilde w=R_K,\\
\widetilde w(0) = w(0).
\end{cases}
$$
By Proposition \ref{prop:trunc_harmless}, we have, for fixed $K$, for any $\sigma \geq 0$
$$
\Vert \chi( w - \widetilde w) \Vert_{L^\infty([0, \epsilon|\log h|], H^\sigma)} = O(h^\infty)\Vert u_0 \Vert_{L^2}.
$$
Therefore, using the above and (\ref{eq:eta_y1})--(\ref{eq:eta_y2}), it suffices to estimate $\chi z$, where
$$
z(t) := \widetilde w(t) - e^{ith\Delta_D}w(0).
$$
This function $z$ is solution to 
$$
i \partial_{t} z + h\Delta_D z= R_K,
$$
with zero initial data. Therefore, by Duhamel formula
$$
z(t) = -i \int_{0}^t e^{i(t-s)h\Delta_D} R_K(s) ds.
$$
Using the fact that $e^{i(t-s)h\Delta_D}$
is an $H^{m}$--isometry (and Cauchy-Schwarz in $y$, taking advantage of the compact support of $\eta$), we get 
\[
\Vert \chi z(t)\Vert_{H^m_x} \leq C\times|t|\times h^{K-2}\times\sup_{\substack{0 \leq s \leq \epsilon |\log h| \\ \xi}}\sum_{J\in\mathcal{I}}\big\Vert \chi_{\rm tr} \Delta w_{K}^{J}(\cdot,s,\xi)e^{i(\varphi_{J}(\cdot,\xi)|\xi|-s\xi^{2})/h}\big\Vert_{H^{m}} \Vert u_0 \Vert_{L^2}.
\]
Hence,
for $0\leq t\leq\epsilon|\log h|$
\begin{equation}
\Vert \chi z(t)\Vert_{H^{m}}\lesssim|\log h|h^{K-2}\sup_{\substack{0 \leq s \leq \epsilon |\log h| \\ \xi}}\sum_{J\in\mathcal{I}}\big\Vert \chi_{\rm tr} \Delta w_{K}^{J}(\cdot,s,\xi)e^{i(\varphi_{J}(\cdot,\xi)|\xi|-s\xi^{2})/h}\big\Vert_{H^{m}} \Vert u_0 \Vert_{L^2}.\label{eq:rK1}
\end{equation}
Let us denote
\[
W_{K}^{J}(x,t) := \chi_{\rm tr} \Delta w_{K}^{J}(\cdot,t,\xi)e^{i(\varphi_{J}(\cdot,\xi)|\xi|-t\xi^{2})/h}
\]
Using Proposition
\ref{prop:essbouds} to control the derivatives coming from $w_{K}$
and Proposition \propref{contrder} to control the derivatives coming
from the phase we get:
\[
\Vert\partial_x^{\alpha}W_{K}^{J}\Vert_{L^{2}}\lesssim \Vert\partial_x^{\alpha}W_{K}^{J}\Vert_{L^{\infty}}\lesssim h^{-|\alpha|},
\]
and thus, by (\ref{eq:rK1}),
for $0\leq t\leq\epsilon|\log h|$
\begin{equation}
\Vert \chi z(t)\Vert_{H^m}\lesssim|\log h|h^{K-m-2}\sup_{\substack{0 \leq s \leq \epsilon |\log h| \\ \xi}}|\left\{ J\in\mathcal{I},\text{ s.t }\chi_{\rm tr}w_{K}^{J}(\cdot, s, \xi)\neq0\right\} |\Vert u_0 \Vert_{L^2}.\label{eq:rK0}
\end{equation}
Note that $\chi_{\rm tr} w_{K}^{J}(s)\neq0$ implies by \lemref{support} that
$|J|\leq c_{1}s$, and $|\left\{ J\in\mathcal{I},\text{ s.t }\chi_{\rm tr}w_{K}^{J}(s)\neq0\right\} |$
is bounded by the number of elements in 
\[
\Bigg\{ \emptyset,(1),(2),(1,2),(2,1),(1,2,1),(2,1,2),\dots,(1,2,\dots),(\underset{\lceil c_{1}s\rceil}{\underbrace{2,1,\dots}})\Bigg\} ,
\]
that is $1+2\lceil c_{1}s\rceil$, so
$$
|\left\{ J\in\mathcal{I},\text{ s.t }\chi_{\rm tr}w_{K}^{J}(s)\neq0\right\} |\lesssim(1+s),\label{eq:wkJnonnul}
$$
and therefore, according to (\ref{eq:rK0}), for $0\leq t\leq\epsilon|\log h|$
$$
\Vert \chi z(t)\Vert_{H^m}  \lesssim |\log h|^{2}h^{K-m-2}\Vert u_0 \Vert_{L^2},
$$
and the result follows by taking $K \geq N + m + 3$. \qed

\subsection{End of the proof of Theorem 1}

\subsubsection{Set-up of the duality argument}
Recall that we are now reduced to prove (\ref{eq:ult_reduc}) for $\epsilon >0$ small enough.
From Proposition \ref{prop:param_approx}, it suffices to show that, $K \geq 1$ being fixed sufficiently large, there exists $\epsilon>0$ small enough so that for any $h>0$ small enough
$$
\big\Vert \chi \mathcal W_{K} u_0 \big\Vert_{L^p ((0, \epsilon |\log h|), L^q)} \lesssim h^{-1/p} \Vert u_0 \Vert_{L^2}.
$$
By duality, this amounts to showing that
\begin{equation}\label{eq:end:duality}
\big\Vert \chi \mathcal W_{K} \mathcal W^*_{K} \chi \big\Vert_{L^{p'} ((0, \epsilon |\log h|), L^{q'}) \to L^{p} ((0, \epsilon |\log h|), L^{q})} \lesssim h^{-2/p}.
\end{equation}
The operator $T_{K} := \chi \mathcal W_{K} \mathcal W^*_{K} \chi$ writes
\begin{align*}
&(T_{K} g)(x,t) = \\
&\hspace{0.5cm} (2 \pi h)^{-6} \int \sum_{\substack{ J_1, J_2 \in \mathcal J \\ 0 \leq k_1, k_2 \leq K}} \exp\Big({ \frac ih \big( {\mathcal S}_{J_1}(x, y, \xi, t) - {\mathcal S}_{J_2}(z, y, \xi', s) \big)} \Big) \\
& \hspace{3cm}\times w_{k_1}^{J_1}(x, t, \xi) \overline{w_{k_2}^{J_2}(z, s, \xi')} h^{k_1 + k_2}\chi(x) \chi(z)\eta(y)^2 g(z, s)  \quad d\xi d\xi' dz ds dy,
\end{align*}
(we used $\chi \chi_{\rm tr} = \chi$) where we recall the notation
$$
{\mathcal S}_{J}(x, y, \xi, t) := \varphi_J(x, y, \xi)|\xi| - t|\xi|^2,
$$
and with the integrals in $\xi$ and $\xi'$ over $\mathbb R^3$ (and actually over $B(0, \sqrt{\beta_0}) \backslash B(0,\sqrt{\alpha_0})$ thanks to the support properties of the symbols), the integrals in $z$ and $y$ over $\Omega$, and the integral in $s$ over $(0, \epsilon |\log h|)$.
For $t, s$ fixed, let us denote $T_{K}(t, s)$ the ``frozen'' operator
\begin{align*}
&(T_{K}(t,s)g)(x) := \\
&\hspace{0.5cm} (2 \pi h)^{-6} \int \sum_{\substack{ J_1, J_2 \in \mathcal J \\ 0 \leq k_1, k_2 \leq K}} \exp\Big({ \frac ih \big( {\mathcal S}_{J_1}(x, y, \xi, t) - {\mathcal S}_{J_2}(z, y, \xi', s) \big)} \Big) \\
& \hspace{3cm}\times w_{k_1}^{J_1}(x, t, \xi) \overline{w_{k_2}^{J_2}(z, s, \xi')} h^{k_1 + k_2}\chi(x) \chi(z)\eta(y)^2 g(z) \quad d\xi d\xi' dz dy \\
& = \int G_{K}(x,z;t,s)  g(z) \; dz,
\end{align*}
where the kernel $G_{K}$, well defined for $x,z \in \Omega \cap \operatorname{supp}\chi$, writes
\begin{align} \label{eq:end:kernel}
&G_{K}(x, z ; t, s) = \\
&\hspace{0.5cm}(2 \pi h)^{-6} \int \sum_{\substack{ J_1, J_2 \in \mathcal J \\ 0 \leq k_1, k_2 \leq K}} \exp\Big({ \frac ih \big( {\mathcal S}_{J_1}(x, y, \xi, t) - {\mathcal S}_{J_2}(z, y, \xi', s) \big)} \Big) \nonumber \\
& \hspace{3cm}\times w_{k_1}^{J_1}(x, t, \xi) \overline{w_{k_2}^{J_2}(z, s, \xi')} h^{k_1 + k_2}\chi(x) \chi(z)\eta(y)^2 \quad d\xi d\xi' dy, \nonumber
\end{align}
with the integrals in $\xi$ and $\xi'$ over $\mathbb R^3$ and the integral in $y$ over $\Omega$.
Observe that
$$
\sup_{s,t \in [0,\epsilon|\log h|} \Vert T_{K}(t,s)\Vert_{L^2 \to L^2} \lesssim 1,
$$
indeed, this comes from example from the fact that $T_{K}(t,s) = (\chi \mathcal W_{K}(t))(\chi \mathcal W_{K}(s))^*$, together with Proposition \ref{prop:param_approx}. 
Therefore, in a classical way, using the Hardy–Littlewood–Sobolev inequality, in order to obtain (\ref{eq:end:duality}) and hence to conclude, it suffices to show that
$$
\forall s,t \in [0, \epsilon|\log h|] \text{ s.t. }s\neq t, \quad \Vert T_{K} (t,s) \Vert_{L^1 \to L^\infty} \lesssim \big(h|t-s| \big)^{-3/2},
$$
and therefore, it suffices to obtain the bound on the kernel
\begin{equation} \label{eq:end:kernel_bound}
\forall s,t \in [0, \epsilon|\log h|] \text{ s.t. }s\neq t, \; \forall x,z \in \Omega \cap \operatorname{supp}\chi, \quad \Big| G_{K}(x, z ; t, s) \Big| \lesssim \big(h|t-s| \big)^{-3/2}.
\end{equation}
Showing (\ref{eq:end:kernel_bound}) occupies the reminder of the proof.

\subsubsection{Decay of the retarded kernel (i.e.\ proof of (\ref{eq:end:kernel_bound}))}
In the sequel, we let $x,z \in \Omega \cap \operatorname{supp}\chi$.

We will perform a stationary phase on each term of the sum arising in  (\ref{eq:end:kernel}), in the variable $(y,\xi, \xi') \in \Omega \times \mathbb R^3 \times \mathbb R^3$ (hence on these three variables at once). From \cite[Theorem 3.16]{semibook} (with $N=1$, $n = 9$), thanks to Lemma \ref{lem:ret_phases}, for any
$ 0 \leq k_1, k_2 \leq K$ and any $J_1, J_2$ so that $\Psi_{J_1, J_2}(x, z, \cdot, \cdot, \cdot, t, s)$ has a critical point,
\begin{align*} 
&\int \exp\Big({ \frac ih \big( {\mathcal S}_{J_1}(x, y, \xi, t) - {\mathcal S}_{J_2}(z, y, \xi', s) \big)} \Big)  w_{k_1}^{J_1}(x, t, \xi) \overline{w_{k_2}^{J_2}(z, s, \xi')} \; d\xi d\xi' dy \\
&\hspace{1cm}= (2\pi h)^{\frac 92} |\det D^2_{\xi, \xi'}\Psi_{J_1, J_2}|^{-\frac 12}e^{i \gamma} w_{k_1}^{J_1} \overline{w_{k_2}^{J_2}} \eta^2  \nonumber \\
&\hspace{1.5cm} + R_{k_1, k_2, J_1, J_2}(x, z, t, s) \nonumber, 
\end{align*}
where the r.h.s. is evaluated in the critical point, $\gamma_{J_1, J_2, h} = \frac{ \pi}{4} \operatorname{sign} (D^2_{\xi, \xi'}\Psi_{J_1, J_2}) + h^{-1} \Psi_{J_1, J_2}$, and
$$
|R_{k_1, k_2, J_1, J_2}(x, z, t, s)| \leq C h^{1+\frac 9 2} \sum_{|\alpha| \leq 12} \sup \big|\partial^\alpha_{\xi, \xi', y} (w_{k_1}^{J_1}(x, t, \xi) w_{k_2}^{J_2}(z, s, \xi'))
\big|.
$$
It follows that, using the upper bound on $ |\det D^2_{\xi, \xi'}\Psi_{J_1, J_2}|$ from Lemma \ref{lem:ret_phases} (2),
\begin{align} 
&|G_{K}(x, z, s, t)| \nonumber \\
&\hspace{0.25cm}\leq (2 \pi h)^{-3/2} \Bigg[ \sum_{\substack{ J_1, J_2 \in \mathcal J_{\rm crit}(x, z, t, s) \\ 0 \leq k_1, k_2 \leq K}}  \max(|t-s|^{-\frac 32},1) \lambda^{-\min({|J_1|}, {|J_2|})} h^{k_1 + k_2} |w_{k_1}^{J_1} w_{k_2}^{J_2} | \nonumber \\
&\hspace{0.25cm}+ C h \sum_{\substack{ J_1, J_2 \in \mathcal J_{\rm crit}(x, z, t, s) \\ 0 \leq k_1, k_2 \leq K}} h^{k_1 + k_2}  \sum_{|\alpha| \leq 12}  \sup\big|\partial^\alpha_{\xi, \xi', y} (w_{k_1}^{J_1}(x, t, \xi) w_{k_2}^{J_2}(z, s, \xi'))
\big| \Bigg] + O(h^\infty), \label{eq:dec_kernel_psa}
\end{align}
where 
$$
\mathcal J_{\rm crit}(x, z, t, s) = \big\{ (J_1, J_2), \; \text{s.t. }\Psi_{J_1, J_2} \text{ has a $(y, \xi, \xi')$--critical point} \big\},
$$ and the $O(h^\infty)$ term corresponds to non-stationary phases (we used (\ref{eq:wkJnonnul}) to obtain that $\sum_{\substack {{J_1, J_2} \\ w_{k_1}^{J_1}  w_{k_2}^{J_2} \neq 0}} O(h^\infty) = O(h^\infty)$ in logarithmic times).
By Lemma \ref{lem:ret_phases} (1c), 
$$
|\mathcal J_{\rm crit}(x, z, t, s)| \leq 1
$$ 
for fixed $(x, z, t, s)$. If $\mathcal J_{\rm crit}(x, z, t, s) \neq \emptyset$, denote $\mathcal J_{\rm crit}(x, z, t, s) := \{ (J_{1, c}(x, z, t, s), J_{2, c}(x, z, t, s)\}$. Then, using Proposition \ref{prop:essbouds} to bound the symbols, then Lemma \ref{lem:diff_times}
\begin{align} \label{eq:dec_kernel_1}
&\sum_{ J_1, J_2 \in \mathcal J_{\rm crit}(x, z, t, s)}  \lambda^{-\min({|J_1|}, {|J_2|})} |w_{k_1}^{J_{1}} w_{k_2}^{J_{2}} | \nonumber \\ &\hspace{1cm}\lesssim \lambda^{-\min({|J_{1,c}|}, {|J_{2,c}|})} \lambda^{\frac{|J_{1,c}|}{2}} \lambda^{\frac{|J_{2,c}|}{2}} \mathbf 1_{w_{k_1}^{J_{1, c}} w_{k_2}^{J_{2,c}} \neq 0} 
= \lambda^{\frac 12 \left||J_{1,c}| - |J_{2,c}|\right|}\mathbf 1_{w_{k_1}^{J_{1,c}} w_{k_2}^{J_{2,c}} \neq 0} \nonumber \\ 
&\hspace{1cm}\leq e^{-\mu|t-s|},
\end{align}
for some $\mu >0$ depending only of the geometry.
On the other hand, still if $\mathcal J_{\rm crit}(x, z, t, s) \neq \emptyset$, using Corollary \ref{cor:boundsdirect} to bound the directional derivatives of the symbols (the $y$-derivative are trivially uniformly bounded), then Lemma \ref{lem:support}, we have for some constants $C_1, C_2, C_3, c > 0$ depending only of the geometry, for $0 \leq t, s \leq \epsilon |\log h|$
\begin{align*}
&h \sum_{ J_1, J_2 \in \mathcal J_{\rm crit}(x, z, t, s) }  \sum_{|\alpha| \leq 12}  \sup\big|\partial^\alpha_{\xi, \xi', y} (w_{k_1}^{J_1}(x, t, \xi) w_{k_2}^{J_2}(z, s, \xi')) 
\big| \nonumber  \\
 &\hspace{1cm}  \lesssim h
C_1
^{|J_{1,c}| + |J_{2,c}|} \mathbf 1_{w_{k_1}^{J_{1,c}} w_{k_2}^{J_{2,c}} \neq 0} \nonumber \\
 &\hspace{1cm}\lesssim h C_2^{t+s}  \nonumber \\
  &\hspace{1cm}\lesssim h \times h^{-C_3 \epsilon} = h^{\frac 12} h^{\frac 12 - C_3 \epsilon}. 
\end{align*}
We now take $\epsilon > 0$ small enough so that $\frac 12 - C_3 \epsilon \geq 0$, and use the fact that for $0 \leq t, s \leq \epsilon |\log h|$ we have $h \leq e^{-\frac 12(t+s)/\epsilon}$ to obtain from the above
\begin{equation}\label{eq:dec_kernel_2}
h \sum_{ J_1, J_2 \in \mathcal J_{\rm crit}(x, z, t, s) }  \sum_{|\alpha| \leq 12}  \sup\big|\partial^\alpha_{\xi, \xi', y} (w_{k_1}^{J_1}(x, t, \xi) w_{k_2}^{J_2}(z, s, \xi')) 
\big| \lesssim e^{-\frac 14(t+s)/\epsilon}.
\end{equation}
Bounding in the same way $O(h^\infty) \leq e^{-\frac 12(t+s)/\epsilon}$ in logarithmic times, we obtain, for $0 \leq t, s \leq \epsilon |\log h|$, injecting (\ref{eq:dec_kernel_1}) and (\ref{eq:dec_kernel_2}) into (\ref{eq:dec_kernel_psa}),
$$
|G_{K}(x, z, s, t)| \lesssim h^{-\frac 32} \max(|t-s|^{-\frac 32},1) e^{-\mu|t-s|}.
$$
The estimate (\ref{eq:end:kernel_bound}), and thus Theorem \ref{thm:main}, follow. \qed

\appendix

\section{Semiclassical analysis of boundary value problems} \label{app:sca}

The purpose of this Appendix is to recall some notions and results of semiclassical  analysis of boundary value problems which are useful in this paper, and to state and prove related results not available in the literature. We will be in particular interested in the propagation of singularities for the semiclassical Schrödinger equation and the Helmholtz equation.

Unless stated otherwise, here $\Omega := \mathbb R^3 \backslash \mathcal O$ is the exterior of a smooth obstacle $\mathcal O$ (with possibly several connected components) with no infinite order contact point.

\subsection{The b-cotangent bundle and microlocalization} \label{a1:ss1}
We first recall the definitions of the b-cotangent bundle $^bT^* \Omega$ (where the quantization takes place), classes of symbols, quantisation, wavefront-sets and operator wavefront-sets. To keep the presentation brief, we present these objects on $^bT^* \Omega$, and highlight that we will also use the corresponding objects on $^bT^* (\Omega \times \mathbb R)$, defined in a strictly similar way.

\textbf{The b-cotangent bundle $^bT^* \Omega$.}
We recall that $^bT^* \Omega$ is the dual to the vector bundle of the vector fields tangent to $\partial \Omega$. In local coordinates so that $\partial\Omega = \{ x_1 = 0 \}$, the vector bundle of the vector fields tangent to $\partial \Omega$ is generated by $\{ x_1 \partial_{x_1}, \partial_{x_2}, \partial_{x_3} \}$, and $T^* \Omega$ can be identified as a subset of $^bT^* \Omega$ using the canonical projection $\pi_b : T^* \Omega \to ^bT^* \Omega$ given by (in local coordinates)
$$
\pi_b : (x, \xi, \xi') \to (x, x_1 \xi_1, \xi'),
$$
where $\xi$ is dual to $x$ and $\xi_1$ to $x_1$.

\textbf{Symbol classes.}
 We say that a $h$-family of symbols $a \in C^\infty( ^bT^* \Omega)$ lies in $S^k ( ^bT^* \Omega)$ iff there is $h_0 >0$ so that
$$
\sup_{h \in (0, h_0)}  \sup \langle \xi \rangle^{|\beta| - k} |\partial_x^\alpha \partial_\xi^\beta a(x, \xi, h)| < +\infty.
$$

\textbf{Quantization.} With $b \in S^k ( ^bT^* \Omega)$ compactly supported in a single coordinate patch so that $\partial\Omega = \{ x_1 = 0 \}$ as above, it can be seen as a function of $(x, x_1 \xi_1, \xi')$, and we define, in local coordinates
$$
(\operatorname{Op}_h(b) v)(x) := (2 \pi h)^{-3} \int e^{i\big[(x_1 - y_1)\cdot \xi_1  + (x'-y')\cdot \xi' \big]/h} b(x, x_1 \xi_1, \xi') v(y) \, dy d\xi.
$$
This is then extended to any $b \in S^k ( ^bT^* \Omega)$ by using a partition of unity (with the standard definition of the quantization away from the boundary).

\textbf{Tempered families.} We say that a family $v_h$ of $L^2_{\rm loc}(\overline \Omega)$ functions is tempered if
$$
\forall \chi \in C^\infty_c(\mathbb R^3), \; \exists C, K >0, \; \text{ s.t. }\Vert \chi v_h \Vert_{L^2(\Omega)} \leq C h^{-K}.
$$
We say that it is polynomially bounded if the same holds with $\chi=1$.

\textbf{Semiclassical wavefront-set.} If $v$ is an $h$-tempered family of $L^2_{\rm loc}$ functions, we say that $\rho_0 \in {}^bT^* \Omega$ does \emph{not} belong to the $b$-wavefront-set of $v$, denoted
$^b \operatorname{WF}^{x}_h v$, if there exists $q \in C^\infty_c(^b T^* \Omega)$, compactly supported, so that $|q(\rho_0)| > 0$ and, for 
 any $N \geq 0$ there exists $C_{N}>0$ so that
\begin{equation} \label{eq:defWF}
\Vert \operatorname{Op}_h(q) v \Vert_{L^2} \leq C_{N} h^N.
\end{equation}
We then write $\rho_0 \notin {}^b \operatorname{WF}_h v$.

\textbf{The corresponding space-time objects.} Similarly, we define symbols, semiclassical pseudo-differential operators and b--wavefront-sets on $^bT^* (\Omega \times \mathbb R)$. For clarity, we will denote
$$
^b \operatorname{WF}^x_h, \quad {}^b \operatorname{WF}^{t,x}_h 
$$
respectively the spatial and space-time b-wavefront sets: when working with a function $u$ of time and space (typically, a solution of the Schrödinger equation), $^b \operatorname{WF}^{t,x}_h u \subset {}^bT^* (\Omega \times \mathbb R)$ denotes the space-time b--wavefront-set of $u$, whereas, for $t_0 \in \mathbb R$ fixed, $^b \operatorname{WF}^x_h u(t_0) \subset {}^bT^* \Omega$ denotes the spatial b--wavefront-set of $u(t_0)$. 

\subsection{The bicharacteristic flow(s)}

Two distinct bicharacteristics flow appear throughout the paper, namely
\begin{itemize}
    \item $\phi_s : {}^bT^*\Omega \mapsto {}^bT^*\Omega$, the generalized flow of Melrose-Sj\"ostrand associated with $ - \Delta_D$, considered on each compressed energy surface
    $\pi_b(\{|\xi|^2=E\})$, $E>0$.
 We refer for example to \cite[Section~3.1]{MR2066943}
 for a precise definition. The idea is that in the interior, the trajectories follow $H_{|\xi|^2}$, transverse reflections follow the laws of geometrical optics, and at tangency, trajectories follow the generalized continuation. 
    \item $\gamma_s : {}^bT^*(\mathbb R \times \Omega) \mapsto {}^bT^*(\mathbb R \times \Omega)$, the generalized flow for $i\partial_t + \Delta_D$, on ${}^bT^*(\mathbb R \times \Omega)$, that is, denoting $(x,\xi,t,\tau)$ the variables on ${}^bT^*(\mathbb R \times \Omega)$
    $$
        \gamma_s(x, \xi, t, \tau) = (\phi_s(x, \xi), t+s, \tau),
    $$
where the flow $\phi_s$ on the right is taken at energy $E = -\tau$.
\end{itemize}

\subsection{Propagation of singularities for the semiclassical Schrödinger flow}

We first recall from \cite{Lebeau} that the wave-front set of the semiclassical Schrödinger flow is contained in the caracteristic set:

\begin{lem}[\cite{Lebeau}] \label{lem:char}
Let $\psi \in C^\infty_c(\mathbb R_+^*)$ and $u_{0,h}$ be a polynomially bounded family in $L^2(\Omega)$. Then
$$
{}^b\operatorname{WF}_h^{t,x}\big(e^{ith\Delta_D}\psi(-h^2\Delta_D)u_{0,h}\big) 
\subset \Sigma_b := \pi_b \Big( \big\{ \tau + |\xi|^2 = 0, \; \tau \in - \operatorname{supp}\psi \big\}\Big).
$$
\end{lem}

The main propagation result for the semiclassical Schrödinger flow is the following.

\begin{thm}[\cite{Lebeau}, \cite{MR0492794, MR644020}] \label{th:prop_sch}
Let $\psi \in C^\infty_c(\mathbb R_+^*)$ and $u_{0,h}$ be a polynomially bounded family in $L^2(\Omega)$. Then
$$
{}^b\operatorname{WF}_h^{t,x}\big(e^{ith\Delta_D}\psi(-h^2\Delta_D)u_{0,h}\big)
$$
is invariant by the flow $\gamma_s$: that is, for any $s\in\mathbb R$, $\rho \in \Sigma_b$ is in the wave-front set above if and only if $\gamma_s(\rho)$ is.
\end{thm}
\begin{proof}[References for the proofs of Lemma \ref{lem:char} and Theorem \ref{th:prop_sch}]
Lemma \ref{lem:char} corresponds to the corollary of \cite[Lemma 1, Section 3]{Lebeau} stated on top of page 278.

Lebeau \cite{Lebeau} proves the result stated in Theorem \ref{th:prop_sch} by 
identifying the semiclassical  b--wavefront set with
the classical b--wavefront set of an auxiliary distribution,
then applying the Melrose--Sjöstrand \cite{MR0492794, MR644020} propagation Theorem to this distribution, see \cite[Section 3, Lemma 1 and equation (18)]{Lebeau}.

Our setting has three little differences with the framework of \cite{Lebeau}: first, \cite{Lebeau} works with bounded families in $L^2(\Omega)$. We get the  statement for polynomially bounded families by multipliying them by a fixed $h^{M}$, which doesn't change their semiclassical wavefront set. Second, \cite{Lebeau} works in a bounded domain, but the argument is local and applies to our setup. Finally, \cite{Lebeau} applies to dyadic families, whereas we work with $h \to 0$ for convenience. The argument for dyadic families can be extended to families $h_{j+1} \leq \frac 12 h_{j}$, and the statement for arbitrary $h \to 0$ can then be recovered by contradiction.

\end{proof}

We will also use a propagation statement for the equation with source term.
The following statement come from \cite[Proposition A.8]{MR2672795}, \cite[Lemma B.7]{plaques}.
\begin{prop}[\cite{plaques, MR2672795}]\label{prop:source1} 
Let $\psi \in C^\infty_c(\mathbb R_+^*)$, $a<b$, $\vartheta\in C_c^\infty((a,b))$, and $Q$ be a pseudo-differential operator of order zero. Let $u_h :=  e^{ith\Delta_D}\psi(-h^2\Delta_D) u_0$ and $w_h$ be solution to
\[
 \begin{cases}
 (ih\partial_t+h^2\Delta_D)w_h=ih\vartheta(t)Q(u_h),\\
 w_h|_{\partial\Omega}=0,\qquad w_h|_{t\leq a}=0.
 \end{cases}
\]
Then, if $\rho_0\in
{}^b\operatorname{WF}_h^{t,x}w_h$, the broken bicharacteristic from $\rho_0$ has a non-empty intersection with $\operatorname{WF}_h^{t,x}Q(u_h)\cap \{ t \in \operatorname{supp}\vartheta\}$.
\end{prop}

We also need a closely similar propagation result with an arbitrary source term.
\begin{prop}[Adapted from \cite{plaques}]\label{prop:source2}
Let $a<b$, $\vartheta\in C_c^\infty((a,b))$, and let $F_h$ be so that, for any bounded interval $I$ and every $j,m\in\mathbb N$, there exists $M \geq 0$ so that
\[
 \sup_{t\in I}\|\partial_t^jF_h(t)\|_{H^m(\Omega)}\lesssim h^{-M}.
\]
Let $u_h$ be solution to
\[
 \begin{cases}
 (i\partial_t+h\Delta_D)u_h=\vartheta(t)F_h,\\
 u_h|_{\partial\Omega}=0,\qquad u_h|_{t\leq a}=0.
 \end{cases}
\]
Then, if $\rho_0=(x_0,\xi_0,t_0,\tau_0)\in
{}^b\operatorname{WF}_h^{t,x}u_h \cap \pi_b (\{ \tau+|\xi|^2 =0 \})$ with $\tau_0 <0$, we have
\[
 \exists s\in\operatorname{supp}\vartheta,\quad s\leq t_0,\qquad
 \gamma_{s-t_0}(\rho_0)\in{}^b\operatorname{WF}_h^{t,x}F_h.
\]
\end{prop}

\begin{proof}
By a contradiction argument, it suffices to show the result for dyadic sequences $h_{j+1}\leq \frac{h_j}{2}$.
We follow \cite[Lemma~B.7]{plaques} (which can be simplified here since we don't need a quantitative statement).
Denote $G_h : =\vartheta F_h$ and let
\[
 \Theta(u)(s,t,x):=\sum_j e^{is/h_j}u_{h_j}(t,x),\qquad
 \Theta(G)(s,t,x)=\sum_j e^{is/h_j}G_{h_j}(t,x),
\]
which satisfy
\[
\begin{cases}
(\partial_s \partial_t + \Delta_D)\Theta(u)=-i\partial_s\Theta(G),\\
\Theta(u)|_{\partial\Omega}=0,\qquad
\Theta(u)|_{t\leq a}=0.
\end{cases}
\]
The correspondence between ${}^b\operatorname{WF}^{t,x}_hu_h$, ${}^b\operatorname{WF}^{t,x}_hG_h$ and the classical wavefront-sets of $ \Theta(u)$, $ \Theta(G)$, stated as
\cite[Lemma~B.5]{plaques}, \cite[Section 3, Lemma 1]{Lebeau}, still valid in our case with the same proof thanks to the polynomial bound in assumption, together with 
the Melrose--Sj\"ostrand propagation
theorem \cite{MR0492794, MR644020} with a source, give the result.
\end{proof}

\subsection{Quantitative propagation of singularities for the Helmholtz equation}

We need some quantitative version of the propagation of singularities for the Helmholtz equation in a semiclassical setting with boundaries in the proof of Theorem \ref{th:res}.

To present this result, we need a slightly more quantitative version of wavefront-set:
\begin{defn*}
If $v$ is an $h$-tempered family of $L^2_{\rm loc}(\overline \Omega)$ functions, we say that $\rho_0 \in {}^bT^* \Omega$ does \emph{not} belongs to the $b$-$\sigma$--oscillatory wavefront-set of $v$, denoted
$^b \operatorname{OF}^\sigma_h v$, if there exists $q \in C^\infty_c(^b T^* \Omega)$, compactly supported, so that $|q(\rho_0)| > 0$ and there exists $C>0$ so that
$$
\Vert \operatorname{Op}_h(q) v \Vert_{L^2} \leq C h^\sigma.
$$
We then write $\rho_0 \notin {}^b \operatorname{OF}^\sigma_h$. 

An analogous object is defined as a subset of $^bT^* (\Omega\times \mathbb R)$ for $h$-family of $L^2_{\rm loc}(\overline \Omega\times \mathbb R)$ functions.
\end{defn*}

The following statement is  well known in greater generality in the case without boundary. In our setting with boundary in the semiclassical setting, we derive it as a consequence of the quantitative propagation Theorem for the semiclassical
Schr\"odinger equation \cite[Th\'eor\`eme~B.6]{plaques}.
\begin{prop} \label{prop:prop_sing_qt}
Let $\sigma \geq 0$. Assume that $u_h$ is a $h$-tempered family of $L^2_{\rm loc}$ functions, satisfying
$$
\begin{cases}
(- h^2 \Delta - 1) u_h = \chi f_h \text{ in } \Omega, \\
u = 0 \text{ on } \partial \Omega
\end{cases}
$$
where $\chi \in C^\infty_c$ and $\Vert f_h \Vert_{L^2} = O(h^{\sigma + 1})$.
Then $^b \operatorname{OF}_h^\sigma u_h$ is invariant by the bicaracteristic flow.
\end{prop}
\begin{proof}
Let $\rho \in {}^bT^*\Omega$ be arbitrary. We will prove the invariance along an arbitrary fixed (broken) segment $\gamma$ of the generalised bicharacteristic from $\rho$:
$$
\gamma = \{ \phi_s(\rho), \; s\in [0, S]\}
$$
Let now $T>S$. We first cut-off $u_h$ in an arbitrary big compact set, in order to have a polynomially bounded function in $L^2$. Let $\chi_0 \in C^\infty_c(\mathbb R^3)$ equal to one near $\pi_x \gamma$, and $R>0$ sufficiently large that
$$
B(0,R-2) \supset \operatorname{supp}\chi \cup \operatorname{supp}\chi_0 \cup \mathcal O,
$$
and $\chi_1\in C^\infty_c(\mathbb R^3)$ equal to one on $B(0,R)$ and supported in $B(0,R+1)$. We let
$$
\widetilde u_h := \chi_1 u_h.
$$
Observe that
\begin{equation}\label{eq:HP1}
(-h^2\Delta_D - 1)\widetilde u_h = \chi f_h - [h^2\Delta, \chi_1] u_h.
\end{equation}
By the temperedness assumption (and elliptic regularity),
\begin{equation}\label{eq:HP2}
\Vert \widetilde u_h \Vert_{L^2} + \Vert [h^2\Delta, \chi_1] u_h\Vert_{L^2} = O(h^{-\kappa})
\end{equation}
for some $\kappa>0$

Observe now that it suffices to work near frequencies $\sim h^{-1}$: indeed, let $\psi_0 \in C^\infty_c((1/2, 2))$ equal to one near $1$. Observe that, as $|\lambda - 1| \geq c >0$ on the support of $1 - \psi_0$, we can write, by the functional calculus homomorphism 
\begin{align*}
\chi_0(1 - \psi_0)(-h^2\Delta_D) &= \chi_0\big(\frac{1-\psi_0(\lambda)}{\lambda - 1} (\lambda - 1)\big)(-h^2\Delta_D) \\
&=  \chi_0\Big[\big(\frac{1-\psi_0(\lambda)}{\lambda - 1} \big)(-h^2\Delta_D)\Big] \circ (- h^2 \Delta_D - 1),
\end{align*}
thus
\begin{align*}
\chi_0(1 - \psi_0)(-h^2\Delta_D) \widetilde u_h &= \chi_0 \big(\frac{1-\psi_0(\lambda)}{\lambda - 1} \big)(-h^2\Delta_D) (\chi f_h) \\
 &\quad - \chi_0 \big(\frac{1-\psi_0(\lambda)}{\lambda - 1} \big)(-h^2\Delta_D)([h^2\Delta,\chi_1]u_h).
\end{align*}
Since $\Big(\frac{1-\psi_0(\lambda)}{\lambda - 1} \Big)(-h^2\Delta_D)$ is uniformly bounded in $L^2$, the first term above is $O_{L^2}(h^{\sigma+1})$ by assumption, and as in addition $d(\operatorname{supp}\chi_0, \operatorname{supp}\nabla \chi_1)>0$, the second term is $O_{L^2}(h^\infty)$ by Lemma \ref{lem:offdiag_fc} below (together with (\ref{eq:HP2})).  Therefore we get 
\begin{equation}\label{eq:HP3}
\Vert \chi_0(1 - \psi_0)(-h^2\Delta_D) \widetilde u_h \Vert_{L^2} \lesssim h^{\sigma+1}.
\end{equation}
Hence, it suffices to deal with $\psi_0(-h^2\Delta_D) \widetilde u_h$. 

Now, define the $h$-family of space-time functions
$$
v_h(t) := e^{-it/h} \psi_0(-h^2\Delta_D) \widetilde u_h.
$$
Our motivation introducing $v_h$ is that, directly from the definition,
\begin{equation}\label{eq:HP35}
(x,\xi,t,-1)\in{}^b\operatorname{OF}_h^\sigma(e^{-it/h}u_h)
\quad\Longleftrightarrow\quad
(x,\xi)\in{}^b\operatorname{OF}_h^\sigma u_h.
\end{equation}
The function $v_h$ satisfies the following semiclassical Schr\"odinger equation with Dirichlet boundary conditions
$$
\begin{cases}
i h\partial_t v_h + h^2\Delta_D v_h  =  - e^{-it/h}  \psi_0(-h^2\Delta_D)\chi  f_h + e^{-it/h}\psi_0(-h^2\Delta_D) [h^2\Delta, \chi_1]u_h, \\
v_h = 0 \text{ on } \partial \Omega, \\
v_h(t=0) = \psi_0(-h^2\Delta_D) \widetilde u_h.
\end{cases}
$$
We decompose
\begin{equation}\label{eq:HP4}
v_h = v_h^1 + v_h^2, \quad v^2_h := e^{ith\Delta_D} \psi_0 (-h^2\Delta_D) \widetilde u_h.
\end{equation}
By Duhamel formula,
\begin{align*}
v_h^1(t)
&=\frac{i}{h}\int_0^t e^{i(t-s)h\Delta_D}e^{-is/h}
       \psi_0(-h^2\Delta_D)\chi f_h\,ds\\
&\quad-\frac{i}{h}\int_0^t e^{i(t-s)h\Delta_D}e^{-is/h}
       \psi_0(-h^2\Delta_D)[h^2\Delta,\chi_1]u_h\,ds.
\end{align*}
The first integral is $O_{L^2}(h^\sigma)$ uniformly for $|t|\leq T$ using the assumption.
For the second, choose $\chi_2\in C_c^\infty(\mathbb R^d)$ equal to
one near $\operatorname{supp}\nabla\chi_1$ and supported in
$\{R-1<|x|<R+2\}$.
Taking $R$ sufficiently with respect to $T$ and
$\operatorname{supp}\chi_0$, the (dual estimate of) the approximate finite speed of propagation of Lemma~\ref{lem:semi_speed} gives
\[
\sup_{|s|\leq T}
\big\|\chi_0e^{ish\Delta_D}\psi_0(-h^2\Delta_D)\chi_2
\big\|_{L^2\to L^2}=O(h^\infty),
\]
and therefore,
\begin{equation}\label{eq:HP5}
\sup_{|t|\leq T}\|\chi_0v_h^1(t)\|_{L^2}=O(h^\sigma).
\end{equation}

Combining (\ref{eq:HP5}) with (\ref{eq:HP4}) and (\ref{eq:HP3}) above gives
\[
\sup_{|t|\leq T}
\|\chi_0(v_h^2(t)-e^{-it/h}u_h)\|_{L^2}
=O(h^\sigma).
\]
Thus, with $V$ be an open neighbourhood of $\pi_x\gamma$ in
$\overline\Omega$ on which $\chi_0=1$,
\begin{equation}\label{eq:HP6}
{}^b\operatorname{OF}_h^\sigma(v_h^2)
 \cap {}^bT^*(V\times(-T,T))
=
{}^b\operatorname{OF}_h^\sigma(e^{-it/h}u_h)
 \cap {}^bT^*(V\times(-T,T)).
\end{equation}
Combining (\ref{eq:HP6}) with (\ref{eq:HP35}), we get, for any $t\in(-T,T)$
$$
\forall(x,\xi) \in {}^bT^*V,\quad (x,\xi)\in{}^b\operatorname{OF}_h^\sigma u_h
\iff (x,\xi,t,-1) \in{}^b\operatorname{OF}_h^\sigma v^2_h.
$$
But by \cite[Th\'eor\`eme~B.6]{plaques}, ${}^b\operatorname{OF}_h^\sigma(v_h^2)$ is invariant under the generalized bicharacteristic
flow, which ends the proofs. (Observe that \cite[Th\'eor\`eme~B.6]{plaques} is stated for bounded $L^2$ families, but polynomially bounded families can also be treated by
applying the statement for bounded families to $h^\kappa v_h^2$
at order $\sigma+\kappa$ since $
{}^b\operatorname{OF}_h^{\sigma+\kappa}(h^\kappa v_h^2)
={}^b\operatorname{OF}_h^\sigma v_h^2$).
\end{proof}

\subsection{Approximate semiclassical finite speed of propagation}

The following Lemma is in the spirit of
the semiclassical finite speed of propagation established by \cite{Lebeau}. We haven't been able to find such a statement in the literature, so we give it for completeness.

The proof of Lemma \ref{lem:semi_speed}, as well as of Lemmas \ref{lem:psi_away}, \ref{lem:offdiag_fc}, \ref{lem:short_free_comparison} below, use the combination of the  finite speed of propagation for the wave equation, with the formula, valid for an even function $f(\lambda)$
\begin{equation}\label{eq:cosines}
f(h\sqrt{-\Delta_D}) = \frac1{2\pi h} \int_{\mathbb R} \int_{\mathbb R}
 e^{is\lambda/h} f(\lambda) \cos(s{\sqrt{-\Delta_D}}) \, d\lambda ds,
\end{equation}
which follows from Fourier inversion for even functions and the spectral theorem,
see for example \cite{Grom}, \cite{Rem}, and references therein for this approach. 

\begin{lem} \label{lem:semi_speed}
Let $E>0$, $\psi \in C^\infty_c(\mathbb R)$ be supported in $[0, E]$, 
and $v > 2\sqrt E$. Then, for any $R,d>0$, 
\begin{equation}\label{eq:tail}
 \sup_{|t| \leq d/v} \bigl\|\mathbf{1}_{|x|\geq R+d}\,e^{ith\Delta_D}\psi(-h^2\Delta_D)\mathbf{1}_{|x|\leq R}
       \bigr\|_{L^2\to L^2} = O(h^\infty).
\end{equation}
The same holds in $L^2 \mapsto H^\sigma$ for any $\sigma \geq 0$ replacing $\mathbf 1_{\{|x|\geq R+d\}}$ by any $\chi_{\rm out} \in C^\infty(\mathbb R^3)$ supported in $\{|x|\geq R+d\}$ with all derivatives
bounded.
\end{lem}

\begin{proof} The following starts as \cite[Proposition 1.1]{Grom} and conclude by non-stationary phase.
As $\lambda\mapsto e^{-it\lambda^2/h}\psi(\lambda^2)$ is even, we can write from (\ref{eq:cosines})
$$
e^{ith\Delta_D}\psi(-h^2\Delta_D) = \frac1{2\pi h} \int_{\mathbb R} \int_{\mathbb R}
 e^{i(s\lambda-t\lambda^2)/h}\psi(\lambda^2) \cos(s\sqrt{-\Delta_D}) \, d\lambda ds.
$$
Now, by finite speed of propagation for the wave equation,
$$
|s|<d \implies \mathbf{1}_{|x|\geq R+d}\cos(s\sqrt{-\Delta_D})\mathbf{1}_{|x|\leq R} = 0.
$$
So we are left with
\begin{multline*}
\mathbf{1}_{|x|\geq R+d}\,e^{ith\Delta_D}\psi(-h^2\Delta_D)\mathbf{1}_{|x|\leq R} \\ = \frac1{2\pi h} \int_{|s|\geq d } \int_{\mathbb R}
 e^{i(s\lambda-t\lambda^2)/h}\psi(\lambda^2) \mathbf{1}_{|x|\geq R+d}\cos(s\sqrt{-\Delta_D})\mathbf{1}_{|x|\leq R} \, d\lambda ds.
\end{multline*}
We write, for $s-2t\lambda \neq 0$,
$$
 e^{i(s\lambda-t\lambda^2)/h} = \mathcal L  e^{i(s\lambda-t\lambda^2)/h},\qquad \mathcal L := \frac{h}{i(s-2t\lambda)} \partial_\lambda.
$$
Now, denote
$$
\frac{h}{(s-2t\lambda)} = \frac hs b_{t,s}(\lambda),\qquad b_{t,s}(\lambda) := \frac{1}{1-2\frac ts \lambda},
$$
and observe that, for $|s|\geq d \geq v|t|$ and $\lambda\in\operatorname{supp} \psi(\cdot^2) \subset[-\sqrt E, \sqrt E]$, we have the uniform control
$$
|\partial^k_\lambda b_{t,s}(\lambda)| \lesssim \left|\frac{t}{s}\right|^k \frac 1{|1-2\frac ts \lambda|^{k+1}} \lesssim_k v^{-k} \bigl(1-2\sqrt E/v\bigr)^{-k-1}.
$$
Thus, integrating $N+1$ times by part, we get for $|s|\geq d \geq v t$,
$$
\Big|\int_{\mathbb R}
 e^{i(s\lambda-t\lambda^2)/h}\psi(\lambda^2) \, d\lambda\Big| = \Big|\int_{\mathbb R}
 e^{i(s\lambda-t\lambda^2)/h} (\mathcal L^*)^{N+1} \psi(\lambda^2) \, d\lambda\Big| \lesssim_N h^{N+1}{|s|^{-N-1}}.
$$
We conclude that, for $d \geq v|t|$
$$
\Vert \mathbf{1}_{|x|\geq R+d}\,e^{ith\Delta_D}\psi(-h^2\Delta_D)\mathbf{1}_{|x|\leq R} \Vert_{L^2 \to L^2} \lesssim_N
h^{N} \int_{|s|\geq d}|s|^{-N-1}  ds \lesssim_N \left(\frac{h}{d}\right)^N,
$$
and (\ref{eq:tail}) follows.
Finally, the $L^2 \mapsto H^\sigma$ claim is obtained by interpolation with $\Vert e^{ith\Delta_D}\psi(-h^2\Delta_D) \Vert_{L^2 \to H^\sigma}
 \lesssim_\sigma h^{-\sigma}$.
\end{proof}

\subsection{Semiclassical functional calculus}
The two Lemma presented here are very close to \cite[Section~4, Lemmas~4.1--4.2]{Sj:97}.
We give direct proofs, close to \cite[Theorems~3.3 and~4.2]{Rem}, to obtain the versions used in this paper.

We first show that away from the boundary, the semiclassical functional calculus coincides with the Fourier multiplier up to a $O(h^\infty)$ reminder. Note that, for compactly supported cut-offs, this is covered by (the Euclidian version of) \cite[Proposition 2.1]{BGT_STR}. However, in Lemma \ref{lem:psi_away} below, we don't need to take $\chi_1, \chi_2$ compactly supported: they can for example be equal to one outside a ball.
\begin{lem}\label{lem:psi_away}
Let $\chi_1,\chi_2\in C^\infty(\mathbb R^3)$ be bounded with all their derivatives uniformly bounded. Assume that $\chi_1, \chi_2$ vanish in
a neighborhood of the obstacles and satisfy
$$
 \operatorname{dist}\bigl(\operatorname{supp}\chi_1,
                  \operatorname{supp}(1-\chi_2)\bigr)>0.
$$
Then, for any $\sigma\geq0$,
\begin{equation}\label{eq:psi_away_global}
 \psi(-h^2\Delta_D)\chi_1
 =\chi_2\operatorname{Op}_h\bigl(\psi(|\xi|^2)\bigr)\chi_1
 +O(h^\infty)_{H^{-\sigma}(\Omega)\to H^\sigma(\Omega)}.
\end{equation}
\end{lem}

\begin{proof}
As in Lemma~\ref{lem:semi_speed}, we write from (\ref{eq:cosines})
\[
 \psi(-h^2\Delta_D)
 =\frac1{2\pi h}\int_{\mathbb R}\int_{\mathbb R}
 e^{is\lambda/h}\psi(\lambda^2)
 \cos(s\sqrt{-\Delta_D})\,d\lambda\,ds,
\]
and similarly for $\psi(-h^2\Delta_{\mathbb R^3}) = \operatorname{Op}_h\bigl(\psi(|\xi|^2)\bigr)$.
Take
\[
 0<d<\operatorname{dist}\bigl(\operatorname{supp}\chi_1,
                         \operatorname{supp}(1-\chi_2)\bigr).
\]
By finite speed of propagation for the wave equation,
\begin{equation}\label{eq:psi_away_wave}
 \cos(s\sqrt{-\Delta_D})\chi_1
 =\chi_2\cos(s\sqrt{-\Delta_{\mathbb R^3}})\chi_1,
 \qquad |s|<d.
\end{equation}
Therefore, we need to estimate
\begin{align*}
 E:=&\psi(-h^2\Delta_D)\chi_1
  - \chi_2\operatorname{Op}_h\bigl(\psi(|\xi|^2)\bigr)\chi_1 \\
 =&  \frac1{2\pi h}\int_{|s|\geq d}k_h(s) \Bigl(\cos(s\sqrt{-\Delta_D})\chi_1
       -\chi_2\cos(s\sqrt{-\Delta_{\mathbb R^3}})\chi_1\Bigr)\,ds,
\end{align*}
where
\[
 k_h(s):=\int_{\mathbb R}e^{is\lambda/h}\psi(\lambda^2)\,d\lambda.
\]
For $|s|\geq d >0$, integrating $N+1$ times by parts gives
$$
 |k_h(s)|
 \lesssim_N h^{N+1}|s|^{-N-1}.
 \label{eq:psi_away_scalar}
$$
Thus we get
$$
 \bigl\| E
 \bigr\|_{L^2\to L^2}
 \lesssim_N h^N\int_{|s|\geq d}|s|^{-N-1}\,ds
 \lesssim_N (h/d)^N,
$$
which gives the claim for $\sigma = 0$.

Finally, to deal with the case $\sigma >0$, it suffices to interpolate with the bound 
$
 \|E\|_{H^{-2m}(\Omega)\to H^{2m}(\Omega)}
 \lesssim_m h^{-4m},
$
valid for any integer $m$ by elliptic regularity.
\end{proof}

In the same vein, we now quantify the pseudo-locality of the semiclassical functionnal calculus.

\begin{lem}\label{lem:offdiag_fc}
Let $\chi_1,\chi_2\in C_c^\infty(\mathbb R^3)$ be so that
$d(\operatorname{supp}\chi_1,
                     \operatorname{supp}\chi_2)>0$.
                     Then, 
for every $N\geq0$ and $\sigma\geq0$, there exist $L \geq 0$ and
$C_{N,\sigma}>0$, such that
for any $f \in C^\infty(\mathbb R)$ so that the right-hand-side is finite,
\[
\|\chi_1f(-h^2\Delta_D)\chi_2\|_{L^2\to H^\sigma}
\leq \max_{0\leq j\leq L}\sup_{\lambda\geq0}
(1+\lambda)^{j+1}|f^{(j)}(\lambda)| C_{N,\sigma}h^N.
\]
\end{lem}

\begin{proof}
It suffices to consider the case of integer $\sigma = 2m$.
Let $d := d(\operatorname{supp}\chi_1, \operatorname{supp}\chi_2)>0$. Let $\eta \in C^\infty(\mathbb R)$ so that $ \operatorname{supp}\eta \subset [-d/2, d/2]$ and $\eta = 1$ near $0$.
By (\ref{eq:cosines}) and finite speed of propagation, as in the
proof of Lemma~\ref{lem:psi_away} above, we can write
\begin{align*}
\chi_1f(-h^2\Delta_D)\chi_2
&= \frac{1}{2\pi h}\int_{\mathbb R}\int_{\mathbb R}
e^{is\lambda/h}f(\lambda^2)\,\chi_1 \cos(s\sqrt{-\Delta_D})\chi_2\,d\lambda ds \\
&= \frac{1}{2\pi h}\int_{\mathbb R}\int_{\mathbb R}
e^{is\lambda/h}f(\lambda^2)\,(1-\eta(s))\chi_1 \cos(s\sqrt{-\Delta_D})\chi_2\,d\lambda ds.
\end{align*}
Thus, by boundedness of the multiplication by $\chi_1$ in $H^{2m}$ then elliptic regularity
\begin{align*}
&\|\chi_1f(-h^2\Delta_D)\chi_2\|_{L^2(\Omega)\to H^{2m}(\Omega)} \\
&\lesssim \Big\|\frac{1}{2\pi h}\int_{\mathbb R}\int_{\mathbb R}
e^{is\lambda/h}f(\lambda^2)\,(1-\eta(s)) \cos(s\sqrt{-\Delta_D})\chi_2\,d\lambda ds \Big\|_{L^2 \to H^{2m}} \\
&\lesssim \sum_{\ell=0}^{m} \Big\| \frac{1}{2\pi h}\int_{\mathbb R}\int_{\mathbb R}
e^{is\lambda/h}f(\lambda^2)\,(1-\eta(s)) (-\Delta_D)^\ell \cos(s\sqrt{-\Delta_D})\chi_2\,d\lambda ds\Big\|_{L^2 \to L^2}.
\end{align*}
But, using the fact that $-\Delta_D\cos(s\sqrt{-\Delta_D})=-\partial_s^2\cos(s\sqrt{-\Delta_D})$, and integrating by parts $2\ell$ times in $s$ we get
\begin{align*}
& \Big\| \frac{1}{2\pi h}\int_{\mathbb R}\int_{\mathbb R}
e^{is\lambda/h}f(\lambda^2)\,(1-\eta(s)) (-\Delta_D)^\ell \cos(s\sqrt{-\Delta_D})\chi_2\,d\lambda ds\Big\|_{L^2 \to L^2} \\
&\lesssim \frac{1}{2\pi h} \int \left|
\partial_s^{2\ell}\left[
(1-\eta(s))
\int_{\mathbb R}e^{is\lambda/h}f(\lambda^2)\,d\lambda
\right]\right| \, ds,
\end{align*}
and integrating by parts enough times with respect to $\lambda$ gives the result (recall that $\eta$=1 near $0$).
\end{proof}

\subsection{Elliptic regularity}

The following permits to go from empty b-wavefront-set to $O(h^\infty)$ error. It is an adaptation of \cite[Section 2]{Lebeau}.

\begin{lem} \label{ref:ell_space}
Let $u_h \in L^2(\Omega)$ be polynomially bounded and so that ${}^b\operatorname{WF}^{x}_h u_h = \emptyset$. Then, for any $\psi \in C^\infty_c(\mathbb R_+^*)$, $\chi\in C^\infty_c(\mathbb R^3)$,  in any $\sigma \geq 0$,
$$
\Vert \chi \psi(-h^2 \Delta_D) u_h \Vert_{H^\sigma} = O(h^\infty) \text{ and } \Vert  \psi(-h^2 \Delta_D) \chi u_h  \Vert_{H^\sigma} = O(h^\infty). 
$$ 
\end{lem}
\begin{proof}
Observe that, by interpolation with $\Vert \psi(-h^2 \Delta_D) v_h \Vert_{H^m},  \lesssim h^{-m} \Vert v_h \Vert_{L^2}$, it suffices to establish the claim for $\sigma = 0$.

By \cite[Lemma 2]{Lebeau}, there exists $b\in C^\infty_c({}^b T^*\Omega)$ so that
$$
(1-\operatorname{Op}_h(b))\chi\psi(-h^2\Delta_D) = O(h^\infty)_{L^2 \to L^2}
$$
(observe that although \cite{Lebeau} works in a bounded domain, we can extend this result to our unbounded setting by the cosine expansion and finite speed of propagation, in the spirit of the proofs of Lemmas \ref{lem:semi_speed}, \ref{lem:psi_away},  \ref{lem:offdiag_fc}). Using the above, the second claim follows by definition of the $b$-wavefront set, by taking the adjoint and writing
$$
\psi(-h^2\Delta_D)\chi u_h = \psi(-h^2\Delta_D)\chi \operatorname{Op}_h(b_h)^* u_h + O(h^\infty)_{L^2},
$$
as $\operatorname{Op}_h(b_h)^*$ can be represented by an operator with compact symbol up to a $O(h^\infty)$ error.
To establish the first claim, take  $\chi_1\in C_c^\infty(\mathbb R^3)$ equal to one near $\operatorname{supp}\chi$, and observe that by Lemma \ref{lem:offdiag_fc}
$$
\chi\psi(-h^2\Delta_D)u_h =\chi\psi(-h^2\Delta_D)\chi_1u_h + O(h^\infty)_{L^2 \to L^2},
$$
then the result follows similarly.
\end{proof}

We also need the following elliptic regularity result for the Helmholtz equation.

\begin{lem}\label{lem:ell_b}
Let 
$$
P_h:=-h^2\Delta_D-1, \quad \Sigma:=\pi_b\big(\{(x,\xi), \; |\xi|^2=1\}\big).
$$
Let $a\in C_c^\infty({}^bT^*\Omega)$ be so that
$\operatorname{supp}a\cap\Sigma=\emptyset$. Then,
there exists $\chi\in C_c^\infty(\mathbb R^d)$, equal to one near
$\pi_x\operatorname{supp}a$, such that, for any $N\geq0$ and any family $u_h \in H^1_{0, {\rm loc}}(\Omega)$
\begin{equation}\label{eq:elliptic_b_quantitative}
 \|\operatorname{Op}_h(a)u_h\|_{L^2}
 \leq C_N\bigl(\|\chi P_hu_h\|_{L^2}
                  +h^N\|\chi u_h\|_{L^2}\bigr).
\end{equation}
\end{lem}
\begin{proof}[References for proof]
This can be obtained for example as the semiclassical version of \cite[Proposition 4.6]{Vasy}, the analog of $\int_X (|d_M v|^2 - |D_t v|^2)$ being $\int_{\Omega} |h \nabla v|^2 - |v|^2$, and where each iteration permits to gain a power of $h$.
\end{proof}

\subsection{Time-traces for data away from boundary in propagation of singularities} \label{app:time_traces}

The aim of this subsection is to show the following result. 

\begin{lem} \label{lem:traces_away_bdd}
Let $\psi \in C^\infty_c(\mathbb R_+^*)$,
and $b \in C^\infty_c(T^* \mathbb R^3)$, with spatial support contained in $\Omega$, so that  
$$
d(\pi_x \operatorname{supp}b, \partial \Omega) \geq c >0.
$$
Then, for any polynomially bounded family $u_{0,h}\in L^2(\Omega)$ and any $t_0 >0$
$$
{}^b\operatorname{WF}^x_h\big( e^{it_0 h\Delta_D} \psi(-h^2 \Delta_D) \operatorname{Op}_h (b) u_{0,h} \big) \subset \phi_{t_0} \big(\operatorname{supp} b \cap \{ |\xi|^2 \in \operatorname{supp}\psi\}\big).
$$
\end{lem}

In order to prove Lemma \ref{lem:traces_away_bdd}, we will need the following pull-back result for semiclassical wavefront-sets.

\begin{lem}\label{lem:WFpullback}
Let $u_h$ be a $h$-tempered family of $L^2_{\rm loc}(\mathbb R \times \Omega)$ and $I \subset \mathbb R$ a finite open time interval so that the following ``ellipticity at $\tau = \infty$'' condition holds: there exists a function $\varphi \in C^\infty_c(\mathbb R)$, equal to one in a compact set, so that
\begin{equation}\label{eq:ell_inft_cond}
\forall \sigma, \hspace{0.3cm} \Vert (1 - \varphi)(hD_t)u_h \Vert_{H^1(I, L^2(\Omega))} = O(h^\infty).
\end{equation}
Then, for any $t_0 \in I$
$$
{}^b \operatorname{WF}^x_h (u(t_0)) \subset \big\{ (x, \xi)\in {}^b T^* \Omega \hspace{0.2cm} \text{s.t.} \hspace{0.2cm} \exists \tau, \;(x, t_0, \xi, \tau) \in {}^b \operatorname{WF}^{t,x}_h u \big\}.
$$
\end{lem}
\begin{proof} 
With $t_0 \in I$ fixed, let $(x_0, \xi_0) \in {}^b T^* \Omega$ be so that
$$
\forall \tau \in \mathbb R, \hspace{0.3cm}(x_0, t_0, \xi_0, \tau) \notin  {}^b \operatorname{WF}_h^{t,x} u.
$$
We want to show that $(x_0, \xi_0) \notin {}^b\operatorname{WF}_h^x (u(t_0))$. As the complementary of ${}^b \operatorname{WF}^{t,x}_h u$ is open, for any fixed $\tau$, we have $(x, t, \xi, \tau) \notin  {}^b \operatorname{WF}^{t,x}_h u$ for $(x, \xi, t)$ in a neighbourhood of $(x_0, \xi_0, t_0)$. Thus, as $\varphi$ is compactly supported, we can 
construct
$q\in C_c^\infty({}^bT^*\Omega)$, equal to one near
$(x_0,\xi_0)$, and $\vartheta\in C_c^\infty(I)$, equal to
one near $t_0$, such that
$$
\operatorname{supp}
 \bigl(q(x,\xi)\vartheta(t)\varphi(\tau)\bigr)
 \cap{}^b\operatorname{WF}^{t,x}_h u_h=\emptyset.
$$
The definition of the $b$-wavefront set gives
$$
\vartheta\operatorname{Op}_h(q)\varphi(hD_t)u_h
=O(h^\infty)_{L^2(\mathbb R \times \Omega)}.
$$
In addition,
\begin{align*}
\partial_t\bigl(
\vartheta(t)\operatorname{Op}_h(q)\varphi(hD_t)u_h
\bigr)
&=\vartheta'(t)\operatorname{Op}_h(q)\varphi(hD_t)u_h +\frac{i}{h}\vartheta(t)\operatorname{Op}_h(q)
(\tau\varphi(\tau))(hD_t)u_h
\\
& = O(h^\infty)_{L^2(\mathbb R \times \Omega)},
\end{align*}
as the symbols of both operators on the right are supported in
$\operatorname{supp}(q(x,\xi)\vartheta(t)\varphi(\tau))$,
hence away from ${}^b\operatorname{WF}^{t,x}_h u_h$.
It follows that
\[
\vartheta\operatorname{Op}_h(q)\varphi(hD_t)u_h
=O(h^\infty)_{H^1(\mathbb R,L^2(\Omega))},
\]
which, combined with our assumption (\ref{eq:ell_inft_cond}), gives,
\[
\vartheta\operatorname{Op}_h(q)u_h
=O(h^\infty)_{H^1(\mathbb R,L^2(\Omega))},
\]
hence, by the embedding $H^1(\mathbb R, L^2(\Omega))
\hookrightarrow C^0(\mathbb R,L^2(\Omega))
$, we get
\[
\operatorname{Op}_h(q)u_h(t_0)
=O(h^\infty)_{L^2(\Omega)},
\]
and the result follows.
\end{proof}

We will also need the following short-time comparison with the free flow for data away from the boundary.

\begin{lem}\label{lem:short_free_comparison}
Let $\psi\in C_c^\infty(\mathbb R_+^*)$ and
$\chi_1, \chi_2 \in C_c^\infty(\mathbb R^3)$ 
satisfy
\[
d(\operatorname{supp}\chi_{j},\partial\Omega)>0,\quad j=1,2,
\]
with $\chi_2 = 1$ near $\operatorname{supp}\chi_1$.
Then, there exists $\epsilon_0>0$ such that
\[
\sup_{|t|\leq\epsilon_0}
\left\|
e^{ith\Delta_D}\psi(-h^2\Delta_D)\chi_1 
-
\chi_2e^{ith\Delta_{\mathbb R^3}}
\psi(-h^2\Delta_{\mathbb R^3})\chi_1 
\right\|_{L^2(\Omega) \to L^2(\Omega)} = O(h^\infty).
\]
\end{lem}

\begin{proof}
Let $\delta >0$ so that
\[
0<\delta<\min\big(d(\operatorname{supp}\chi_1,\partial\Omega), d(\operatorname{supp}\chi_1, \operatorname{supp}(1-\chi_2))\big).
\]
By finite speed of propagation for the wave equation,
\[
\cos(s\sqrt{-\Delta_D})\chi_1 
=
\chi_2\cos(s\sqrt{-\Delta_{\mathbb R^3}})\chi_1 ,
\qquad |s|<\delta.
\]
Thus, by using (\ref{eq:cosines}) and finite speed of propagation as in 
Lemma~\ref{lem:semi_speed}, the difference we want to estimate is bounded in $L^2(\Omega)$ by
\[
\frac{C}{h}
\int_{|s|\geq\delta}
\left|
\int_{\mathbb R}
e^{i(s\lambda-t\lambda^2)/h}\psi(\lambda^2)\,d\lambda
\right|\,ds.
\]
Now, let $E>0$ so that $\operatorname{supp}\psi\subset[0,E]$,
and let $\epsilon_0>0$ so small that $4\epsilon_0\sqrt E<\delta$.
For $|t|\leq\epsilon_0$, $|s|\geq\delta$, and
$\lambda^2\in\operatorname{supp}\psi$, we have
\[
|s-2t\lambda|\geq |s| - 2 \epsilon_0 \sqrt E \geq |s| - \frac \delta 2 \geq \frac{|s|}{2}.
\]
Thus, exactly as in Lemma~\ref{lem:semi_speed}, integrating
$N+1$ times by parts in $\lambda$ using
$$
e^{i(s\lambda-t\lambda^2)/h} = \frac{h}{i(s-2t\lambda)} \partial_\lambda e^{i(s\lambda-t\lambda^2)/h}
$$
gives, uniformly in $|t|\leq\epsilon_0$,
\[
\left|
\int_{\mathbb R}
e^{i(s\lambda-t\lambda^2)/h}\psi(\lambda^2)\,d\lambda
\right|
\leq C_Nh^{N+1}|s|^{-N-1},
\]
And the result follows by integrating in $s$.
\end{proof}

\begin{proof}[Proof of Lemma \ref{lem:traces_away_bdd}]

Let us denote
$$
u_h(t) := e^{it h\Delta_D} \psi(-h^2 \Delta_D) \operatorname{Op}_h (b) u_0
$$

Let $\varphi\in C_c^\infty(\mathbb R)$ equal to one near
$-\operatorname{supp}\psi$, in such a way that
\begin{equation}\label{eq:tr_time_loc}
\varphi(hD_t)u_h=u_h.
\end{equation}

For convenience, we will order the variables or ${}^b T^* (\Omega \times \mathbb R)$ as $(x,\xi,t,\tau)$.
We first show that
\begin{equation}\label{eq:tr_initial}
{}^b\operatorname{WF}_h^{t,x}u_h\cap\{t=0\}
\subset
\bigl\{(x,\xi,0,\tau):
(x,\xi)\in\operatorname{supp}b,\quad
|\xi|^2\in\operatorname{supp}\psi\bigr\}.
\end{equation}
This will follow by comparison with the \emph{free} flow in small times.
Indeed, as $b$ is supported away from the boundary (inserting a cut-off $\chi_1 \in C^\infty_c(\mathbb R^3)$ supported away from the boundary on the left of $\operatorname{Op}_h b$ at an $O(h^\infty)$ cost), by Lemma \ref{lem:short_free_comparison}, there is $\epsilon_0 > 0$ small enough and $\chi_2 \in C^\infty_c(\mathbb R^3)$, supported away from the boundary, so that
 \begin{equation}\label{eq:tr_comp}
u(t) = \chi_2 w(t)
+ O(h^\infty)\quad{ in }\; L^\infty([-\epsilon_0, \epsilon_0], L^2(\Omega)),
\end{equation}
where
$$
w(t) = e^{it h\Delta_{\mathbb R^3}} \psi(-h^2\Delta_{\mathbb R^3}) \operatorname{Op}_h (b) u_0.
$$

Let now
$$
\rho_0 = (x_0, \xi_0)\notin
\operatorname{supp}b\cap
\{|\xi|^2\in\operatorname{supp}\psi\}.
$$
If $x_0 \in \operatorname{supp}\chi_2$, in particular $x_0$ is in the interior, and
by Egorov for the free flow $e^{ith\Delta_{\mathbb R^3}}$
there exist $0<\epsilon_1<\epsilon_0$ and
$c\in C_c^\infty({}^bT^*\Omega)$, equal to one near $\rho_0$,
such that
\begin{equation}\label{eq:Eg}
\sup_{|t|\leq\epsilon_1}
\|\operatorname{Op}_h(c)\chi_2w(t)\|_{L^2(\Omega)}
=O(h^\infty).
\end{equation}
On the other hand, if $x_0 \notin \operatorname{supp}\chi_2$, for $c\in C_c^\infty({}^bT^*\Omega)$, equal to one near $\rho_0$ with sufficiently small support around $\rho_0$, $\operatorname{Op}_h(c)\chi_2 = O(h^\infty)$ in $L^2\to L^2$, and (\ref{eq:Eg}) holds as well. Now, combining (\ref{eq:Eg})  with (\ref{eq:tr_comp}), we get
\begin{equation}\label{tr:in__}
\sup_{|t|\leq\epsilon_1}
\|\operatorname{Op}_h(c)u(t)\|_{L^2(\Omega)}
=O(h^\infty).
\end{equation}
Choose $\vartheta\in C_c^\infty((-\epsilon_1,\epsilon_1))$
equal to one near zero. Observe that, from (\ref{tr:in__}),
\begin{align*}
\operatorname{Op}_h
\bigl(c(x,\xi)\vartheta(t)\varphi(\tau)\bigr)u_h
&=\vartheta(t)\operatorname{Op}_h(c)
  \varphi(hD_t)u_h\\
&=\vartheta(t)\operatorname{Op}_h(c)u_h
 =O(h^\infty)_{L^2(\mathbb R\times\Omega)}.
\end{align*}
It follows that for any $\tau_0$ such that $\varphi(\tau_0)=1$, $(x_0, \xi_0, 0,\tau_0) \not\in {}^b\operatorname{WF}_h^{t,x}u_h$. As on the other hand, by (\ref{eq:tr_time_loc})  
${}^b\operatorname{WF}_h^{t,x}u_h
\subset \{ \varphi(\tau) = 1 \}$, (\ref{eq:tr_initial}) follows.

Now, by the invariance of $\operatorname{WF}_h^{t,x} u_h$ by the flow $\gamma_s$ (Theorem \ref{th:prop_sch})
\[
{}^b\operatorname{WF}_h^{t,x}u_h\cap\{t=t_0\}
=
\gamma_{t_0}\bigl(
{}^b\operatorname{WF}_h^{t,x}u_h\cap\{t=0\}
\bigr),
\]
and combining the above with (\ref{eq:tr_initial}) gives
\begin{equation}\label{eq:trPB}
{}^b\operatorname{WF}_h^{t,x}u_h\cap\{t=t_0\}
\subset
\bigl\{(\phi_{t_0}(x,\xi),t_0,\tau):
(x,\xi)\in\operatorname{supp}b,\quad
|\xi|^2\in\operatorname{supp}\psi\bigr\}.
\end{equation}
Finally, from the pull-back Lemma \ref{lem:WFpullback} (the assumption being satisfied thanks to (\ref{eq:tr_time_loc})),
\begin{equation*}\label{eq:trPB_}
{}^b \operatorname{WF}^x_h (u(t_0)) \subset \pi_{x,\xi}\big({}^b\operatorname{WF}_h^{t,x}u_h\cap\{t=t_0\}\big),
\end{equation*}
and combining the above with (\ref{eq:trPB}) concludes the proof.
\end{proof}
}

\appendix
\addtocounter{section}{1}
\section{Regularity of the billiard flow}

\begin{lem}
\label{lem:hold12}Let $\Theta$ be a smooth, strictly convex compact
subset of $\mathbb{R}^{n}$, with no infinite order contact point,
and
\[
W(\Theta):=\left\{ (x,\xi),x\in\mathbb{R}^{n}\backslash\Theta,\xi\in\mathbb{R}^{n},\text{ s.t. }\exists t\geq0,x+t\xi\in\Theta\right\} .
\]
We denote by $t$ the application
\[
t:(x,\xi)\in W(\Theta)\longrightarrow\text{the smallest \ensuremath{t'} such that }x+t'\xi\in\partial\Omega,
\]
and by
\[
W_{\text{tan}}(\Theta)=\{(x,\xi)\in W\text{ s.t. }\xi\cdot n(x+t(x,\xi)\xi)=0\},
\]
the tangent rays. Then:
\begin{enumerate}
\item on $W(\Theta)\backslash W_{\text{tan}}(\Theta)$, $t$ is $C^{\infty},$
\item t is locally H\"{o}lder on $W(\Theta)$. 
\end{enumerate}
\end{lem}
\begin{proof}
Let $g\in C^{\infty}(\mathbb{R}^{n},\mathbb{R})$ be such that $\partial\Theta$
is given by the implicit relation $g(x)=0$\emph{. }We denote, for
$x,\xi\in\mathbb{R}^{n}$, $t\in\mathbb{R}$
\[
h(x,\xi,t)=g(x+t\xi).
\]

\emph{Away from the tangent rays:} we pick $(x_{1},\xi_{1})\in W\backslash W_{\tan}$.
Let $t_{1}=t(x_{1},\xi_{1})$. Then, $\partial_{t}h(x_{1},\xi_{1},t_{1})=\xi_{1}\cdot\nabla g(x_{1}+t_{1}\xi_{1})\neq0$
by definition of $W\backslash W_{\text{tan}}$. Therefore, by the
implicit functions theorem, $t$ is a $C^{\infty}$ function of $(x,\xi)$
for $(x,\xi)$ in a neighborhood of $(x_{1},\xi_{1})$.

\emph{General case: }Let $(x_{1},\xi_{1})\in W$ and $t_{1}=(x_{1},\xi_{1})$.
Then $h(x_{1},\xi_{1},t_{1})=0$. Let $k\geq1$ be the smallest integer
such that
\[
h((x_{1},\xi_{1}),t_{1})=0,\partial_{t}h((x_{1},\xi_{1}),t_{1})=0,\dots,\partial_{t}^{k}h((x_{1},\xi_{1}),t_{1})\neq0.
\]
Note that $k$ exists because $\Theta$ has no infinite order contact
point. By the $C^{\infty}$ preparation theorem due to Malgrange \cite{Malgrange},
there exists $a_{1},\dots,a_{k-1}\in C^{\infty}(\mathbb{R}^{n}\times\mathbb{R}^{n})$,
and $c\in C^{\infty}(\mathbb{R\times}\mathbb{R}^{n}\times\mathbb{R}^{n})$
not vanishing on $((x_{1},\xi_{1}),t_{1})$, such that, on a neighborhood
of $((x_{1},\xi_{1}),t_{1})$, $h$ writes
\[
h(x,\xi,t)=c(x,\xi,t)(t^{k}+a_{k-1}(x,\xi)t^{k-1}+\dots+a_{0}(x,\xi)).
\]
Then, for $(x,\xi,t)$ in a neighborhood of $(x_{1},\xi_{1},t_{1})$,
\[
t=t(x,\xi)\iff t^{k}+a_{k-1}(x,\xi)t^{k-1}+\dots+a_{0}(x,\xi)=0.
\]
But, by \cite{Brink}, the roots of a monic polynomial are H\"{o}lder
with respect to the coefficients - of power one over the muliplicity
of the root. Because the $a_{i},\ 0\leq i\leq k-1$ are $C^{\infty}$
with respect to $(x,\xi)$, we conclude that $t$ is H\"{o}lder in a neighborhood
of $(x_{1},\xi_{1})$.
\end{proof}
\begin{rem}
\label{rem:geoK}The worst power in the H\"{o}lder inequality in a neighborhood
of a tangent ray is $\frac{1}{k}$, where $k$ is the order of contact
of the tangent ray. Note that in our framework of strictly geodesically
convex obstacles, $k=2$.
\end{rem}

\section*{Aknowlegments}
The author warmly thanks Nicolas Burq for his time and enlightening
discussions.

\rev{The idea for the proof of Proposition \ref{prop:reduc_comp} came from ChatGPT-6. The argument was then  reworked, written and checked carefully by the author.}

\providecommand{\bysame}{\leavevmode\hbox to3em{\hrulefill}\thinspace}
\providecommand{\MR}{\relax\ifhmode\unskip\space\fi MR }
\providecommand{\MRhref}[2]{%
  \href{http://www.ams.org/mathscinet-getitem?mr=#1}{#2}
}
\providecommand{\href}[2]{#2}


\begin{thebibliography}{HTW06}

\bibitem[ANV04]{boundedlp}
Ryuichi Ashino, Michihiro Nagase, and R{\'e}mi Vaillancourt,
  \emph{Pseudodifferential operators in {$L^p(\Bbb R^n)$} spaces}, Cubo
  \textbf{6} (2004), no.~3, 91--129. \MR{2124828}

\bibitem[BGH10]{MR2720226}
Nicolas Burq, Colin Guillarmou, and Andrew Hassell, \emph{Strichartz estimates
  without loss on manifolds with hyperbolic trapped geodesics}, Geom. Funct.
  Anal. \textbf{20} (2010), no.~3, 627--656. \MR{2720226 (2012f:58068)}

\bibitem[BGT04a]{MR2068304}
N.~Burq, P.~G{\'e}rard, and N.~Tzvetkov, \emph{On nonlinear {S}chr\"odinger
  equations in exterior domains}, Ann. Inst. H. Poincar\'e Anal. Non Lin\'eaire
  \textbf{21} (2004), no.~3, 295--318. \MR{2068304 (2005g:35264)}

\bibitem[BGT04b]{BGT_STR}
\bysame, \emph{Strichartz inequalities and the nonlinear Schr\"odinger equation on compact manifolds}, Amer. J. Math. {\bf 126} (2004), no.~3, 569--605; MR2058384

\bibitem[Bou11]{Bouclet}
Jean-Marc Bouclet, \emph{Strichartz estimates on asymptotically hyperbolic
  manifolds}, Anal. PDE \textbf{4} (2011), no.~1, 1--84. \MR{2783305}

\bibitem[Bri10]{Brink}
David Brink, \emph{H\"older continuity of roots of complex and p-adic
  polynomials}, Communications in Algebra (2010).

\bibitem[BT07]{BoucletTzvetkov}
Jean-Marc Bouclet and Nikolay Tzvetkov, \emph{Strichartz estimates for long
  range perturbations}, Amer. J. Math. \textbf{129} (2007), no.~6, 1565--1609.
  \MR{2369889}

\bibitem[Bur93]{plaques}
Nicolas Burq, \emph{Contr\^ole de l'\'equation des plaques en pr\'esence
  d'obstacles strictement convexes}, M\'em. Soc. Math. France (N.S.) (1993),
  no.~55, 126. \MR{1254820}


\bibitem[Bur98]{MR1618254}
\bysame, \emph{D\'ecroissance de l'\'energie locale de l'\'equation des ondes
  pour le probl\`eme ext\'erieur et absence de r\'esonance au voisinage du
  r\'eel}, Acta Math. \textbf{180} (1998), no.~1, 1--29. \MR{1618254
  (99j:35119)}

\bibitem[Bur02]{B02}
\bysame, \emph{Semi-classical estimates for the resolvent in nontrapping geometries}, Int. Math. Res. Not. {\bf 2002}, no.~5, 221--241; MR1876933

\bibitem[Bur04]{MR2066943}
\bysame, \emph{Smoothing effect for {S}chr\"odinger boundary value problems},
  Duke Math. J. \textbf{123} (2004), no.~2, 403--427. \MR{2066943
  (2006e:35026)}


\bibitem[CGT82]{Grom}
J. Cheeger, M. Gromov and M.~E. Taylor, \emph{Finite propagation speed, kernel estimates for functions of the Laplace operator, and the geometry of complete Riemannian manifolds}, J. Differential Geometry {\bf 17} (1982), no.~1, 15--53; MR0658471

\bibitem[CK01]{CK}
M. Christ and A.~A. Kiselev, \emph{Maximal functions associated to filtrations}, J. Funct. Anal. {\bf 179} (2001), no.~2, 409--425; MR1809116

\bibitem[DV12]{DV}
K.~R. Datchev and A. Vasy, \emph{Propagation through trapped sets and semiclassical resolvent estimates}, Ann. Inst. Fourier (Grenoble) {\bf 62} (2012), no.~6, 2347--2377 (2013); MR3060760

\bibitem[Da95]{Dav}
E.~B. Davies, The functional calculus, J. London Math. Soc. (2) {\bf 52} (1995), no.~1, 166--176; MR1345723

\bibitem[DZ19]{DZ}
S. Dyatlov and M. Zworski, {\it Mathematical theory of scattering resonances}, Graduate Studies in Mathematics, 200, Amer. Math. Soc., Providence, RI, 2019; MR3969938

\bibitem[GV85]{GV85}
J.~Ginibre and G.~Velo, \emph{Scattering theory in the energy space for a class
  of nonlinear {S}chr\"odinger equations}, J. Math. Pures Appl. (9) \textbf{64}
  (1985), no.~4, 363--401. \MR{839728}

\bibitem[HTW06]{HassellTaoWunsch}
Andrew Hassell, Terence Tao, and Jared Wunsch, \emph{Sharp {S}trichartz
  estimates on nontrapping asymptotically conic manifolds}, Amer. J. Math.
  \textbf{128} (2006), no.~4, 963--1024. \MR{2251591}

\bibitem[HoI]{Lars}
L.~V. H\"ormander, \emph{The analysis of linear partial differential operators. I}, reprint of the second (1990) edition, 
Classics in Mathematics, Springer, Berlin, 2003; \MR{1996773}

\bibitem[Ika82]{IkawaMult}
Mitsuru Ikawa, \emph{Decay of solutions of the wave equation in the exterior of
  two convex obstacles}, Osaka J. Math. \textbf{19} (1982), no.~3, 459--509.
  \MR{676233}

\bibitem[Ika88]{Ikawa2}
\bysame, \emph{Decay of solutions of the wave equation in the exterior of
  several convex bodies}, Ann. Inst. Fourier (Grenoble) \textbf{38} (1988),
  no.~2, 113--146. \MR{949013}

\bibitem[IP]{Square}
Oana Ivanovici and Fabrice Planchon, \emph{Square function and heat flow
  estimates on domains}.

\bibitem[Iva10]{MR2672795}
Oana Ivanovici, \emph{On the {S}chr\"odinger equation outside strictly convex
  obstacles}, Anal. PDE \textbf{3} (2010), no.~3, 261--293. \MR{2672795
  (2011j:58037)}

\bibitem[KT98]{KeelTao}
Markus Keel and Terence Tao, \emph{Endpoint {S}trichartz estimates}, Amer. J.
  Math. \textbf{120} (1998), no.~5, 955--980. \MR{1646048}

\bibitem[Leb92]{Lebeau}
G.~Lebeau, \emph{Contr\^ole de l'\'equation de {S}chr\"odinger}, J. Math. Pures
  Appl. (9) \textbf{71} (1992), no.~3, 267--291. \MR{1172452}

\bibitem[Lun44]{Luneberg}
R.~K. Luneberg, \emph{Mathematical {T}heory of {O}ptics}, Brown University,
  Advanced Instruction and Research in Mechanics, Providence, R. I., 1944.
  \MR{0011035}

\bibitem[Mal64]{Malgrange}
Bernard Malgrange, \emph{The preparation theorem for differentiable functions},
  Differential Analysis, Bombay Colloq.,Oxford Univ. Press, London (1964),
  203--208. \MR{182695}

\bibitem[MS78]{MR0492794}
R.~B. Melrose and J.~Sj{\"o}strand, \emph{Singularities of boundary value
  problems. {I}}, Comm. Pure Appl. Math. \textbf{31} (1978), no.~5, 593--617.
  \MR{0492794 (58 \#11859)}

\bibitem[MS82]{MR644020}
\bysame, \emph{Singularities of boundary value problems. {II}}, Comm. Pure
  Appl. Math. \textbf{35} (1982), no.~2, 129--168. \MR{644020 (83h:35120)}

\bibitem[Rem07]{Rem}
C. Remling, \emph{Finite propagation speed and kernel estimates for Schr\"odinger operators}, Proc. Amer. Math. Soc. {\bf 135} (2007), no.~10, 3329--3340; MR2322765

\bibitem[Sj97]{Sj:97}
J.~Sj\"{o}strand.
\newblock A trace formula and review of some estimates for resonances.
\newblock In {\em Microlocal analysis and spectral theory ({L}ucca, 1996)},
  volume 490 of {\em NATO Adv. Sci. Inst. Ser. C Math. Phys. Sci.}, pages
  377--437. Kluwer Acad. Publ., Dordrecht, 1997.

\bibitem[ST02]{StaffTata}
Gigliola Staffilani and Daniel Tataru, \emph{Strichartz estimates for a
  {S}chr\"odinger operator with nonsmooth coefficients}, Comm. Partial
  Differential Equations \textbf{27} (2002), no.~7-8, 1337--1372. \MR{1924470}

\bibitem[Str77]{Strichartz}
Robert~S. Strichartz, \emph{Restrictions of {F}ourier transforms to quadratic
  surfaces and decay of solutions of wave equations}, Duke Math. J. \textbf{44}
  (1977), no.~3, 705--714. \MR{0512086}

\bibitem[VZ00]{MR1764368}
Andr{\'a}s Vasy and Maciej Zworski, \emph{Semiclassical estimates in
  asymptotically {E}uclidean scattering}, Comm. Math. Phys. \textbf{212}
  (2000), no.~1, 205--217. \MR{1764368 (2002b:58047)}

\bibitem[V08]{Vasy}
Andr{\'a}s Vasy, \emph{Propagation of singularities for the wave equation on manifolds with corners}, Ann. of Math. (2) {\bf 168} (2008), no.~3, 749--812; MR2456883

\bibitem[Zwo12]{semibook}
Maciej Zworski, \emph{Semiclassical analysis}, Graduate Studies in Mathematics,
  vol. 138, American Mathematical Society, Providence, RI, 2012. \MR{2952218}

\end{thebibliography}
\end{document}